\documentclass[11pt,reqno]{amsart}
\usepackage[hmargin=2.5cm, bmargin=2.5cm, tmargin=3cm]{geometry}
\usepackage{enumerate}
\usepackage{amsmath,amsthm,amssymb,amsfonts,latexsym}
\usepackage{mathbbol}
\usepackage{pgfplots}
\usepackage[colorlinks, linkcolor=blue, filecolor=blue,
     citecolor = olive, urlcolor=purple]{hyperref}
\usepackage[gen]{eurosym}
\usepackage[german,english]{babel}
\usepackage[normalem]{ulem}
\usepackage{latexsym}
\usepackage{orcidlink}
\usepackage{tikz}					

\makeatletter
\newcommand{\opnorm}{\@ifstar\@opnorms\@opnorm}
\newcommand{\@opnorms}[1]{%
  \left|\mkern-1.1mu\left|\mkern-1.1mu\left|#1\right|\mkern-1.1mu\right|\mkern-1.1mu\right|
}
\newcommand{\@opnorm}[2][]{%
  \mathopen{#1|\mkern-1.1mu#1|\mkern-1.1mu#1|}
  #2
  \mathclose{#1|\mkern-1.1mu#1|\mkern-1.1mu#1|}
}
\makeatother

\DeclareMathOperator{\dist}{dist}
\DeclareMathOperator{\Diam}{Diam}

\allowdisplaybreaks

\newcommand{\Graph}{{\mathcal G}}

\DeclareMathOperator{\Mean}{Mean}

\newcommand{\Vmbc}{(\mV, m, b, c)}
\newcommand{\Mp}{\Mean_p}
\newcommand{\Mq}{\Mean_q}
\newcommand{\Mt}{\Mean_2}
\newcommand{\Vmb}{(\mV,m,b,c\equiv 0)}

\newcommand{\N}{\mathbb{N}}

\newcommand{\R}{\mathbb{R}}

\usepackage{aliascnt}

\theoremstyle{plain}
\newtheorem{theo}{Theorem}[section]

\newaliascnt{cor}{theo}
\newaliascnt{prop}{theo}
\newaliascnt{lemma}{theo}

\newtheorem{lemma}[lemma]{Lemma}
\newtheorem{prop}[prop]{Proposition}
\newtheorem{cor}[cor]{Corollary}

\aliascntresetthe{cor}
\aliascntresetthe{prop}
\aliascntresetthe{lemma}

\theoremstyle{defi} 

\newaliascnt{defi}{theo}
\newaliascnt{assum}{theo}
\newaliascnt{assums}{theo}
\newaliascnt{prob}{theo}
\newaliascnt{conv}{theo}

\newtheorem{defi}[defi]{Definition}
\newtheorem{assum}[assum]{Assumption}

\newtheorem{conv}[conv]{Convention}

\aliascntresetthe{defi}
\aliascntresetthe{assum}
\aliascntresetthe{assums}
\aliascntresetthe{prob}
\aliascntresetthe{conv}

	\theoremstyle{rem}

\newaliascnt{rems}{theo}
\newaliascnt{rem}{theo}
\newaliascnt{exa}{theo}
\newaliascnt{exs}{theo}

\newtheorem{rem}[rem]{Remark}
\newtheorem{exa}[exa]{Example}

\aliascntresetthe{rems}
\aliascntresetthe{rem}
\aliascntresetthe{exa}
\aliascntresetthe{exs}

\numberwithin{equation}{section}
\numberwithin{lemma}{section}

\newcommand{\mG}{\mathsf G}
\newcommand{\mP}{\mathsf P}
\newcommand{\mS}{\mathsf S}
\newcommand{\mT}{\mathsf T}
\newcommand{\mV}{\mathsf V}

\newcommand{\mw}{\mathsf w}
\newcommand{\mv}{\mathsf v}

\newcommand{\me}{\mathsf e}

\newcommand{\tauneubc}{\tau_{p}^\mG}
\newcommand{\tauneubcmulambda}{\tau_{p}^{\mG_{\mu,\lambda}}}
\newcommand{\tauneubct}{\tau_{p}^{\widetilde{\mG}}}
\newcommand{\hneubc}{h_{p}^{\widetilde{\mG}}}

\newcommand{\taudirbc}{\tau_{p}^{\mG;\mV_0}}
\newcommand{\taudirbcmulambda}{\tau_{p}^{{\mG_{\mu,\lambda}};\mV_0}}

\newcommand{\taudirb}{\tau_{p}^{\mG;\mV_0}}
\newcommand{\taudirbpath}{\tau_{p}^{\mP;\mV_0}}
\newcommand{\udirbpath}{u_{p}^{\mP;\mV_0}}
\newcommand{\taudirbpathpr}{\tau_{p}^{\mP';\mV_0}}
\newcommand{\taudirbpathlin}{\tau_{2}^{\mP;\mV_0}}
\newcommand{\taudirbstar}{\tau_{p}^{\mathsf{S};\mV_0}}
\newcommand{\taudirbstarlin}{\tau_{2}^{\mathsf{S};\mV_0}}

\newcommand{\Lpformbc}{L_{p}^\mG}
\newcommand{\Qpformbc}{Q_{p}^\mG}

\newcommand{\Qneu}{\mathcal{Q}_p}
\newcommand{\Ldirbc}{\mathcal{L}^{\mG;\mV_0}_{p}}

\newcommand{\Ldirbpath}{\mathcal{L}^{\mP;\mV_0}_{p}}

\newcommand{\Ldirbstar}{\mathcal{L}^{\mathsf{S};\mV_0}_{p}}

\newcommand{\Lneubc}{\mathcal{L}_{p}^\mG}

\newcommand{\Qdirbc}{\mathcal{Q}^{\mG;\mV_0}_{p}}

\newcommand{\Qneubcpr}{\mathcal{Q}_{p}^{\mG'}}
\newcommand{\Qneubc}{\mathcal{Q}_{p}^\mG}
\newcommand{\Qneubct}{\mathcal{Q}_{p}^{\widetilde{\mG}}}

\newcommand{\Ddirbc}{\overset{\circ}{\mathcal{D}}{}^p_{b,c;\mV_0}(\mV)}
\newcommand{\Dneubc}{\mathcal{D}^p_{b,c}(\mV)}
\newcommand{\wdirbc}{\overset{\circ}{w}^{1,p,1}_{b,c;\mV_0}}
\newcommand{\wqpdirbc}{\overset{\circ}{w}^{1,p,q}_{b,c;\mV_0}}
\newcommand{\wppdirbc}{\overset{\circ}{w}^{1,p,p}_{b,c;\mV_0}}

\newcommand{\wneubc}{w^{1,p,1}_{b,c}}\newcommand{\wqpneubc}{w^{1,p,q}_{b,c}}
\newcommand{\wppneubc}{w^{1,p,p}_{b,c}}

\newcommand{\wdirbcpr}{\overset{\circ}{w}^{1,p,1}_{b',c';\mV_0'}}
\newcommand{\wdirb}{\overset{\circ}{w}^{1,p,1}_{b;\mV_0}}
\newcommand{\wneubcpr}{w^{1,p,1}_{b',c'}}

\newcommand{\wqpneu}{w^{1,q,p}}

\newcommand{\wqpdir}{\overset{\circ}{w}^{1,p,q}_{\mV_0}}
\newcommand{\Fneubc}{ \mathfrak{F}_{p}^\mG}
\newcommand{\Polneubc}{{\mathcal{P}}_{p}^\mG}
\newcommand{\Polneubchat}{{\mathcal{P}}_{p}^{\hat \mG}}
\newcommand{\Polneubct}{{\mathcal{P}}_{p}^{ \widetilde{\mG}}}
\newcommand{\Poldirbc}{{\mathcal{P}}_{p}^{\mG;\mV_0}}
\newcommand{\Polneubcpr}{{\mathcal{P}}_{p}^{\mG'}}

\title[On the $p$-torsional rigidity of combinatorial graphs]{On the $p$-torsional rigidity of combinatorial \\graphs}
\dedicatory{Dedicated to the memory of Luca Trevisan (1971--2024)}


\author[P.~Bifulco]{Patrizio Bifulco\orcidlink{0009-0004-0628-374X}}
\author[D.~Mugnolo]{Delio Mugnolo\orcidlink{0000-0001-9405-0874}}

\address{Lehrgebiet Analysis, Fakult\"at Mathematik und Informatik, Fern\-Universit\"at in Hagen, D-58084 Hagen, Germany}
\email{patrizio.bifulco@fernuni-hagen.de}

\email{delio.mugnolo@fernuni-hagen.de}


\thanks{
We are grateful to Lorenzo Liverani (Milano) for fruitful discussions concerning some steps in the proof of \autoref{lem:fp-convex}. We also wish to thank Lorenzo Brasco (Ferrara) for pointing us at~\cite{BraRuf17} and for insightful observations. We are also indebted to the referee for their very careful reading and their very valuable comments.\\
Both authors were partially supported by
  the Deutsche Forschungsgemeinschaft (Grant 397230547).}

\begin{document}

\begin{abstract}
We study the $p$-\emph{torsion function} and the corresponding $p$-\emph{torsional rigidity} associated with $p$-Laplacians and, more generally, $p$-Schrödinger operators, for $1<p<\infty$, on possibly infinite combinatorial graphs. We present sufficient criteria for the existence of a summable $p$-torsion function and we derive several upper and lower bounds for the $p$-torsional rigidity. Our methods are mostly based on novel surgery principles.
As an application, we also find some new estimates on the bottom of the spectrum of the $p$-Laplacian with Dirichlet conditions, thus complementing some results recently obtained in~\cite{MazTol23} in a more general setting. Finally, we prove a Kohler-Jobin inequality for combinatorial graphs (for $p=2$): to the best of our knowledge, graphs thus become the third ambient where a Kohler-Jobin inequality is known to hold.
\end{abstract}

\maketitle

\tableofcontents

\section{Introduction}

The $p$-Laplacian is among the most thoroughly studied nonlinear operators. As nicely described in the extensive historical survey~\cite{BenGirKot18}, this class of operators has been used in the analysis of fluid dynamics models (in the range $p\in [\frac{3}{2},2]$) since the 19th century, especially in the filtration problem in porous materials. More recently, further physical systems have been studied by means of the $p$-Laplacians (with $p$ in different ranges), including continuum mechanics captured by the Ramberg--Osgood relationship \cite{ThoSan06,Kne06} and superconductivity in the Bean model \cite{Yin01,ChoKimLaf18}.

In the linear world, the Laplace transform offers a natural connection between the parabolic and elliptic equation associated with the Laplacian: accordingly, the so-called \emph{torsion function}, i.e., the solution $u$ of
\[
-\Delta u=\mathbf{1},
\]
has a far-reaching probabilistic interpretation for evolution equations \cite{BanBerCar02}. If the Laplacian is replaced by any of its nonlinear siblings, the $p$-Laplacians ($p\ne 2$), this connection is missing; however, a large part of the torsional analysis can be still performed. 
The aim of this paper is to discuss the torsional properties of a class of homogeneous, $p$-Laplacian-like nonlinear operators on combinatorial graphs, with a special focus on the interplay with their spectral theory. 

In this article we focus on infinite combinatorial graphs. 
Nonlinear difference operators -- and, especially, $p$-Laplacians -- on graphs play a pivotal role in Minty's theory of monotone networks (electrical networks of nonlinear resistors)~\cite{Min60,NakYam76,Cal96}, as well as in modern spectral clustering techniques~\cite{BuhHei09,SzlBre10,HeiLenMug15}. In addition, it is also natural to regard graphs as a convenient mathematical environment to describe the discretization of domains: at least for $p=2$, various notions of convergence of the lattice Laplacian towards the continuum Laplacian have been proved rigorously, see~\cite{NakTad21} and references therein.

The class of nonlinear Schrödinger operators of our interest here is rather natural: it appears as the nonlinear relaxation of classical Schrödinger-type operators on graphs studied over the last decades \cite{Soa94,KelLenWoj21}. As $p$-Schrödinger operators are designed to be homogeneous (of degree $p-1 \in (0,\infty)$), it is still possible to study them by variational or potential-theoretic methods, as already done since~\cite{Yam75}, see e.g.~\cite{Amg03,BuhHei09,Mug13,KelMug16} for more modern literature. In the non-local context of interest here, $p$-Schrödinger operators on graphs have been introduced 
in~\cite{Fis23}.

In the case of the (linear) free Laplacian with Dirichlet conditions on bounded planar domains $\Omega$, the torsion function was introduced by Pólya in \cite{Pol48}, who especially studied its $L^1$ norm, the ``torsional rigidity'' of $\Omega$; 
the torsion function recently gained popularity as it has been used  since~\cite{FilMay12,Ber12} as a ``landscape function'' to deliver pointwise upper bounds on eigenfunctions, see \cite{Mug23} for extensions of this theory to the nonlinear, homogeneous setting and for historical remarks. While the theory of torsion on domains is fairly well understood, much less attention has been devoted to the case of graphs: among recent articles, let us mention \cite{FilMayTao21}, where a (generalized) torsion function has been used to derive Agmon-type estimates on the eigenfunctions of linear Schrödinger operators on finite graphs; \cite{MugPlu23,Ozc24}, where it is shown that the torsion function of metric graphs can be computed using the
torsion function on the underlying combinatorial graph; and \cite{AdrSet23}, where a probabilistic interpretation of the torsion function for the (linear) graph Laplacian with Dirichlet conditions is used as a tool in a potential theoretic analysis of infinite combinatorial graphs.

It was demonstrated in \cite{Bra14} that the  variational approach first proposed by Pólya \cite{Pol48} can be successfully and almost seamlessly adapted to general $p$-Laplacians. Brasco's ideas partially extend to a non-smooth setting, too, and indeed Mazón and Toledo have very recently discussed the torsional properties of the Laplacian in the context of so-called \textit{random walk spaces} \cite{MazTol23}:
we only recall that possibly infinite combinatorial graphs represent a subclass of random walk spaces and refer to~\cite{MazSolTol23} for an extensive introduction to this theory.
In particular, different but comparable estimates of the $p$-torsional rigidity in terms of the $p$-Cheeger constant, and of the infimum of a Rayleigh-type quotient with mixed exponents, are derived in~\cite{MazTol23} and~\cite{BraRuf17}, respectively, in the case of random walk spaces and possibly unbounded Euclidean domains, respectively.

In the general non-smooth setting, the existence of a (non-trivial) $p$-torsion function is by no means obvious. For bounded open domains of $\R^d$, a solution of
\[
-\Delta_p u=\mathbf{1},
\]
with Dirichlet boundary conditions can be found by elementary variational methods \cite{Bra14}; in the unbounded case, the analysis of (non-trivial) $p$-torsion functions has been performed ever since~\cite{BanBerCar02}:
in~\cite{BraRuf17}, in particular, the authors first introduce a weak notion of $p$-torsion function by exhaustion, and then prove the equivalence between its integrability properties, the validity of enhanced Poincaré-type inequalities, and the compactness of certain Sobolev embedding. 
(In the case of combinatorial graphs, Poincaré inequalities are known to hold under appropriate geometric conditions that typically concern the curvature of the graph, cf.~\cite{Woe00,CouKos04,FuePes13,
MazSolTol23}.)
An exhaustion method that generally only leads to introducing a pointwise defined torsion function (which is a priori not required to enjoy any summability property) has been worked out on infinite, combinatorially locally finite graphs in~\cite{AdrSet23}. 
In particular, \cite[Theorem~5.4]{AdrSet23} implies that, for $p=2$, failure of the $\ell^1$-Liouville property is a necessary condition for the existence of a torsion function of class $\ell^1$ on graphs with Dirichlet conditions; then,~\cite[Theorem~3.5]{AdrSet23} rules out the existence of an $\ell^1$-torsion function on, e.g., stochastically complete model graphs. Because we will be interested in the \textit{torsional rigidity} of a graph $\mG$, which is defined in terms of the $\ell^1$ norm of the torsion function, we here prefer to focus on a weak formulation of the relevant elliptic equation that automatically implies that the $p$-torsion function is summable, provided $\mG$ has finite measure. 

The aim of this paper is twofold. On one hand we 
investigate existence and uniqueness of a $p$-torsion function for $p$-Schrödinger operators with or without Dirichlet conditions: we are going to prove that, under the assumption that suitable Poincaré inequalities are satisfied,
a unique $p$-torsion function -- which in the simplest instances \\
n by $\tau_p:= (-\Delta_p)^{-1}\mathbf{1}$ -- exists and is summable; in particular, the corresponding torsional rigidity $\| \tau_p \|_{\ell^1}^{p-1}$ is finite. On the other hand, we present a few upper and lower bounds on the $p$-torsional rigidity: these investigations are pursued by adapting to the discrete setup suitable surgery principles developed since \cite{BerKenKur17}. Many of these bounds are sharp and rigid and can thus also be understood as torsional shape optimization results in the spirit of \cite{ButRufVel14,MugPlu23}.
As a by-product, we can deduce information about the spectral-geometric properties of $p$-Schrödinger operators, and in particular the interplay between the bottom of their spectrum and their $p$-torsional rigidity.

In \autoref{sec:general} we formally introduce the relevant class of homogeneous $p$-Schrödinger operators, as well as the corresponding spaces of functions of finite energy: in the case of  vanishing potential, these objects have been studied since~\cite{NakYam76,Yam77}. The existence and uniqueness theory for the $p$-torsion function is developed in \autoref{sec:torsionfunctionandrigidity}. In comparison with known investigations on the torsion of Euclidean domains, we relax the necessity of imposing Dirichlet conditions at some vertices: in particular, \autoref{cor:wellp} shows that for several relevant classes of graphs  (including all finite graphs and certain rapidly branching symmetric trees) adding a non-trivial potential is sufficient to enforce the existence of a $p$-torsion function. Many of our results are formulated under the (rather strong) assumption that the graph has finite measure, which is comparable to that in \cite{MazTol23}; but, crucially, our main existence results, \autoref{thm:well-p-l1}, is independent of this condition:
by \autoref{cor:compact-admiss}, compactness of a Sobolev-type embedding is decisive, like in the continuous case discussed in~\cite{BraRuf17}.
In \autoref{sec:torsionalrig} we introduce the $p$-torsional rigidity, discuss its summability properties and discuss its monotonicity properties with respect to subgraph inclusion.
Inspired by~\cite{MugPlu23,MazTol23}, we derive in \autoref{sec:bounds}  a number of upper (Section \ref{sec:upper-bounds}) and lower bounds (Section \ref{sec:lower-bounds}) for the $p$-torsional rigidity in terms of different metric and measure theoretical quantities. Our proofs are based on surgery principles that boil down to comparisons between the given graph and a simpler reference graph, typically a path graph. Similar ideas are used in Section \ref{sec:appliaction-lower-estimates} to estimate the bottom of the spectrum of the $p$-Laplacian with Dirichlet conditions: unlike in the linear case (see, e.g., \cite{Fie73}, \cite[Chapters 1 and 3]{Chu97}, \cite[Section~8]{KenKurmal16}) such bounds of geometric flavour are surprisingly rare in the literature. In particular \autoref{prop:estimate-inradius} yields a geometric condition on a certain dual graph that enforces the existence of a $p$-torsion function. Finally, in Section~\ref{sec:kohler} we turn to discussing a graph counterpart of the celebrated \textit{Kohler-Jobin inequality}, a sharp lower bound on a product of $p$-torsional rigidity and bottom of the $p$-spectrum that, for $p=2$ that was first conjectured by Pólya in the case of planar domains. Very few instances of such bounds are known so far: we are only aware of Kohler-Jobin inequalities on Euclidean domains~\cite{Koh78} and on metric graphs~\cite{MugPlu23}, both under Dirichlet conditions. We here prove a Kohler-Jobin inequality for both normalized and normalized graph Laplacians in \autoref{theo:lower-bound-classical-kohler-jobin-product} and \autoref{cor:kohlerjobin-unnorm}, respectively, also in this case under Dirichlet conditions.
  
\section{General setting: $p$-Schrödinger operators on graphs}\label{sec:general}

Let $\mV$ be a finite or (countably) infinite set and let $m:\mV \to (0,\infty)$ be a point measure on $\mV$: then $(\mV,m)$ becomes a discrete measure space via $m(A) := \sum_{\mv \in A} m(\mv)$ for every subset $A \subset \mV$. Following the terminology and notation in~\cite{KelLenWoj21}, a \emph{locally finite graph}, or simply a \textit{graph} $\mG$ over $(\mV,m)$ is determined by
\begin{itemize}
\item  a function $b: \mV  \times \mV \to [0,\infty)$ such that
\begin{enumerate}[(i)]
\item\label{item:symmetry} $b(\mv,\mw)=b(\mw,\mv)$ for all $\mv,\mw \in \mV$,
\item\label{item:no-loops} $b(\mv,\mv) = 0$ for all $\mv \in \mV$,
\item\label{item:locally-finite} $\sum\limits_{\mw\in \mV} b(\mv,\mw)<\infty$ for every $\mv\in\mV$; along with
\end{enumerate}
\item a function $c: \mV \to [0,\infty)$.
\end{itemize}
In the following, we denote such a graph by $\mG = \Vmbc$.
We call it a \textit{graph with standard edge weights}, and write $b=b_{\mathrm{st}}$, if $b(\mv,\mw)=1$ whenever $b(\mv,\mw)>0$.
A special role in our investigation will be played by graphs with $c\equiv 0$: in this case, we denote our graph by $\Vmb$ and remark that these objects agree with \textit{weighted graphs} from classical graph theory.
The above property \eqref{item:locally-finite} means that the graph is \textit{locally finite}. The graph is said to be 
\textit{combinatorially locally finite} if the neighbour set $N_\mv:=\{\mw\in \mV:b(\mv,\mw)>0\}$ is finite for each $\mv\in \mV$. (Observe that each combinatorially locally finite graph is locally finite, and the converse holds if, additionally, $\inf_{\mv\in \mV}\inf_{\mw\in N_\mv}b(\mv,\mw)>0$.)

We call each $\mv\in \mV$ a \textit{vertex}, and each couple of (distinct) vertices $\mv,\mw\in \mV$ with $b(\mv,\mw)>0$ an \textit{edge} (clearly, the set of edges is at most countable). Intuitively, this means that two vertices $\mv,\mw$ with $b(\mv,\mw)=0$ are not connected by any edge:  if,  however, $b(\mv,\mw)>0$, then we say that $\mv,\mw$ are \textit{adjacent}. The \emph{degree} of the graph $\mG$ is the function $\deg_\mG: \mV \rightarrow [0,\infty)$ defined by
\begin{align*}\label{eq:definition-deg}
\deg_\mG(\mv) :=  \sum_{\mw \in \mV} b(\mv,\mw) + c(\mv), \qquad \mv \in \mV;
\end{align*}
in particular, $\deg_\mG(\mv)$ is the number of edges in $\mG$ connected to $\mv$ if $b = b_{\mathrm{st}}$ and $c \equiv 0$. If there is no danger of confusion, we will often simply denote $\deg_\mG(\mv)$ by $\deg(\mv)$; however, it will frequently happen that we ``rewire'' a given graph and, hence, end up considering a new graph $\mG'$ on the same vertex set $\mV$, so that $\deg_\mG:\mV\to [0,\infty)$ and $\deg_{\mG'}:\mV\to [0,\infty)$ will indeed be different. 

\begin{rem}\label{rem:multiple}
Technically speaking, the above formalism does not apply to \textit{multigraphs}, i.e., it does not allow for the existence of two or more distinct edges between any two given vertices $\mv,\mw$: i.e., only weighted \emph{simple} graphs can be described as above. In practice, if we need to consider a weighted multigraph -- and this will, indeed, repeatedly happen when dealing with surgical operations in \autoref{sec:bounds} --, it will for our purposes be sufficient to ``reduce'' it to a graph in the above sense. To perform this reduction, we replace all edges $\me^{1}_{\mv,\mw},\ldots,\me^{m}_{\mv,\mw}$ between any two vertices $\mv,\mw$ by \textit{one} edge whose weight $b(\mv,\mw)$ is the sum of the weights of all parallel edges  $\me^{i}_{\mv,\mw}$.

Luckily, it can be easily checked that the adjacency matrix of a weighted multigraph agrees with the matrix $(b(\mv,\mw))_{\mv,\mw\in \mV}$, i.e., the above reduction procedure is consistent with the usual theory of multigraph Laplacians (see, e.g., \cite[Chapter 2]{Mug14}).
\end{rem}

A \textit{subgraph} of $\mG=\Vmbc$ is a new graph $\mG'=(\mV',m',b',c')$ such that $\mV'\subset \mV$, $m':=m\vert_{\mV'}$, $c'(\mv):=[0,c(\mv)]$, and $b'(\mv,\mw)\in [0,b(\mv,\mw)]$ for any $\mv,\mw\in \mV'$: roughly speaking, a subgraph of a graph $\mG$ may ``lose'' both vertices and edges.

In most cases we will however only consider the \textit{subgraph induced} by some $\mV'\subset \mV$: by definition, this is the graph whose vertex set is $\mV'$ and such that 
$b' = b\vert_{\mV' \times \mV'}$ and $c'=c\vert_{\mV'}$; in other words, 
any two vertices $\mv,\mw$ in $\mV'$ are adjacent if and only if so are they in $\mV$, and in this case they have the same weight as in the original graph.

A \textit{path connecting two pairwise distinct vertices $\mv,\mw$} in $\Vmbc$ is, by definition, a sequence $(\mv=\mv_0,\ldots,\mv_n=\mw)$ of vertices such that $\mv_0=\mv$, $\mv_n=\mw$, and $b(\mv_i,\mv_{i+1})>0$ for each $i=0,\ldots,n-1$; moreover, we call it a \emph{cycle} whenever $\mv = \mw$.
A graph is \textit{connected} if there is at least one path between any two vertices. 

A \textit{tree} is a graph that does not contain any cycle as a (not necessarily induced) subgraph. A \emph{path graph} is a tree such that the vertex set can be written as $\mV = \{ \mv_j \: : \: j \in J \}$, where 
\[
\hbox{either }J=\{0,1,\ldots,n-1\}, \qquad\hbox{ or }J=\N_0,
\]
and $b(\mv,\mw) > 0$ for any $\mv,\mw\in \mV$ if and only if $\mv=\mv_i$ and $\mw=\mv_{i+1}$ for some $i \in J$. A \emph{star graph with center $\mv_c \in \mV$} is a tree such that $b(\mv,\mw) > 0$  for any $\mv,\mw\in \mV$ if and only if $\mv=\mv_c$ and $\mw\ne \mv_c$.

The \textit{length} of a path $(\mv_0,\ldots,\mv_n)$ is $\sum_{i=0}^{n-1} b(\mv_i,\mv_{i+1})$. The \textit{distance}  $\dist_{b}$ between any two vertices $\mv,\mw \in \mV$ is defined as the infimal length among all paths connecting $\mv,\mw$. Observe that $(\mV,m,\dist_b)$ is a metric measure space; also, $(\mV,m,\nu)$ is a random walk space, cf.\ \cite[Example~1.41]{MazSolTol23}, where $\nu(A):=\sum_{\mv\in A}\left(c(\mv)+\sum_{\mw\in\mV}b(\mv,\mw)\right)$ for any $A\subseteq \mV$.

 Given $q>0$, we also define the \emph{$q$-inradius} of any graph $\mG=\Vmbc$ with respect to $\mV_0 \subset \mV$ by 
\begin{equation}\label{eq:inradius-defi}
\mathrm{Inr}_q(\mG;\mV_0) := \sup_{\mv \in \mV} \mathrm{dist}_{q,b}(\mv;\mV_0)^{q-1} = \sup_{\mv \in \mV} \Big(\inf_{\mw \in \mV_0} \dist_{q,b}(\mv,\mw)  \Big)^{q-1},
\end{equation}
for the \textit{$q$-distance} (with respect to the edge weight $b$) defined by
\begin{equation}\label{eq:qbdist-def}
\dist_{q,b}(\mv,\mw) := \inf \sum_{j=0}^{n-1} b(\mv_j,\mv_{j+1})^{\frac{1}{q-1}},
\end{equation}
where the infimum is taken over all paths $(\mv=\mv_0,\ldots,\mv_n=\mw)$ connecting $\mv$ and $\mw$: observe that the $q$-distance agrees with the usual one for $q=2$ and that, apparently, these metric quantities do not depend on either the measure $m$ or the potential $c$.

We can hence obtain more refined information on the fine geometric structure of a graph $\mG=\Vmbc$ by means of the \emph{$q$-mean distance}
\begin{equation}\label{eq:q-meanstdist-defi}
\Mq({\mG};\mV_0) := \frac{1}{m(\mV\setminus\mV_0)} \sum_{\mv \in \mV\setminus \mV_0} \mathrm{dist}_{q,b}(\mv;\mV_0)^{q-1} m(\mv).
\end{equation}
Clearly one has that
\begin{align*}
\mathrm{Inr}_q({ \mG};\mV_0) \geq \Mq({\mG};\mV_0) \qquad \text{for every $m: \mV \rightarrow (0,\infty)$:}
\end{align*}
we stress that neither of these quantities depend on $c$.

Following \cite{RigSalVig97}, we denote by $\nabla f$ the discrete gradients defined by
\[
\nabla_{\mv,\mw}f := f(\mv)-f(\mw),\qquad f\in \R^\mV,\  \mv,\mw\in \mV.
\]

We let throughout the paper 
\[
p\in (1,\infty),
\]
and denote $p' := \frac{p}{p-1}$.
This article is devoted to the analysis of a discrete Poisson-type equation for $p$-Schrödinger operators, $1<p<\infty$, given by
\begin{equation}\label{eq:p-schrödinger}
 \Lpformbc f(\mv):= \frac{1}{m(\mv)}\sum_{\mw\in\mV} b(\mv,\mw)|\nabla_{\mv,\mw}f|^{p-2}\nabla_{\mv,\mw}f+\frac{c(\mv)}{m(\mv)}|f(\mv)|^{p-2}f(\mv),\quad f\in c_{00}(\mV),\ \mv\in\mV,
\end{equation}
where $c_{00}(\mV)$ denotes the vectors space of functions from $\mV$ to $\R$ with finite support. In this paper we are going to consider equations of the form
\[
L^\mG_p f=g,
\]
for some $g:\mV\to \R$. 

With the ultimate aim of applying variational methods as done in \cite{Bra14, MugPlu23, MazTol23},  we are led to considering the smooth energy functional
\begin{equation}\label{eq:quad-form}
 \Qpformbc: u \mapsto \frac{1}{2p} \sum_{\mv,\mw \in \mV} b(\mv,\mw)\vert \nabla_{\mv,\mw}u \vert^{p} + \frac{1}{p}\sum_{\mv \in \mV} c(\mv)\vert u(\mv) \vert^{p}
\end{equation}
defined on a domain that contains $c_{00}(\mV)$ and which will be specified later.

Observe that \(u\mapsto \Qpformbc (u)^\frac{1}{p}\) already defines a semi-norm on
\[
\Dneubc := \bigg\{ u \in \R^\mV :\Qpformbc(u)< \infty \bigg\}\quad\hbox{and}\quad
\Ddirbc :=\{u\in \Dneubc  :u(\mv)=0\hbox{ for all }\mv\in \mV_0\};
\]
and in fact a norm 
\begin{itemize}
\item on $\Dneubc$ if $\sup_{\mv\in\mV}c(\mv) > 0$, as well as
\item on $\Ddirbc$ for any $\emptyset\ne \mV_0\subset \mV$.
\end{itemize}

We also introduce the discrete Sobolev-type spaces
\[
\wqpneubc (\mV, m) := \Dneubc \cap  \ell^q(\mV,m)
\quad
\hbox{and}
\quad  \wqpdirbc(\mV, m):=\Ddirbc \cap  \ell^q(\mV,m),
 \]
for $q \in [1,\infty)$, which we equip with the canonical norm
\[
u \mapsto \Vert u \Vert_{\ell^q(\mV,m)} + \Qpformbc (u)^\frac{1}{p}.
\]

In the following, we are going to consider a class of conditions usually referred to as \textit{$\ell^{q}-\ell^p$-Poincaré inequalities}, i.e., the validity of the estimate
\begin{equation}\label{eq:poinc-neu}
\|u\|_{\ell^{q}(\mV,m)}\le M \Qpformbc(u)^\frac{1}{p}\qquad \hbox{for all }u\in \wqpneubc(\mV,m)
\end{equation}
or its Dirichlet counterpart
\begin{equation}\label{eq:poinc-neu-dir}
\|u\|_{\ell^{q}(\mV\setminus \mV_0,m)}\le M \Qpformbc (u)^\frac{1}{p}\qquad \hbox{for all }u\in \wqpdirbc(\mV,m)
\end{equation}
for some $M>0$; we are mostly going focus on the case where $q=1$ (cf.\ \autoref{sec:torsionalrig} and \autoref{sec:bounds}).
For the sake of later reference let us formulate the following.

\begin{assum}\label{ass:poincare-inequality}
Let $q \in [1,\infty)$. Either $\sup_{\mv\in\mV}c(\mv)>0$ and the $\ell^q-\ell^p$-Poincaré inequality~\eqref{eq:poinc-neu}
is satisfied; or $\mV_0\ne\emptyset$ and the $\ell^q-\ell^p$-Poincaré inequality~\eqref{eq:poinc-neu-dir} is satisfied.
\end{assum}

\begin{prop}\label{prop:intro-plapl}
Let $\mG=\Vmbc$ be a graph, and let $q \in [1,\infty)$. Under the \autoref{ass:poincare-inequality}, let $\sup_{\mv\in\mV} c(\mv)>0$ (resp., $\mV_0\ne \emptyset$).
Then
\(
u \mapsto \Qpformbc (u)^\frac{1}{p}
\)
defines an equivalent norm on $\wqpneubc(\mV,m)$ (resp., on $\wqpdirbc(\mV,m)$), with respect to which it is a 
separable Banach space; it is densely and continuously embedded in $\ell^q(\mV,m)$; it is reflexive if $q \neq 1$. Also, $\wqpdirbc(\mV,m)$ is a closed subspace and a lattice ideal of $\wqpneubc(\mV,m)$ for every $\mV_0\subset \mV$.
\end{prop}

\begin{proof}
Equivalence of the norms is an immediate consequence of the $\ell^q-\ell^p$-Poincaré inequality.
Separability (as well as reflexivity if $q \neq 1$) of $\wqpneubc(\mV,m)$ and $\wqpdirbc(\mV,m)$ can be proved as in \cite[Lemma~3.6]{Mug14}, respectively. Also, they are continuously and densely embedded in $\ell^q(\mV,m)$, see \cite[Proposition~3.8]{Mug14}. The ideal property can be checked directly.
\end{proof}

Under the \autoref{ass:poincare-inequality} we are, therefore, mostly going to endow the discrete Sobolev spaces with the norm
\begin{equation}\label{eq:fromnowequiv}
\|u\|_{\wqpneubc}:=  \Qpformbc (u)^\frac{1}{p}
\end{equation}
throughout.
Also, (whenever $q \in [1,\infty)$ is fixed) in the following we are going to denote by $\Qneubc$ and $\Qdirbc$ the restriction of $\Qpformbc$ to $\wqpneubc(\mV,m)$ and $\wqpdirbc(\mV,m)$, respectively.

\begin{rem}
At the risk of being tautological, let us stress that the Sobolev spaces do depend on the coefficients $b,c$, as so does $Q^\mG_p$.
Observe that $\wqpneubc(\mV,m)=w^{1, q,p}_{\mathbf 1,\mathbf 1}(\mV,m)=:\wqpneu(\mV,m)$ and $ \wqpdirbc(\mV,m)=\overset{\circ}{w}^{1, q,p}_{\mathbf 1,\mathbf 1;\mV_0}(\mV,m)=: \wqpdir(\mV,m)$ whenever there exist $\beta,B>0$ and $\gamma,\Gamma>0$ such that
\[
\beta<b(\mv,\mw)\le B\quad\hbox{and}\quad \gamma<c(\mv)\le \Gamma\qquad\hbox{for all }\mv,\mw\in \mV .
\]
Also, $\wqpneubc(\mV,m)=\ell^q(\mV,m)=\R^\mV$ and $\wqpdirbc(\mV,m)=\ell^q(\mV\setminus\mV_0,m)=\R^{\mV\setminus\mV_0}$ if $\mV$ is finite.
\end{rem}

Let us present a setting where such $\ell^q-\ell^p$-inequalities can be easily proved.

To begin with, observe that if $\mV$ is finite and $c(\mv_0)>0$ for some $\mv_0$, we can consider the norm 
\[
u\mapsto \opnorm{u} :=\left(\frac{1}{2p}\sum\limits_{\mv,\mw\in \mV}b(\mv,\mw)|\nabla_{\mv,\mw}u|^p+\frac{1}{p}c(\mv_0)|u(\mv_0)|^p\right)^\frac{1}{p}
\]
 on $\Dneubc$. Then clearly $\opnorm{u} \lesssim \Qneubc(u)^\frac{1}{p}$, but on the other hand a Poincaré-type inequality involving the norms  $\|\cdot\|_{\ell^1}$ (and thus in particular $\|\cdot\|_{\ell^q}$) and $\opnorm{\: \cdot \:}$ holds by \cite[Lemma~2.1]{Yam77}. In fact, more can be said.
\begin{prop}\label{cor:compact-admiss}
Let $\mG=\Vmbc$ be a graph. Then the following assertions hold.
\begin{enumerate}[(i)]
\item\label{item:poincare-robin} Let $\sup_{\mv\in\mV}c(\mv)>0$. If $\wqpneubc(\mV,m)$
is compactly embedded in $\ell^q(\mV,m)$, then an $\ell^q-\ell^p$-Poincaré inequality \eqref{eq:poinc-neu} holds.
\item\label{item:poincare-dirichlet} Let $\emptyset \ne\mV_0\subset \mV$. If $\wqpdirbc(\mV,m)$ is compactly embedded in $\ell^q(\mV,m)$, then an $\ell^q-\ell^p$-Poincaré inequality \eqref{eq:poinc-neu-dir} holds.
\end{enumerate}
\end{prop}

\begin{proof}
By compactness, $\wqpneubc(\mV,m) \ni u\mapsto \Qneubc(u)$ attains its minimum on the unit ball of $\ell^{q}(\mV,m)$, hence there exists $u_0 \in \ell^{q}(\mV,m)$ which, without loss of generality, we assume to be normalized, i.e., 
\begin{align}\label{eq:minimizer-u0}
\Qneubc(u_0) = \min_{u \in \wqpneubc(\mV,m)} \frac{\Qneubc(u)}{\Vert u \Vert_{\ell^q(\mV,m)}^p}. 
\end{align}
We show that $\Qneubc(u_0)$ has to be strictly positive: indeed, as 
\begin{align}\label{eq:bottom-of-p-spec-ground-state}
\Qneubc(u_0) = \frac{1}{2p} \sum_{\mv, \mw \in \mV} b(\mv,\mw) \vert \nabla_{\mv,\mw} u_0 \vert^p + \frac{1}{p}\sum_{\mv \in \mV} c(\mv) \vert u_0(\mv) \vert^p,
\end{align}
it follows that $u_0$ cannot be constant if $c(\mv) \vert u_0(\mv) \vert^p = 0$ for all $\mv \in \mV$ as $\sup_{\mv \in \mV} c(\mv) > 0$ (note that $u_0$ is  non-trivial by \eqref{eq:minimizer-u0}) yielding $\sum_{\mv,\mw \in \mV} b(\mv,\mw) \vert \nabla_{\mv,\mw} u_0 \vert^p >0$ and hence $\Qneubc(u_0) > 0$ in this case. Likewise, if $c(\mv_0) \vert \varphi_0^\mG(\mv_0) \vert^p > 0$ for some $\mv_0 \in \mV$, it follows -- again by \eqref{eq:bottom-of-p-spec-ground-state} -- that $\Qneubc(u_0) \geq c(\mv_0) \vert u_0(\mv_0) \vert^p > 0$.
Therefore, an $\ell^q-\ell^p$-Poincaré inequality \eqref{eq:poinc-neu} holds with $M:=\Qneu(u_0)^{-1}$:
this implies \eqref{item:poincare-robin}. 

The proof of \eqref{item:poincare-dirichlet} can be performed similarly.
\end{proof}

The embedding of $\wqpneubc(\mV,m)$ into $\ell^q(\mV,m)$ is certainly compact whenever $\mV$ is finite, but some classes of infinite graphs (including radially symmetric trees with fast growth \cite{BonGolKel15,Mel17}) are also known to fulfill this condition, at least for $p=q=2$. Let us present another interesting class: graphs of finite measure and finite diameter (with respect to a specific edge weight).  To this end, for a given graph $\mG = \Vmbc$, we define the \emph{edge weight inverted graph} 
\begin{equation}\label{eq:def-g-1}
\mG^{-1} := (\mV,m,b^{-1},c),
\end{equation}
where \[
b^{-1}(\mv,\mw) :=\begin{cases}
 b(\mv,\mw)^{-1},\quad&\hbox{ whenever }b(\mv,\mw)>0,\\
 0, &\hbox{ else}.
 \end{cases}
 \]
 (Note that $\mG = \mG^{-1}$ if and only if $b=b_{\mathrm{st}}$.)
Here we use again the notation introduced in~\eqref{eq:qbdist-def}.

\begin{lemma}\label{lem:geohaekelp}
Let $\mG=\Vmbc$ be a graph. Assume that 
\[
m(\mV)<\infty\qquad\hbox{and}\qquad \Diam_{p}(\mG^{-1}):=\sup_{\mv,\mw\in \mV}\dist_{p,b^{-1}}(\mv,\mw)^{p-1}<\infty.
\]
Then $\wqpneubc(\mV,m)$ is compactly embedded in $\ell^r(\mV,m)$, for any $q,r\in [1,\infty)$.
\end{lemma}

\begin{proof}
To begin with, we mention that $\ell^\infty(\mV)$ is compactly embedded in $\ell^r(\mV,m)$ whenever $m(\mV)<\infty$: this has been observed in~\cite[Lemma~2.2]{HofKenMug22}. 

The missing step is achieved adapting the proof of~\cite[Theorem~4.3]{GeoHaeKel15} to the case $p\ne 2$: we present it for the sake of self-containedness. Take $f\in \wqpneubc(\mV,m)$: it can be proved like in~\cite[Lemma~3.4]{GeoHaeKel15} that for any $\mv\in \mV$, any reference point $\mv_0\in \mV$ and any path $(\mv_0,\ldots,\mv_n=\mv)$ connecting them, one has
\[
\begin{split}
|f(\mv)-f(\mv_0)|&\le \left( \sum_{i=1}^n b(\mv_i,\mv_{i-1})|f(\mv_i)-f(\mv_{i-1})|^p\right)^\frac{1}{p}
\left(\sum_{i=1}^n b(\mv_i,\mv_{i-1})^{ \frac{1}{1-p}} \right)^\frac{p-1}{p}\\
&\le \|f\|_{\wqpneubc}\Diam_{p}(\mG^{-1})^\frac{1}{p}.
\end{split}
\]
It immediately follows from the triangle inequality that $f\in \ell^\infty(\mV)$.
\end{proof}

\begin{rem}
(i) Graphs that satisfy $\Dneubc\subset \ell^\infty(\mV)$ are said to be \emph{canonically $p$-compactifiable}; this notion has been introduced in~\cite{GeoHaeKel15} for $p=2$, and has been actively studied ever since: several examples of such graphs are known (for $p=2$ only; however, observe that if $b\in \ell^1(\mV\times \mV)$ and $c\in \ell^1(\mV)$, then $\mathcal D^{p'}_{b,c}\subset \Dneubc$ for all $p'>p$). If $\mG$ is canonically $p$-compactifiable and the measure space $(\mV,m)$ is finite, then $ \wqpneubc(\mV,m) \subset \ell^\infty(\mV)\hookrightarrow \ell^p(\mV,m)$, and hence in particular the Poincaré inequality \eqref{eq:poinc-neu}, hold for any canonically $p$-compactifiable graph.

(ii) $\ell^p-\ell^p$-Poincaré inequalities for Dirichlet boundary conditions have been proved in~\cite[Section~1.6.2]{MazSolTol23} for $c\equiv 0$ and a specific choice of probability measure $m$.

(iii) In the specific case of infinite $\mV$ we may introduce further discrete Sobolev spaces by
\[
{\overset{\circ\circ}{w}}^{1,p,p}_{b,c}(\mV,m) :=\overline{c_{00}(\mV)}^{\|\cdot\|_{\wppneubc}}\quad\hbox{and}\quad 
{\overset{\circ\circ}{w}}^{1,p,p}_{b,c;\mV_0}(\mV,m):=\overline{c_{00}(\mV \setminus \mV_0)}^{\|\cdot\|_{\wppneubc}}
\]
and again, by \cite[Corollary~3.10]{Mug14}, these are separable, reflexive Banach spaces, and indeed lattice ideals of the corresponding spaces introduced above, see \cite[Lemma~3.11]{Mug14}.

In the case of $b\in \ell^1(\mV\times \mV)$, $\ell^1-\ell^2$-Poincaré-inequalities for ${\overset{\circ\circ}{w}}^{1,p,p}_{b,c}(\mV,m)$ can be derived from the results in~\cite[Chapter I, Section 4.A]{Woe00}.
\end{rem}

\section{The $p$-torsion function}\label{sec:torsionfunctionandrigidity}

In this section we are going to look for a solution of
\begin{equation}\label{eq:discr-ellipt-1-p}
\Lneubc u = \mathbf{1}\quad \text{in }\mV,
\end{equation}
or else of its version
\begin{equation}\label{eq:discr-ellipt-1-p-dir}
\Ldirbc u = \mathbf{1}\quad \text{in }\mV\setminus\mV_0,
\end{equation}
 with Dirichlet conditions on some $\mV_0\subset\mV$:
these equations are purely symbolic, and we introduce a convenient notion of solutions as follows.

\begin{defi}\label{defi:torsion}
Let $\mG=\Vmbc$ be a graph, and let $q\in [1,\infty)$. A \emph{weak $\ell^q$-solution} of~\eqref{eq:discr-ellipt-1-p} (resp., of \eqref{eq:discr-ellipt-1-p-dir}) is a function $u \in {\wqpneubc(\mV,m)}$ (resp., $u \in { \wqpdirbc(\mV,m)}$) such that
\begin{equation}\label{eq:weakneu}
\begin{split}
&\frac12 \sum_{\mv,\mw \in \mV} b(\mv,\mw)\vert \nabla_{\mv,\mw}u \vert^{p-2}\nabla_{\mv,\mw}u\nabla_{\mv,\mw}h\\
&\qquad + \sum_{\mv \in \mV} c(\mv)\vert u(\mv) \vert^{p-2}u(\mv)h(\mv)=\sum_{\mv\in \mV}h(\mv)m(\mv)\quad \hbox{for all }h\in \wqpneubc(\mV,m) \cap \ell^1(\mV,m)
\end{split}
\end{equation}
(resp., such that
\begin{equation}\label{eq:weakdir}
\begin{split}
&\frac12 \sum_{\mv,\mw \in \mV} b(\mv,\mw)\vert \nabla_{\mv,\mw}u \vert^{p-2}\nabla_{\mv,\mw}u\nabla_{\mv,\mw}h\\
&\qquad+ \sum_{\mv \in \mV \setminus \mV_0} c(\mv)\vert u(\mv) \vert^{p-2}u(\mv)h(\mv)=\sum_{\mv\in \mV\setminus  \mV_0}h(\mv)m(\mv) \quad\hbox{for all }h\in \wqpdirbc(\mV,m) \cap \ell^1(\mV,m)\ ).
\end{split}
\end{equation}

A \emph{$p$-torsion function} 
for $\mG$  (resp., for $\mG$ with Dirichlet conditions at $\mV_0 \neq \emptyset$) is any weak $\ell^q$-solution  $\tauneubc:=u$ of~\eqref{eq:discr-ellipt-1-p} (resp., any weak $\ell^q$-solution  $\taudirbc:=u$ of~\eqref{eq:discr-ellipt-1-p-dir}), for some $q\in [1,\infty)$.

Finally, $\mG$ is called \emph{(uniquely) $q-p$-torsional-admissible} if there exists a (unique) weak $\ell^q$-solution of \eqref{eq:discr-ellipt-1-p} and, likewise, \emph{(uniquely) $q-p$-torsional-admissible} with Dirichlet conditions at $\mV_0 \neq \emptyset$ if there exists a (unique) weak solution of
\eqref{eq:discr-ellipt-1-p-dir}.\end{defi}

\begin{rem}\label{rem:pointw}
Observe that weak solutions are also pointwise solutions, as one sees testing both sides of \eqref{eq:discr-ellipt-1-p} or \eqref{eq:discr-ellipt-1-p-dir} against the Dirac delta $h:=\delta_\mv$ for any point $\mv$: indeed, pointwise solutions of elliptic equations associated with the $p$-Schrödinger operators on graphs have been studied in several articles, up to the recent investigations in~\cite{Fis23}. Moreover, any weak $\ell^q$-solution $u$ of \eqref{eq:discr-ellipt-1-p} (resp., \eqref{eq:discr-ellipt-1-p-dir})  also satisfies the balance condition
\[
\sum_{\mv \in \mV} c(\mv) \vert u(\mv) \vert^{p-2}u(\mv) = \sum_{\mv \in \mV} m(\mv)
\]
whenever $m(\mV) < \infty$; this can be seen letting $h := \mathbf{1} \in \wqpneubc(\mV,m)$ (resp., $h := \mathbf{1}_{\mV\setminus \mV_0} \in \wqpdirbc(\mV,m)$).
\end{rem}

\begin{rem}
Let $\mG=\Vmbc$ be a combinatorially locally finite graph. Mimicking~\cite{BraRuf17} (Euclidean domains, general $p$) and~\cite{AdrSet23} (combinatorially locally finite graphs, $p=2$), one may introduce a very weak notion of $p$-torsion function for $\mG$ with Dirichlet conditions at $\mV_0$ as follows: 
Consider a growing family of finite subgraphs $\mG_n$ induced by vertex sets $(\mV_n)_{n\in \N}$ that exhausts  $\mG$, see~\cite[Definition~3.3]{Mug13}: because $\mG$ is combinatorially locally finite, the graphs $\mG_n$ are finite, \autoref{ass:poincare-inequality} (by \autoref{cor:compact-admiss}) and \autoref{assum:finite-meas} are satisfied and we deduce from that \autoref{cor:wellp} that there exists a sequence of corresponding $p$-torsion functions $\tau^{\mG_n}_p$ with Dirichlet conditions outside $\mV_n$, i.e., of (unique) solutions to
\[
\mathcal L^{\mG_n;\mV\setminus \mV_n}_p u=\mathbf{1}_{\mV_n}.
\]
Clearly, $\mathcal L^{\mG_n;\mV\setminus \mV_n}_p u=\mathcal L^{\mG;\mV\setminus \mV_n}_p u$ whenever $u$ vanishes outside $\mV_n$: because
$(\mathbf{1}_{\mV_n})_{n\in \N}$ is pointwise monotonically increasing, and by the comparison principle in~\cite[Theorem~1]{ParChu11}  we deduce that $(\tau^{\mG_n}_p)_{n\in \N}$ is a sequence of positive functions that is pointwise monotonically increasing. We may, thus, consider the function $\tilde{\tau}^\mG_p$ defined by $\tilde\tau_p^\mG(\mv):=\lim_{n\to\infty}\tau^{\mG_n}_p(\mv)$, $\mv \in \mV$. At this point, one can still observe that while in general $\tilde\tau_p^\mG$ may be identically $+\infty$, the local Harnack inequality in~\cite[Lemma~4.4]{Fis23} guarantees that the function is everywhere finite if it is finite at least at one vertex, provided $\mG$ is connected.	
\end{rem}

We will in the present article follow a different approach, better suited to our general setting, as described in the following. To this aim we will sometimes restrict to a classic ensemble of graphs.


\begin{assum}\label{assum:finite-meas}
The underlying discrete measure space $(\mV,m)$ is finite, i.e., $m(\mV) < \infty$.
\end{assum}

Clearly, if $q=1$, every $p$-torsion function belongs, in particular, to $\ell^1(\mV,m)$, regardless of \autoref{assum:finite-meas}. Observe that, under \autoref{assum:finite-meas}, $\wqpneubc(\mV,m)\hookrightarrow \ell^p(\mV,m)\hookrightarrow \ell^1(\mV,m)$ (and, \textit{a fortiori}, $\wqpdirbc(\mV,m)\hookrightarrow \ell^p(\mV,m)\hookrightarrow \ell^1(\mV,m)$), hence the conditions in~\eqref{eq:weakneu} and~\eqref{eq:weakdir} simplify a bit. In particular, we observe the following.

\begin{lemma}
Let $\mG=\Vmbc$ be a graph.
Under~\autoref{assum:finite-meas}, any $p$-torsion function for $\mG$,  or for $\mG$ with Dirichlet conditions at $\mV_0$, is of class $\ell^1(\mV,m)$.
\end{lemma}

Plugging $h=u$ into~\eqref{eq:weakneu} or in~\eqref{eq:weakdir}
leads us to considering the functional 
 \begin{equation}\label{eq:fneu-def}
\Fneubc (u):=
\Qneubc(u) -\sum_{\mv\in\mV}u(\mv)m(\mv),\quad u\in \wqpneubc(\mV,m),
 \end{equation}
or else its restriction to $\wqpdirbc(\mV,m)$. One immediately sees that the map $\Qneubc:\wqpneubc(\mV,m) \to \R$ and hence $\Qdirbc:\wqpdirbc(\mV,m) \to \R$ are differentiable and, in particular,
\begin{align}\label{eq:q'pq}
\begin{aligned}
&(\Qneubc)'(u)(h)= \frac12 \sum_{\mv,\mw \in \mV} b(\mv,\mw)\vert \nabla_{\mv,\mw}u \vert^{p-2}\nabla_{\mv,\mw}u\nabla_{\mv,\mw}h\\
&\qquad\quad\quad + \sum_{\mv \in \mV} c(\mv)\vert u(\mv) \vert^{p-2}u(\mv)h(\mv) \quad\hbox{for all } u,h \in \wqpneubc(\mV,m)
\end{aligned}
\end{align}
and
\begin{align}\label{eq:q'pq-dir}
\begin{aligned}
&(\Qdirbc)'(u)(h)= \frac12 \sum_{\mv,\mw \in \mV} b(\mv,\mw)\vert \nabla_{\mv,\mw}u \vert^{p-2}\nabla_{\mv,\mw}u\nabla_{\mv,\mw}h\\
&\qquad\quad\quad + \sum_{\mv \in \mV \setminus \mV_0} c(\mv)\vert u(\mv) \vert^{p-2}u(\mv)h(\mv) \quad\hbox{for all } u,h \in \wqpdirbc(\mV,m).
\end{aligned}
\end{align}
This shows that a (unique) minimizer of $\Fneubc$ in $\wqpneubc(\mV,m)$ or of its restriction to $\wqpdirbc(\mV,m)$ yields a (unique) $p$-torsion function in the sense of \autoref{defi:torsion}, and vice versa. We state sufficient conditions for the existence of such a (unique) minimizer, next.

\begin{lemma}\label{lem:fp-convex}
Let $\mG=\Vmbc$ be a graph and let \autoref{ass:poincare-inequality} be satisfied for $q \in [1,\infty)$.
Then both $\Fneubc$ and its restriction to $\wqpdirbc$ are differentiable on $\wqpneubc(\mV,m)$ and $\wqpdirbc(\mV,m)$, respectively (and in particular lower semicontinuous), strictly convex. If, additionally,  $q \in (1,\infty)$ and \autoref{assum:finite-meas} is satisfied; or if $q=1$, then both $\Fneubc$ and its restriction to $\wqpdirbc$ are coercive.
\end{lemma}

\begin{proof}
We will only consider the case $\sup_{\mv\in\mV}c(\mv)>0$, the Dirichlet case being completely analogous.

As $\wqpneubc(\mV,m) \ni u \mapsto \Qneubc(u)\in \R$ is differentiable it follows immediately that $\Fneubc:\wqpneubc(\mV,m)\to \R $ is differentiable with
 \begin{equation}\label{eq:discr-ellipt-1-p-weak}
\big(\Fneubc\big)'(u)(h)=(\Qneubc)'(u)(h)- \sum_{\mv\in\mV}h(\mv)m(\mv)\quad\hbox{for all }u,h\in \wqpneubc (\mV,m).
\end{equation}
To check strict convexity, we write
 \[
 \begin{split}
 \mathfrak F^{(1)}_{p,b} (u)&:=
 \sum_{\mv,\mw \in \mV} \frac{b(\mv,\mw)}{2p}\vert \nabla_{\mv,\mw}u \vert^{p},\\
  \mathfrak F^{(2)}_{p,c} (u)&:= \sum_{\mv \in \mV} \left(\frac{c(\mv)}{p}\vert u(\mv) \vert^{p}-u(\mv)m(\mv)\right),
  \end{split}
 \]
and observe that $\Fneubc $ is strictly convex as it is the sum of a convex mapping $ \mathfrak F^{(1)}_{p,b}$ and a strict(!) convex mapping $\mathfrak F^{(2)}_{p,c}$: indeed, convexity of the former follows from the convexity of $|\cdot|^p$, 
whereas strict convexity of the latter follows from convexity of $\frac{c}{p}|\cdot|^p-m( \cdot )$ and the fact that $\sup_{\mv \in \mV} c(\mv) > 0$ (note that (not necessarily strict) convexity remains true for any $c\ge 0$ and any $m\in \R$).

To show the coercivity we have to show that, for every $\beta \in \mathbb{R}$, the sub-level set
\[
\mathcal U_{\beta} = \{ u \in \wqpneubc(\mV,m) \: : \: \Fneubc(u) \leq \beta \}
\]
is bounded in $\wqpneubc(\mV,m)$ with respect to the
norm $u \mapsto \Qneubc(u)^\frac{1}{p}$ and we first consider the case where $q=1$. To this aim, fix $\beta\in\mathbb{R}$ and let $u\in\mathcal U_{\beta}$, i.e., $\Qneubc(u) - \sum\limits_{\mv \in \mV} u(\mv)m(\mv) \leq \beta$ and therefore by~\eqref{eq:poinc-neu}
\begin{align}\label{main-b}
\begin{aligned}
\Qneubc(u) &\leq \beta +  \sum\limits_{\mv \in \mV} u(\mv)m(\mv) \leq \beta + \Vert u \Vert_{\ell^1(\mV,m)} \\&\leq \beta + M\Qneubc(u)^\frac{1}{p} \leq \beta + \frac{M^{p'}}{p'} + \frac{\Qneubc(u)}{p},
\end{aligned}
\end{align}
by Young's inequality. Thus, there exists a constant $C := C(\beta,M,p)$ depending on $\beta,M,p$ such that  
\[
\Qneubc(u) \leq C(\beta,M,p) \qquad \text{for all $u \in U_\beta$,}
\]
implying that the level set $\mathcal{U}_\beta$ is indeed bounded.

If $q \in (1,\infty)$ and in addition \autoref{assum:finite-meas} holds, then by Hölder's inequality that for any $u \in \wqpneubc(\mV,m)$
\begin{align}
\begin{aligned}
\Qneubc(u) &\leq \beta + \Vert u \Vert_{\ell^1(\mV,m)} \leq \beta + m(\mV)^\frac{1}{q'}\Vert u \Vert_{\ell^q(\mV,m)} \leq \beta + Mm(\mV)^\frac{1}{q'}\Qneubc(u)^\frac{1}{p}
\end{aligned}
\end{align}
with $q':=\frac{q}{q-1}$. Again, Young's inequality then implies the coercivity.
\end{proof}

On general infinite graphs, $p$-torsion functions need not exist, or -- if they do exist -- need not enjoy good summability properties: we are now finally in the position to deduce appropriate $q-p$-torsional-admissibility results. Let us begin with the case $q \in (1,\infty)$.
\begin{theo}\label{cor:wellp}
Let $q \in (1,\infty)$. Under \autoref{ass:poincare-inequality} and \autoref{assum:finite-meas}, any graph $\mG = \Vmbc$ is uniquely $q-p$-torsional-admissible (resp., uniquely $q-p$-torsional-admissible with Dirichlet conditions at $\mV_0 \neq \emptyset$). In particular, the corresponding $p$-torsion function is the minimizer of $\Fneubc$ on $\wqpneubc$ (resp., on $\wqpdirbc$).
\end{theo}
 In the case $q=p$, the same result has been obtained -- with a slightly different proof -- in \cite{MazTol23} for the normalized $p$-Laplacian with Dirichlet conditions on some $\mV_0\ne\emptyset$ (so, $c\equiv 0$ and $m^{-1}=\sum\limits_{\mw\in \mV}b(\cdot,\mw)$).

\begin{proof}
In order to prove that~\eqref{eq:discr-ellipt-1-p} has a unique weak solution, it suffices to show that the differentiable functional $\Fneubc $ has a unique minimum $\tau$. But this is an immediate consequence of standard results in convex analysis (cf.~\cite[Theorem~E.38 and Remark E.39]{ChiFas10} or~\cite[Theorem~5.1.1 and Corollary~5.1.1]{Kes04}), since by \autoref{prop:intro-plapl} and \autoref{lem:fp-convex} $\Fneubc$ is a lower semicontinuous, strictly convex coercive functional defined on the reflexive Banach space $\wqpneubc(\mV,m)$.

A natural variation of the above arguments can be used to complete the proof in the Dirichlet case, too: this finishes the proof.
\end{proof}

The upcoming condition will be crucial for us to guarantee the unique $p$-torsinal-adimissibility of a given graph in the remaining case where $q=1$. This will also be the central assumption in \autoref{sec:torsionalrig} and \autoref{sec:bounds} below.
 
\begin{assum}\label{ass:compact-embedding-l1}
Either $\sup_{\mv \in \mV} c(\mv) > 0$ and $\wneubc(\mV,m)$ is compactly embedded in $\ell^1(\mV,m)$; or $\mV_0 \neq \emptyset$ and $\wdirbc(\mV,m)$ is compactly embedded in $\ell^1(\mV,m)$.
\end{assum}
Note that by \autoref{cor:compact-admiss}, \autoref{ass:compact-embedding-l1} implies the validity of \autoref{ass:poincare-inequality}; in particular, it is important to assume $\sup_{\mv \in \mV} c(\mv) > 0$ (resp., $\mV_0 \neq \emptyset$) to guarantee that the constant $M > 0$ obtained in the proof of \autoref{cor:compact-admiss} is not equal to $+\infty$, cf.\ \eqref{eq:poinc-neu} (resp.\ \eqref{eq:poinc-neu-dir}).

\begin{theo}\label{thm:well-p-l1}
Under \autoref{ass:compact-embedding-l1}, any graph $\mG = \Vmbc$ is uniquely $1-p$-torsional-admissible (resp., uniquely $1-p$-torsional-admissible with Dirichlet conditions at $\mV_0 \neq \emptyset$). In particular, the corresponding $p$-torsion function is the minimizer of $\Fneubc$ on $\wneubc$ (resp., on $\wdirbc$).
\end{theo}

\begin{proof}
Again, we will only consider the case $\sup_{\mv \in \mV} c(\mv) > 0$, the Dirichlet case being analogous.  To this end, let $(u_n)_{n \in \mathbb{N}} \subseteq \wneubc(\mV,m)$ be a minimizing sequence for $\Fneubc$ on $\wneubc(\mV,m)$, i.e., $\mathfrak{F}_p^\mG(u_n) \rightarrow \inf_{v \in \wneubc(\mV,m)} \mathfrak{F}_p^\mG(v)$ as $n \rightarrow \infty$, which is necessarily bounded in $w^{1,p}_{b,c}(\mV;m)$ due to coercivity. Using the compactness of the embedding $\wneubc(\mV,m) \hookrightarrow \ell^1(\mV,m)$, it follows that (up to a subsequence) $u_n \rightarrow u$ in $\ell^1(\mV,m)$ for some $u \in \ell^1(\mV,m)$. In particular,
\[
\vert u_n(\mw) - u(\mw) \vert m(\mw) \leq \sum_{\mv \in \mV} \vert u_n(\mv) - u(\mv) \vert m(\mv)  \rightarrow 0 \qquad \text{as $n \rightarrow \infty$},
\]
and thus, $u_n(\mw) \rightarrow u(\mw)$ for every $\mw \in \mV$. Fatou's lemma then implies $\liminf_{n \rightarrow \infty} \Qneubc(u_n) \geq \Qneubc(u)$ and since the linear term appearing in $\Fneubc$ is even continuous w.r.t.\ $\ell^1(\mV,m)$, we deduce 
 \begin{align}\label{eq:minizing-sequence}
\inf_{v \in \wneubc(\mV,m)} \mathfrak{F}_p^\mG(v) = \liminf_{n \rightarrow \infty} \mathfrak{F}_p^\mG(u_n) \geq \mathfrak{F}_p^\mG(u) = \mathcal{Q}_p^\mG(u) - \sum_{\mv \in \mV} u(\mv)m(\mv), 
\end{align}
yielding that $u \in \wneubc(\mV,m)$ (as $u \in \ell^1(\mV,m)$), in particular, $u$ is a minimizer for $\mathfrak{F}_p^\mG$ on $\wneubc(\mV,m)$. 

Uniqueness follows according to the strict convexity of $\Fneubc$ shown in \autoref{lem:fp-convex}. This finishes the proof.
\end{proof}

We now turn to the order properties of torsion functions and present the following \textit{minimum principle}, see~\cite[Theorem~1.7]{KelLenWoj21} and~\cite[Proposition~4.1]{Amg08} for comparable results in the case of $p=2$ and finite graphs, respectively: since any minimizer $u$ of $\Fneubc$ on $\wqpneubc(\mV,m)$ (resp., $\wqpdirbc(\mV,m)$) can be replaced by $\vert u \vert$ implying $\Qneubc(\vert u \vert) \leq \Qneubc(u)$ by the inverse triangle inequality, every minimizer has to be non-negative. The following minimum principle represents a refinement of this observation.

\begin{lemma}\label{lem:torsionfunctionpositive}
Let $\mG=\Vmbc$ be a graph and let $q \in [1,\infty)$. 
Any weak $\ell^q$-solution of \eqref{eq:discr-ellipt-1-p} (resp, \eqref{eq:discr-ellipt-1-p-dir} if $\mV_0 \neq \emptyset$) is strictly positive. In particular, any $p$-torsion function $\tauneubc$ (resp., $\taudirbc$) of a graph $\mG$ (resp., of a graph $\mG$ with Dirichlet conditions at $\mV_0 \neq \emptyset$) is strictly positive.
\end{lemma}

\begin{proof}
We only show the case without Dirichlet condition as the Dirichlet case is completely analogous: Consider the set $\mV_- := \{ \mv \in \mV \: : \: \tauneubc(\mv) < 0 \}$ and let $\tauneubc \in \wqpneubc(\mV,m) \cap \ell^1(\mV,m)$ be a weak $\ell^q$-solution in the sense of \eqref{eq:discr-ellipt-1-p}. Then, testing against $h:= (-\tauneubc)_+ \in \wqpneubc(\mV,m) \cap \ell^1(\mV,m)$, where
\[
\big(-\tauneubc \big)_+(\mv) := \max\big\{ -\tauneubc(\mv),0 \big\}, \qquad \text{for $\mv \in \mV$,}
\]
it follows on the one hand that
\begin{align*}
\frac{1}{2} \sum_{\mv,\mw \in \mV} b(\mv,\mw) \vert \nabla_{\mv,\mw} \tauneubc \vert^{p-2} \nabla_{\mv,\mw} \tauneubc \nabla_{\mv,\mw} h + \sum_{\mv \in \mV} c(\mv) \vert \tauneubc(\mv) \vert^{p-2} \tauneubc(\mv) h(\mv) = \sum_{\mv \in \mV} (-\tauneubc)_+(\mv) m(\mv) \geq 0.
\end{align*}
On the other hand, one observes that
\begin{align*}
&\frac{1}{2} \sum_{\mv,\mw \in \mV} b(\mv,\mw) \vert \nabla_{\mv,\mw} \tauneubc \vert^{p-2} \nabla_{\mv,\mw} \tauneubc \nabla_{\mv,\mw} h + \sum_{\mv \in \mV} c(\mv) \vert \tauneubc(\mv) \vert^{p-2} \tauneubc(\mv) h(\mv) \\&\qquad\quad\quad = - \sum_{\stackrel{\mv \in \mV_-}{\mw \in \mV \setminus \mV_-}} b(\mv,\mw) \vert \nabla_{\mv,\mw} \tauneubc \vert^{p} - \sum_{\mv \in \mV_-} c(\mv) \vert \tauneubc(\mv) \vert^{p} \leq 0.
\end{align*}
This finally implies that $h = (-\tauneubc)_+ \equiv 0$, i.e., $\mV_- = \emptyset$ which shows positivity of $\tau^\mG_p$.

To deduce strict positivity, assume now contrarily that $\tauneubc(\mv_0)=0$ for some $\mv_0 \in \mV$. Then 
\[
1=\Lneubc \tauneubc(\mv_0)=\sum_{\mw\in\mV} b(\mv_0,\mw)(-\tauneubc(\mw))|\tauneubc(\mw)|^{p-2}\le 0,
\]
a contradiction.
\end{proof}

\section{The $p$-torsional rigidity}\label{sec:torsionalrig}
Extending the ideas in~\cite{MazTol23}, where only the case $c\equiv 0$ was considered, let us
define the $p$-torsional rigidity of a graph $\mG=\Vmbc$ variationally, in terms of the so-called \textit{Pólya quotients} 
\begin{equation}\label{eq:polya-def}
\Polneubc(u):=\frac{\|u\|^p_{\ell^1(\mV,m)}}{p\Qneubc(u)}\quad\hbox{or}\quad
\Poldirbc(u):=\frac{\|u\|^p_{\ell^1(\mV\setminus \mV_0,m)}}{p\Qdirbc(u)}:
\end{equation}
we will afterwards discuss its relation with the $p$-torsion function and from now on, we restrict onto the case where $q=1$. In particular, the $p$-torsion function for $\mG$ (resp., for $\mG$ with Dirichlet conditions at $\mV_0$) is the unique weak $\ell^1$-solution of \eqref{eq:discr-ellipt-1-p} (resp., of \eqref{eq:discr-ellipt-1-p-dir}).

\begin{defi}
Let $\mG=\Vmbc$ be a graph, and $\mV_0 \subset \mV$.
The associated \emph{$p$-torsional rigidity} $T_p(\mG)$ or $T_p(\mG;\mV_0)$
 is defined as
\begin{equation}\label{eq:tors-def-1}
T_p(\mG) := \sup_{u \in \wneubc(\mV,m)} \frac{\|u\|^p_{\ell^1(\mV,m)}}{p\Qneubc(u)} \quad\hbox{or}\quad
 T_p(\mG;\mV_0) := \sup_{u \in \wdirbc(\mV,m)} \frac{\|u\|^p_{\ell^1(\mV\setminus \mV_0,m)}}{p\Qdirbc(u)}\hbox{ if }\mV_0 \neq \emptyset,
\end{equation}
respectively.
\end{defi}

Formally, the right hand sides in~\eqref{eq:tors-def-1} may be $+\infty$ if $\wneubc(\mV,m)$ is not embedded in $\ell^1(\mV,m)$. Furthermore, the supremum over the Pólya quotients in \eqref{eq:tors-def-1} is actually a maximum whenever $\wneubc(\mV,m)$ is compactly embedded in $\ell^1(\mV,m)$, i.e., under \autoref{ass:compact-embedding-l1}.

It will be occasionally convenient to identify all elements in $\mV_0$: to this purpose, observe that if the subgraph of $\mG= \Vmbc$ induced by $\mV\setminus\mV_0$ is connected, so is the subgraph of $\mG' = (\mV',m',b',c')$ induced by $\mV'\setminus\mV'_0$, where
\[
\mV':=(\mV\setminus \mV_0)\sqcup\mV'_0
\]
for a singleton $\mV'_0:=\{\mv_0'\}$, with
\[
b'(\mv,\mw)
:=\begin{cases}
b(\mv,\mw), \quad &\mv\in \mV'\setminus \mV_0'\simeq \mV\setminus \mV_0\hbox{ and/or }\mw\in \mV'\setminus \mV_0'\simeq \mV\setminus \mV_0,\\
0 &\mv',\mw'\in \mV'_0,
\end{cases}
\]
along with
\[
c'(\mv):=\begin{cases}
c(\mv), \quad &\mv\in \mV'\setminus \mV_0',\\
\sum_{\mw \in \mV_0} c(\mw), &\mv\in\mV_0'.
\end{cases}
\qquad\hbox{and}\qquad  m'(\mv):=\begin{cases}
m(\mv), \quad &\mv\in \mV'\setminus \mV_0',\\
\sum_{\mw \in \mV_0} m(\mw), &\mv\in\mV_0'.
\end{cases}
\]
We motivate the above operation on $\mG$ by formulating the following, which shows that two relevant quantities are invariant under the above reduction: the bottom of the $p$-spectrum will be introduced in~\eqref{eq:variational-characterization-ground-state} and successively studied in more detail.
\begin{lemma}\label{lem:v0-singleton} Under the above reduction of $\mG$ to a graph $\mG'$ having only one Dirichlet condition, the bottom of the $p$-spectrum on $\mG$ and $\mG'$ coincide. Likewise, the $p$-torsional rigidity of $\mG$ and $\mG'$ coincide.
\end{lemma}
\begin{proof}
First of all, let us observe that $\wdirbc(\mV,m)\simeq \overset{\circ}{w}^{1,p,1}_{b,c;\mV'_0}(\mV',m')$, as every function in $\wdirbc(\mV,m)$ (resp., in $\wneubc(\mV,m)$) vanishes outside $\mV \setminus \mV_0$ (resp., outside $\mV' \setminus \mV_0'$). Hence the bottom of the $p$-spectrum does not change upon the identification $\mV_0\simeq \mV_0'$, since the measure of the vertices in $\mV_0$  is irrelevant for the Rayleigh quotients of both $\mG$ and $\mG'$, cf.\ \eqref{eq:variational-characterization-ground-state}.

 Likewise,  the $p$-torsional rigidity does not change upon the identification $\mV_0\simeq \mV_0'$, since the measure of the vertices in $\mV_0$ is also irrelevant for the Pólya quotients of both $\mG$ and $\mG'$.
\end{proof}
In the proof of \autoref{lem:v0-singleton} we mentioned the notion of \emph{bottom of the spectrum} of the $p$-Laplacian $\Lneubc$ and $\Ldirbc$ (or, simply, bottom of the $p$-spectrum of $\mG$): this is found taking the infimum of the \emph{Rayleigh quotient}, i.e.,
\begin{align}\label{eq:variational-characterization-ground-state}
\begin{aligned}
\lambda_{0,p}(\mG) &:= \inf_{0 \neq f \in \wppneubc(\mV,m)} \frac{p\Qneubc(f)}{\Vert f \Vert_{\ell^p(\mV,m)}^p} \\ \Bigg(\text{resp., } \quad 
    \lambda_{0,p}(\mG;\mV_0) &:= \inf_{0\neq f \in \wppdirbc(\mV,m)} \frac{p\Qdirbc(f)}{\Vert f \Vert_{\ell^p(\mV \setminus \mV_0, m)}^p}\Bigg);
\end{aligned}
\end{align}
which need not be an eigenvalue unless the infimum is actually attained, i.e., if it is a minimum: this happens, in particular, if $\wppneubc(\mV,m)$ (resp., $\wppdirbc(\mV,m)$) is compactly embedded in $\ell^p(\mV,m)$, and in this case any minimizer for this Rayleigh quotient is called \emph{ground state}, whereas $\lambda_{0,p}(\mG)$ (resp., $\lambda_0(\mG;\mV_0)$) is referred to as the \emph{smallest eigenvalue}.

Concerning the reduction of the Dirichlet vertices we have discussed above, \autoref{lem:v0-singleton} motivates us to stipulate the following.

\begin{conv}\label{conv:dirichlet-singleton}
If Dirichlet conditions are imposed on a set $\mV_0$, then $\mV_0$ may without loss of generality be assumed to be a singleton.
\end{conv}
We will later illustrate this reduction in the case of star graphs, see  \autoref{rem:reduction-star-graphs} below.

It is known that the ($p$-)torsional rigidity of domains with Dirichlet boundary conditions can be described in terms of the integral of the ($p$-)torsion function~\cite{Pol48,Bra14}:
let us extend this observation to the $p$-torsional rigidity on combinatorial graphs.

\begin{prop}\label{prop:variational-char}
Let $\mG = (\mV,m,b,c)$ be a graph. If $T_p(\mG) < \infty$ (resp., $T_p(\mG;\mV_0)< \infty$ if $\mV_0 \neq \emptyset$), then 
\begin{align}\label{eq:torsional-rigidity-fp}
T_p(\mG) =  \bigg( \frac{p}{1-p} \inf_{u \in \wneubc(\mV,m)} \Fneubc(u) \bigg)^{p-1} \quad \text{holds,}
\end{align}
in particular, $\inf_{u \in \wneubc(\mV,m)} \Fneubc(u) > -\infty$ (resp., $\inf_{u \in \wdirbc(\mV,m)} \Fneubc(u) > -\infty$). 

If, additionally, \autoref{ass:compact-embedding-l1} is satisfied, then the Pólya quotient admits a maximum over $\wneubc(\mV,m)$ (resp., over $\wdirbc(\mV,m)$) and, 
in addition, 
its maximum is precisely the $p$-torsion function for $\mG$ (resp., for $\mG$ with Dirichlet conditions at $\mV_0 \neq \emptyset$), 
i.e.,
\begin{align}\label{eq:torsion-p-Pólya}
\begin{aligned}
    T_p(\mG) &= \Vert \tauneubc \Vert_{\ell^1(\mV,m)}^{p-1} = \bigg( \frac{p}{1-p}\min_{u \in \wneubc(\mV,m)} \Fneubc(u)\bigg)^{p-1} \\
    \Bigg(\hbox{resp., }
    T_p(\mG;\mV_0) &=  \Vert \taudirbc \Vert_{\ell^1(\mV \setminus \mV_0,m)}^{p-1} = \bigg( \frac{p}{1-p}\min_{u \in \wdirbc(\mV,m)} \Fneubc(u) \bigg)^{p-1} \Bigg)
    \end{aligned}
\end{align}
holds.
\end{prop}

\begin{proof} 
The proof closely follows that of \cite[Proposition~2.2]{Bra14}. We only discuss in detail the $p$-torsional rigidity if $\mV_0 = \emptyset$, the Dirichlet case being analogous.

First, for every $v \in \wneubc(\mV,m)$ we see that
\begin{align}\label{eq:first-estimate-function}
\begin{aligned}
    -\Fneubc(v) \leq \Vert v \Vert_{\ell^1(\mV,m)} - \Qneubc(v) &\leq \max_{\lambda \geq 0} \Big( \lambda \Vert v \Vert_{\ell^1(\mV,m)} - \lambda^p \Qneubc(v) \Big) \\&
    \le \frac{p-1}{p} \bigg[ \frac{\Vert v\Vert_{\ell^1(\mV,m)}^p}{p\Qneubc(v)} \bigg]^\frac{1}{p-1} \\
    & \leq \frac{p-1}{p} \bigg[ \sup_{u \in \wneubc(\mV,m)} \frac{\Vert u \Vert_{\ell^1(\mV,m)}^p}{p\Qneubc(u)} \bigg]^\frac{1}{p-1},
    \end{aligned}
\end{align}
where the penultimate inequality holds by \cite[Lemma~2.1]{Bra14}: equivalently,
\begin{align}\label{eq:proof-p-Pólya-p-torsion}
\frac{1-p}{p} T_p(\mG)^\frac{1}{p-1} = \frac{1-p}{p} \bigg[ \sup_{u \in \wneubc(\mV,m)} \frac{\Vert u \Vert_{\ell^1(\mV,m)}^p}{p\Qneubc(u)} \bigg]^\frac{1}{p-1} \leq \Fneubc(v),
\end{align}
for every $v \in \wneubc(\mV,m)$, and in particular
\begin{align}
-\infty<\frac{1-p}{p} T_p(\mG)^\frac{1}{p-1} \leq \inf_{u \in \wneubc(\mV,m)} \Fneubc(u),
\end{align}
i.e.,
\begin{align}
T_p(\mG) \geq \bigg( \frac{p}{1-p} \inf_{u \in \wneubc(\mV,m)}\Fneubc(u) \bigg)^{p-1}.
\end{align}

To prove the converse inequality, let now $(u_n)_{n \in \mathbb{N}} \subset \wneubc(\mV,m)$ be a sequence achieving the supremum of the Pólya quotient, i.e.,
\[
\lim_{n \rightarrow \infty} \frac{\Vert u_n \Vert_{\ell^1(\mV,m)}^p}{p\Qneubc(u_n)} = T_p(\mG)
\]
and define
\[
v_n:= \bigg( \frac{\Vert u_n \Vert_{\ell^1(\mV,m)}}{p\Qneubc(u_n)} \bigg)^\frac{1}{p-1}
u_n \in \wneubc(\mV,m) \qquad \text{for $n \in \mathbb{N}$}.
\]
Then we observe that
\begin{align*}
     \inf_{u \in \wneubc(\mV,m)} \Fneubc(u) &\leq \Fneubc(v_n) = -\Vert v_n \Vert_{\ell^1(\mV,m)} + \Qneubc(v_n) \\&=
    -\bigg[ \frac{\Vert u_n \Vert_{\ell^1(\mV,m)}}{p\Qneubc(u_n)} \bigg]^\frac{1}{p-1} \Vert u_n \Vert_{\ell^1(\mV,m)} + \bigg[ \frac{\Vert u_n \Vert_{\ell^1(\mV,m)}}{p\Qneubc(u_n)} \bigg]^\frac{p}{p-1} \Qneubc(u_n) \\&=
    -\bigg[ \frac{\Vert u_n \Vert_{\ell^1(\mV,m)}^p}{p\Qneubc(u_n)} \bigg]^\frac{1}{p-1}  + \frac{1}{p} \bigg[ \frac{\Vert u_n \Vert_{\ell^1(\mV,m)}^p}{p\Qneubc(u_n)} \bigg]^\frac{1}{p-1} \\&= \frac{1-p}{p}\bigg[ \frac{\Vert u_n \Vert_{\ell^1(\mV,m)}^p}{p\Qneubc(u_n)} \bigg]^\frac{1}{p-1},
\end{align*}
and thus, letting $n \rightarrow \infty$,
\begin{align*}
    T_p(\mG) \leq \bigg( \frac{p}{1-p} \inf_{u \in \wneubc(\mV,m)} \Fneubc(u) \bigg)^{p-1}
\end{align*}
yields \eqref{eq:torsional-rigidity-fp}.

Because under \autoref{ass:compact-embedding-l1} the $p$-torsion function $\tauneubc$ (which exists according to \autoref{thm:well-p-l1}) is the minimizer of the functional $\Fneubc$ (cf.~\eqref{eq:q'pq}), we have that
\begin{equation}\label{eq:minizer-functional-equation}
\begin{split}
\min_{u \in \wneubc(\mV,m)} \Fneubc(u) 
&= -\sum_{\mv \in \mV} \tauneubc(\mv)m(\mv) + \Qneubc(\tauneubc) \\
&=  -\Vert \tauneubc \Vert_{\ell^1(\mV,m)} + \frac{1}{p} \Vert \tauneubc \Vert_{\ell^1(\mV,m)} 
= \frac{1-p}{p}\Vert \tauneubc \Vert_{\ell^1(\mV,m)}:
\end{split}
\end{equation}
by \eqref{eq:torsional-rigidity-fp}, this implies \eqref{eq:torsion-p-Pólya}. Moreover plugging $\tauneubc \in \wneubc(\mV,m)$ into the Pólya quotient, it follows that
\begin{align}
\Polneubc(\tauneubc) = \frac{\Vert \tauneubc \Vert_{\ell^1(\mV,m)}^{p}}{p\Qneubc(\tauneubc)} = \frac{\Vert \tauneubc \Vert_{\ell^1(\mV,m)}^{p}}{\Vert \tauneubc \Vert_{\ell^1(\mV,m)}} = \Vert \tauneubc \Vert_{\ell^1(\mV,m)}^{p-1},
\end{align}
and by \eqref{eq:torsion-p-Pólya}, the Pólya quotient indeed admits a maximum over $\wneubc(\mV,m)$ given by the $p$-torsion function $\tauneubc$. This finishes the proof.
\end{proof}
\begin{rem}
We might also allow for $q\in (1,\infty)$ and define the $p$-torsional rigidity in \autoref{defi:torsion} by taking the supremum of the Pólya quotients over functions in $\wqpneubc(\mV,m)$ (resp., $\wqpdirbc(\mV,m)$).
However, as $\Vert \cdot \Vert_{\ell^1(\mV,m)}$ is the norm appearing in the Pólya quotient $\Polneubc(u)$ (resp., $\Poldirbc(u)$), a compact embedding of $\wqpneubc(\mV,m)$ (resp., $\wqpdirbc(\mV,m)$) into $\ell^1(\mV,m)$ would not be sufficient to guarantee the existence of a maximizer for the Pólya quotient: this is one of the main reasons for the choice of $q=1$.

Moreover, it is unclear whether the additional \autoref{assum:finite-meas} which is needed in the case where $q \neq 1$ (cf.\ \autoref{lem:fp-convex} and \autoref{cor:wellp}) is a sufficient condition for the embedding $\wneubc(\mV,m) \hookrightarrow \ell^1(\mV,m)$ to be compact, similarly to the metric graph case, cf.\ \cite[Theorem~3.1]{DüfKenMugPlüTäu22}.
\end{rem}
Using this variational characterization for the $p$-torsional rigidity, we can now show the following monotonicity property with respect to the notion of subgraph as introduced in \autoref{sec:general}.

\begin{prop}\label{prop:monotonicity-subgraphs}
Let $\mG' = (\mV,m,b',c')$ and $ \mG =  \Vmbc$ be two graphs on the same vertex set $\mV$. 
If $b' \leq b$ and $c' \leq c$ pointwise, then  $T_p(\mG) \leq T_p(\mG')$. If, additionally, $ \emptyset\ne \mV_0' \subset \mV_0$, then  also $T_p(\mG;\mV_0) \leq T_p(\mG';\mV_0')$. 
\end{prop}
We should emphasize that this statement especially implies that the $p$-torsional rigidity increases under taking subgraphs over the \emph{same} vertex set.
\begin{proof}
Because $\Qneubcpr(u) \leq \Qneubc (u)$ for every $u \in \wneubc(\mV,m)$ it follows that 
 $\wneubc(\mV,m) \subset \wneubcpr (\mV,m)$ (whence every test function for the Pólya quotient $\Polneubc$ is a valid test function for $\Polneubcpr$, too). Using the variational characterization \eqref{eq:tors-def-1}, we deduce that
    \[
    T_p(\mG) = \sup_{u \in \wneubc(\mV,m)} \frac{\Vert u \Vert_{\ell^1(\mV,m)}^p}{p\Qneubc(u)} \leq \sup_{u \in \wneubcpr(\mV,m)} \frac{\Vert u \Vert_{\ell^1(\mV,m)}^p}{p\Qneubcpr(u)} = T_p(\mG'),
    \]
which proves the first inequality.
     
An analogous proof can be performed to compare $T_p(\mG;\mV_0)$ with $T_p(\mG';\mV_0)$, in view of $\wdirbc(\mV,m) \subset \wdirbcpr (\mV,m)$
\end{proof}
In particular, the proof reveals that each graph that contains a subgraph on the same vertex set satisfying an $\ell^1-\ell^p$-Poincaré-inequality is itself $1-p$-torsional-admissible, with smaller $p$-torsional rigidity. This is reminiscent of classical results for the algebraic connectivity of Laplacians that go back to~\cite[Corollary~3.2]{Fie73}.

As a counterpart to the previous monotonicity result, let us show that suitably attaching a graph at another graph may increase the $p$-torsional rigidity. The following can be arguably considered as a generalization of~\cite[Definition~3.8]{BerKenKur19}.

\begin{defi}
Let $\mV^\dagger \subset \mV$ be a (nonempty) set of distinguished vertices in a graph $\mathsf{G} := \Vmbc$, and let $\mathsf{G}' := (\mV',m',b',c')$ be another graph. Form a new graph $\widetilde{\mathsf{G}} := (\widetilde{\mV},\widetilde{m},\widetilde{b},\widetilde{c})$ by letting $\widetilde{\mV}:=\mV\sqcup \mV'$ with $\tilde m\vert_{\mV}=m$, $\tilde m\vert_{\mV'}=m'$,
 and considering arbitrary $\tilde{b}:\widetilde{\mV}\times \widetilde{\mV}\to [0,\infty)$, $\widetilde{c}:\widetilde{\mV}\to [0,\infty)$ in such a way that 
\begin{itemize}
\item $\widetilde{b}\vert_{\mV\times \mV}=b$, \item $\tilde{b}\vert_{\mV'\times \mV'}=b'$, 
\item $\widetilde{c}\vert_{\mV}=c$, 
\item $\widetilde{c}\vert_{\mV'}=c'$,
\item $\widetilde{b}(\mv,\mv') = 0$ for all $\mv \notin \mV^\dagger$ and all $\mv' \in \mV'$.
\end{itemize}
We then say that $\widetilde{\mathsf{G}}$ is formed by \emph{inserting $\mathsf{G}'$ into $\mathsf{G}$ at $\mV^\dagger$}.
\end{defi}
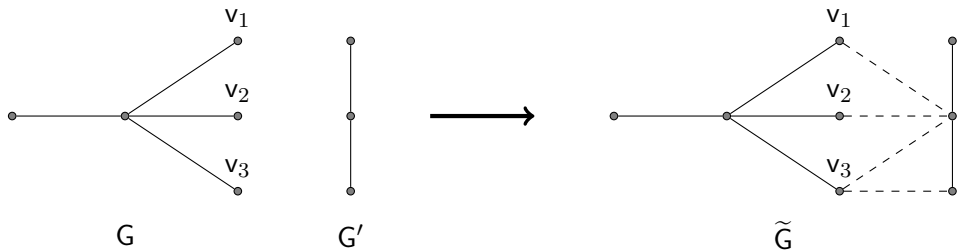
\begin{figure}[h]
\begin{tikzpicture}[scale=0.50]
      \tikzset{enclosed/.style={draw, circle, inner sep=0pt, minimum size=.10cm, fill=gray}, every loop/.style={}}

      \node[enclosed] (Z) at (-1,0) {};
      \node[enclosed, label={above: $\mv_1$}] (A) at (2,2) {};
      \node[enclosed] (A') at (5,2) {};
      \node[enclosed] (A'') at (5,0) {};
      \node[enclosed] (A''') at (5,-2) {};
      \node[enclosed] (B) at (-4,0) {};
      \node[enclosed, label={above: $\mv_3$}] (C) at (2,-2) {};
      \node[enclosed, label={above: $\mv_2$}] (D) at (2,0) {};
      \node[enclosed, white] (E) at (7,0) {};
      \node[enclosed, white] (F) at (10,0) {};
      
      \node[enclosed] (Z') at (15,0) {};
      \node[enclosed, label={above: $\mv_1$}] (X) at (18,2) {};
      \node[enclosed] (X') at (21,2) {};
      \node[enclosed] (B') at (12,0) {};
      \node[enclosed, white, label={below: $\mathsf{G}$}] (G) at (-1,-2.5) {};
       \node[enclosed, white, label={below: $\mathsf{G}'$}] (G') at (5,-2.5) {};
      \node[enclosed, white, label={below: $\widetilde{\mathsf G}$}] (H) at (16.5,-2.4) {};
      \node[enclosed, label={above: $\mv_3$}] (Y) at (18,-2) {};
      \node[enclosed] (Y') at (21,-2) {};
      \node[enclosed, label={above: $\mv_2$}] (W) at (18,0) {};
      \node[enclosed] (W') at (21,0) {};

      \draw (Z) edge node[above] {} (A) node[midway, above] (edge1) {};
      \draw (A') edge node[above] {} (A'') node[midway, above] (edge1rightgraph) {};
      \draw (A'') edge node[above] {} (A''') node[midway, above] (edge2rightgraph) {};
      \draw (Z) edge node[above] {} (B) node[midway, above] (edge2) {};
      \draw (Z) edge node[above] {} (C) node[midway, above] (edge3) {};
      \draw (Z) edge node[above] {} (D) node[midway, above] (edge4) {};
      \draw[->, ultra thick] (E) edge node[above] {} (F) node[midway, above] (edge5) {};
      
      \draw (Z') edge node[above] {} (X) node[midway, above] (edge1) {};
      \draw[dashed] (X) edge node[above] {} (W') node[midway, above] (edge1') {};
      \draw (Z') edge node[above] {} (B') node[midway, above] (edge2) {};
      \draw (Z') edge node[above] {} (Y) node[midway, above] (edge3) {};
      \draw[dashed] (Y) edge node[above] {} (Y') node[midway, above] (edge3') {};
      \draw[dashed] (Y) edge node[above] {} (W') node[midway, above] (edge3') {};
      \draw (Z') edge node[above] {} (W) node[midway, above] (edge4) {};
      \draw[dashed] (W) edge node[above] {} (W') node[midway, above] (edge4) {};
      \draw (X') edge node[above] {} (W') node[midway, above] (edge4) {};
     \draw (W') edge node[above] {} (Y') node[midway, above] (edge4) {};
     \end{tikzpicture}
     \caption{A graph $\widetilde{\mathsf{G}}$ which is formed by inserting a graph $\mathsf{G}'$ into another graph $\mathsf{G}$ at $\mV^\dagger := \{\mv_1,\mv_2,\mv_3 \}$.}
     \end{figure}
     The following is close in spirit \cite[Proposition~4.1(3)]{MugPlu23} on metric graphs.

\begin{prop}\label{cor:monotonocitiy-torsion-subgraphs}
 Let $\mG=\Vmbc$ and $\mG'= (\mV',m',b',c' \equiv 0)$ be two graphs and consider the graph $\tilde{\mG}= (\widetilde{\mV},\widetilde{m},\widetilde{b},\widetilde{c})$ formed by inserting $\mG'$ into $\mG$ at some subset $\mV^\dagger \subset \mV$. If \autoref{ass:compact-embedding-l1} is satisfied for $\mG$, then the following assertions hold. 
\begin{enumerate}[(i)]
\item If $\sup_{\mv \in \mV} c(\mv) > 0$ and 
 $\tauneubc$ is constant over $\mV^\dagger$, then
\[
T_p(\mG) \leq T_p(\widetilde{\mG} ).
\]
\item Else if $\mV_0 \neq \emptyset$ such that $\mV_0 \cap \mV^\dagger = \emptyset$ and
$\taudirbc$ is constant on $\mV^\dagger$, then 
\[
T_p(\mG;\mV_0) \leq T_p(\widetilde{\mG};\mV_0 ).
\]
\end{enumerate}
\end{prop}

\begin{proof}
Let $\tauneubc \in \wneubc (\mV,m)$ be a maximizer of the Pólya quotient $\Polneubc(u)$. Let $\mv^\dagger \in \mV^\dagger$ and define $\hneubc$ on $\widetilde{\mV} = \mV \sqcup \mV'$ via
\[
\hneubc(\mv) := \begin{cases}
\tauneubc(\mv), & \text{if} \:\: \mv\in\mV, \\ 
\tauneubc(\mv^\dagger), &\text{if} \:\:  \mv\in \mV'.
\end{cases} 
\]
Then clearly $\hneubc \in w_{p,\widetilde{b},\widetilde{c}}^{1,p,2}(\widetilde{\mV},m)$ is a valid test function for $\Polneubct$ and (noting that $\tauneubct$ is constant on $\mV^\dagger$ and that $c' \equiv 0$), we estimate
\begin{align*}
T_p(\widetilde{\mG}) &\geq \Polneubct(\hneubc) = \frac{\Vert \hneubc \Vert_{\ell^1(\mV,m)}^p}{p\Qneubct(\hneubc)} \\&= \frac{\Big(\sum\limits_{\mv\in\mV} \vert \tauneubc(\mv) \vert m(\mv) + \sum\limits_{\mv\in \mV'} \vert \tauneubc(\mv^\dagger) \vert m'(\mv) \Big)^p}{\frac{1}{2} \sum\limits_{\mv,\mw \in \mV} b(\mv,\mw)\vert \tauneubc(\mv) - \tauneubc(\mw) \vert^p + \sum\limits_{\mv\in\mV} c(\mv)|\tauneubc(\mv)|^p} \\&\geq \frac{\Big(\sum\limits_{\mv\in\mV} \vert \tauneubc(\mv) \vert m(\mv) \Big)^p}{\frac{1}{2} \sum\limits_{\mv,\mw \in \mV} b(\mv,\mw)\vert \tauneubc(\mv) - \tauneubc(\mw) \vert^p + \sum\limits_{\mv\in\mV} c(\mv)|\tauneubc(\mv)|^p} \\&= \frac{\Vert \tauneubc \Vert_{\ell^1(\mV,m)}^p}{p\Qneubc(\tauneubc)} = \Polneubc(\tauneubc) = T_p(\mG),
\end{align*}
as $\widetilde{b}(\mv,\mw) > 0$ on $\mV \times \mV'$ (resp.,\ on $\mV' \times \mV$) is, by construction of the graph $\widetilde{\mG}$, only possible if $\mv \in \mV^\dagger$ (or resp.,\ $\mw \in \mV^\dagger$). In a similar manner, one can show that $T_p(\mG;\mV_0) \leq T(\widetilde{\mG};\mV_0)$. 
\end{proof}

Note that the choice of taking $\mV^\dagger$ as a singleton is allowed and leads to $\mG=\Vmbc$ being a tree.

For the sake of completeness, let us also note a simple but useful scaling relation. 

\begin{lemma}\label{lem:scaling}
Let $\mG=\Vmbc$ be a
graph and let $\lambda, \mu> 0$. Then the $p$-torsional rigidity and the bottom of the $p$-spectrum are given by
\[
T_p\big(\mG_{\mu,\lambda} \big) = \frac{\mu^p}{\lambda} T_p(\mG) \quad \bigg(\text{resp., $T_p\big(\mG_{\mu,\lambda};\mV_0 \big) = \frac{\mu^p}{\lambda} T_p(\mG;\mV_0)$} \text{ if }\mV_0 \neq \emptyset \bigg),
\]
and
\[
\lambda_{0,p} \big(\mG_{\mu,\lambda} \big) = \frac{\lambda}{\mu}\lambda_{0,p}(\mG) \quad \bigg(\text{resp., $\lambda_0 \big(\mG_{\mu,\lambda};\mV_0 \big) = \frac{\lambda}{\mu}\lambda_{0,p}(\mG;\mV_0)$} \text{ if }\mV_0 \neq \emptyset \bigg),
\]
where $\mG_{\mu,\lambda} := (\mV,\mu m,\lambda b, \lambda c)$.

If, additionally, \autoref{ass:compact-embedding-l1} holds, then $\mG_{\mu,\lambda}$ is $q-p$-torsional admissible, too, with $p$-torsion function given by
\[
\tauneubcmulambda = \bigg(\frac{\mu}{\lambda} \bigg)^\frac{1}{p-1} \tauneubc \quad \bigg(\text{resp., $\taudirbcmulambda =\bigg( \frac{\mu}{\lambda} \bigg)^{\frac{1}{p-1}} \taudirbc$}\bigg).
\]
    \end{lemma}
    
\begin{proof}
Using the variational characterization for the $p$-torsional rigidity in \eqref{eq:tors-def-1}, we observe that (taking into account that in this case $\wneubc(\mV,m) = w_{\lambda b,\lambda c}^{1,p,2}(\mV,\mu m)$)
\begin{align*}
T_p(\mG_{\mu,\lambda}) = \sup_{u \in \wneubc(\mV,m)} \frac{\Vert u \Vert_{\ell^1(\mV,\mu m)}^p}{p\mathcal{Q}_{p}^{\mG_{\mu, \lambda}}(u)} = \frac{\mu^p}{\lambda} \sup_{u \in \wneubc(\mV,m)} \frac{\Vert u \Vert_{\ell^1(\mV,m)}^p}{p\Qneubc(u)} = \frac{\mu^p}{\lambda} T_p(\mG),
\end{align*}
because $\Vert u \Vert_{\ell^1(\mV,\mu m)} = \mu \Vert u \Vert_{\ell^1(\mV,m)}$ and $\mathcal{Q}_{p}^{\mG_{\lambda, \mu}}(u) = \lambda \Qneubc(u)$ for each $u \in \wneubc(\mV,m)$; and likewise in the case of Dirichlet conditions. Furthermore, using the variational characterization  now for the bottom of the $p$-spectrum represented in \eqref{eq:variational-characterization-ground-state}, the scaling property can be verified in the same way as for the $p$-torsional rigidity.

 Moreover, by definition of the $p$-torsion function in terms of a weak solution, see \eqref{eq:discr-ellipt-1-p}, $\tauneubcmulambda = (\frac{\mu}{\lambda})^\frac{1}{p-1} \tauneubc$ follows immediately; same holds for the Dirichlet case (and \eqref{eq:discr-ellipt-1-p-dir}).
\end{proof}

\subsection{Two examples}
We know from \autoref{cor:wellp} that, in particular, graphs that support an $\ell^p-\ell^p$-Poincaré inequality have a unique $p$-torsion function. Let us present two cases  where an explicit formula for this function  can be deduced: possibly infinite paths and finite stars.

\begin{prop}\label{prop:p-torsion-function-pat}
Let $\mP=(\mV,m,b,c\equiv 0)$ be a path graph whose vertex set is, without loss of generality, $\mV=\{\mv_j:j\in J\}$ with 
\[
\hbox{either }J=\{0,1,\ldots,n-1\}, \qquad\hbox{ or }J=\N_0.
\]
If Dirichlet conditions are imposed at $\mV_0=\{\mv_0\}$ and \autoref{assum:finite-meas} is satisfied, 
then
\begin{align}\label{eq:p-torsion-function-path-left-dirichlet}
\udirbpath(\mv_j)
=
\begin{cases}
0, &j=0,\\
  \sum\limits_{\ell=1}^j \bigg(\frac{1}{b(\mv_{\ell-1}, \mv_{\ell})}  \sum\limits_{\substack{k \in J\\ k>\ell-1 }} m(\mv_k) \bigg)^\frac{1}{p-1},\qquad &j\ne 0.
  \end{cases}
\end{align}
defines a pointwise solution for the equation $L_p^\mG f = \mathbf{1}$. Moreover, $\mP$ is $1-p$-torsional-admissible if and only if $\udirbpath \in \wdirb(\mV,m)$. In this case, the unique $p$-torsion function is given by $\taudirbpath = \udirbpath$.

In the special case of $p=2$, $b=b_{\mathrm{st}}$ and $\mV\setminus \mV_0$ consisting of $n-1$ vertices, 
\begin{equation}\label{eq:torsional-rigidity-path-m1}
T_2(\mP;\mV_0)=\frac{(n-1)n(2n-1)}{6}\qquad\hbox{for }m=\mathbf{1}
\end{equation}
and
\begin{equation}\label{eq:torsional-rigidity-path-mdeg}
T_2(\mP;\mV_0)=\frac{(n-1)(2n-3)(2n-1)}{3}\qquad\hbox{for }m=\deg.
\end{equation}
\end{prop}

More generally, in the non-linear case of $p \neq 2$ it follows from \eqref{eq:p-torsion-function-path-left-dirichlet} that if  $\mV\setminus \mV_0$ consists of $n-1$ vertices then for $m\equiv \mathbf{1}$ and standard edge weights $b=b_{\mathrm{st}}$, the $p$-torsional rigidity is
\begin{equation}\label{eq:torsional-rigid-infty}
\|\taudirbpath\|_{\ell^1(\mV\setminus \mV_0,m)}^{p-1} =\bigg(\sum_{j=1}^{n-1} \sum_{\ell=1}^{j}(n-\ell)^\frac{1}{p-1}\bigg)^{p-1}\quad \hbox{with}\quad \|\taudirbpath\|_{\ell^\infty(\mV \setminus \mV_0)} =\sum_{\ell=1}^{n-1} (n-\ell)^\frac{1}{p-1}.
\end{equation}
We should emphasize that the assumption of $1-p$-torsional-admissibility of $\mP$ is automatically satisfied whenever $\wdirbc(\mV,m)$ is compactly embedded in $\ell^p(\mV,m)$, and in particular if $\mP$ is a finite path graph, as we already saw in \autoref{cor:compact-admiss}. 
A geometric condition for $1-p$-torsional-admissibility of an infinite path graph can also be deduced from \autoref{prop:estimate-inradius} below.
\begin{proof}
It is an immediate observation that the function $\udirbpath$ defined in \eqref{eq:p-torsion-function-path-left-dirichlet} yields a pointwise solution for $L_p^\mG f=\mathbf{1}$ in both the finite and the infinite case. To prove 
the remaining part of \autoref{prop:p-torsion-function-pat} we distinguish between the two cases where $J$ is finite or infinite.

\emph{Case 1:} Starting with the finite case, we first note that in this case, the equivalence holds trivially and that $\udirbpath$ indeed yields a $p$-torsion function; thus we only have to deduce the representation \eqref{eq:p-torsion-function-path-left-dirichlet} for any arbitrary weak $\ell^1$-solution, to show the uniqueness: to this end, we will argue inductively on the possible number $n  \in \mathbb{N}$ of vertices. 

Let us first consider the case of two vertices $\mv_0,\mv_1$, with Dirichlet conditions imposed on $\mV_0:=\{\mv_0\}$: of course, $\taudirbpath(\mv_0)=0$, and furthermore
\[
\frac{b(\mv_0,\mv_1)}{m(\mv_1)} \vert \nabla_{\mv_1,\mv_0}\taudirbpath \vert^{p-2} \nabla_{\mv_1,\mv_0} \taudirbpath= \frac{b(\mv_0,\mv_1)}{m(\mv_1)} \vert \taudirbpath(\mv_1) \vert^{p-2} \taudirbpath(\mv_1)  = 1.
\]
In particular, $\taudirbpath(\mv_1)>0$ and  we deduce that
\begin{align}\label{eq:p-torsion-path-two-vertices}
\taudirbpath(\mv_0) = 0, \quad \text{and} \quad \taudirbpath(\mv_1) = \bigg(\frac{m(\mv_1)}{b(\mv_0,\mv_1)}\bigg)^\frac{1}{p-1}:
\end{align}
this shows the assertion for $n=2$.

Let now $n \geq 3$. Then to obtain $\taudirbpath$ one has to solve, in particular,
\begin{align*}
\frac{1}{m(\mv_{n-2})} &\big(b(\mv_{n-3}, \mv_{n-2})\vert \nabla_{\mv_{n-2},\mv_{n-3}}\taudirbpath \vert^{p-2} \nabla_{\mv_{n-2},\mv_{n-3}}\taudirbpath \\ &\qquad +  b(\mv_{n-2}, \mv_{n-1})\vert \nabla_{\mv_{n-2},\mv_{n-1}}\taudirbpath \vert^{p-2} \nabla_{\mv_{n-2},\mv_{n-1}}\taudirbpath \big) = 1,
\end{align*}
and
\begin{align*}
\frac{1}{m(\mv_{n-1})}b(\mv_{n-2},\mv_{n-1}) \vert \nabla_{\mv_{n-1},\mv_{n-2}}\taudirbpath \vert^{p-2}\nabla_{\mv_{n-1},\mv_{n-2}}\taudirbpath = 1
\end{align*}
Plugging the latter into the former equation, we obtain
\begin{align*}
\frac{1}{m(\mv_{n-2})} b(\mv_{n-3}, \mv_{n-2})\vert \nabla_{\mv_{n-2},\mv_{n-3}}\taudirbpath \vert^{p-2} \nabla_{\mv_{n-2},\mv_{n-3}}\taudirbpath = \frac{m(\mv_{n-2}) + m(\mv_{n-1})}{m(\mv_{n-2})}.
\end{align*}
Now, define another point measure $m'$ on $\mV \setminus \{\mv_{n-1} \}$ given by
\[
m'(\mv)= \begin{cases} m(\mv), & \text{if} \:\: \mv \neq \mv_{n-2}, \\ m(\mv_{n-2})+m(\mv_{n-1}), & \text{else}, \end{cases} \qquad \text{for } \mv \in \mV \setminus \{\mv_{n-1} \}.
\] 
Then the system reads like
\begin{align*}
\frac{1}{m'(\mv_j)}&\big( b(\mv_{j-1},\mv_j)\vert \nabla_{\mv_j,\mv_{j-1}}\taudirbpath \vert^{p-2} \nabla_{\mv_j,\mv_{j-1}}\taudirbpath \\ &\qquad +  b(\mv_{j+1},\mv_j)\vert \nabla_{\mv_j,\mv_{j+1}}\taudirbpath \vert^{p-2} \nabla_{\mv_j,\mv_{j+1}}\taudirbpath \big) = 1
\end{align*}
for $j=1,\dots,n-2$ and 
\[
\frac{1}{m'(\mv_{n-2})} b(\mv_{n-3},\mv_{n-2}) \vert \nabla_{\mv_{n-2},\mv_{n-3}} \taudirbpath \vert^{p-2} \nabla_{\mv_{n-2},\mv_{n-3}}\taudirbpath = 1
\]
for the last equation. We have thus reduced our original system in $n-1$ equations and $n-1$ unknowns to a system with only $n-2$ equations and $n-2$ unknowns; solving this smaller system then leads to
\[
\taudirbpath(\mv_0)=0, \quad {and} \quad \taudirbpath(\mv_j) = \sum_{\ell=1}^j \bigg(\frac{1}{b(\mv_{\ell-1}, \mv_{\ell})} \sum_{k=\ell}^{n-2} m'(\mv_k) \bigg)^\frac{1}{p-1}
\]
and since $m'(\mv_{n-2})=m(\mv_{n-2})+m(\mv_{n-1})$ by definition, this yields
\[
\taudirbpath(\mv_j) = \sum_{\ell=1}^j \bigg(\frac{1}{b(\mv_{\ell-1}, \mv_{\ell})} \sum_{k=\ell}^{n-1} m(\mv_k) \bigg)^\frac{1}{p-1}
\]
for $j=1,\dots,n-2$. Looking once again at the equation 
\[
\frac{1}{m(\mv_{n-1})}b(\mv_{n-2},\mv_{n-1}) \vert \nabla_{\mv_{n-1},\mv_{n-2}}\taudirbpath  \vert^{p-2}\nabla_{\mv_{n-1},\mv_{n-2}}\taudirbpath =1
\]
we conclude (as in the case for $n=2$) that $\taudirbpath(\mv_{n-1}) > \taudirbpath(\mv_{n-2})$, because otherwise the left-hand side would be non-positive and this finally leads to
\begin{align*}
\taudirbpath(\mv_{n-1}) &= \bigg(\frac{m(\mv_{n-1})}{b(\mv_{n-2},\mv_{n-1})}\bigg)^\frac{1}{p-1} + \taudirbpath(\mv_{n-2}) \\&= \bigg(\frac{m(\mv_{n-1})}{b(\mv_{n-2},\mv_{n-1})}\bigg)^\frac{1}{p-1} + \sum_{\ell=1}^{n-2} \bigg( \frac{1}{b(\mv_{\ell-1}, \mv_{\ell})} \sum_{k=\ell}^{n-1} m(\mv_k) \bigg)^\frac{1}{p-1} \\&= \sum_{\ell=1}^{n-1} \bigg( \frac{1}{b(\mv_{\ell-1}, \mv_{\ell})} \sum_{k=\ell}^{n-1} m(\mv_k) \bigg)^\frac{1}{p-1},
\end{align*}
which eventually yields 
$\taudirbpath = \udirbpath$ in the finite case.

\emph{Case 2:} Next, suppose that $J$ is infinite, in other words $J = \N_0$: it is immediate that $\taudirb$ yields a $p$-torsion function whenever it belongs to $\wdirb(\mv,m)$. To prove the other direction, we are going to reduce this to the finite case. Let $j \in \N$. As $\mathsf{P} = \Vmb$ is $1-p$-torsional-admissible, there exists a weak $\ell^1$-solution $\taudirbpath$ of $\Ldirbpath u = \mathbf{1}$ on $\mV \setminus \mV_0$, 
in particular, considering the function $h_j \in \wdirb(\mV,m)$ given by
\begin{align*}
h_j(\mv_i) := \begin{cases}
0, & \text{if $i=0,\dots,j-1$,} \\ 1, & \text{else,}
\end{cases} \qquad i \in \mathbb{N}_0, 
\end{align*}
one has
\begin{align*}
\sum_{k=1}^\infty  h_j(\mv_k)m(\mv_k) &= \sum_{k=1}^\infty b(\mv_{k-1}, \mv_{k})\vert \nabla_{\mv_{k},\mv_{k-1}}\taudirbpath \vert^{p-2} \nabla_{\mv_{k},\mv_{k-1}}\taudirbpath  \nabla_{\mv_{k},\mv_{k-1}}h_j 
\end{align*}
and consequently
\begin{align}
b(\mv_{j-1}, \mv_{j})\vert \nabla_{\mv_{j},\mv_{j-1}}\taudirbpath \vert^{p-2} \nabla_{\mv_{j},\mv_{j-1}}\taudirbpath = \sum_{k=j}^\infty m(\mv_k).
\end{align}
Moreover,
\begin{align}
\begin{aligned}
\frac{1}{m(\mv_{i})} &\big(b(\mv_{i-1}, \mv_{i})\vert \nabla_{ \mv_{i},\mv_{i-1}}\taudirbpath \vert^{p-2} \nabla_{ \mv_{i},\mv_{i-1}}\taudirbpath \\ &\qquad +  b(\mv_{i}, \mv_{i+1})\vert \nabla_{\mv_{i},\mv_{i+1}}\taudirbpath \vert^{p-2} \nabla_{\mv_{i},\mv_{i+1}}\taudirbpath \big) = 1,
\end{aligned}
\end{align}
holds for all $i=1,\dots,j-1$. Thus, defining $m'$ on $\mV' := \{\mv_0,\dots,\mv_j\}$ by
\[
m'(\mv_i)= \begin{cases} m(\mv_i), & \text{if} \:\: i \neq j, \\ \sum_{k=j}^\infty m(\mv_k), & \text{else}, \end{cases} \qquad \text{for } i =0,1,\ldots,j,
\] 
one reaches at a \emph{finite} path graph $\mP' = (\mV',b',m',c\equiv 0)$ with $b' := b\vert_{\mV' \times \mV'}$. Therefore, according to the first case, we observe that
\begin{align*}
\taudirbpath(\mv_j) &= \sum_{\ell=1}^j \bigg( \frac{1}{b'(\mv_{\ell-1}, \mv_\ell)} \sum_{k=\ell}^j m'(\mv_k) \bigg)^\frac{1}{p-1} \\&= \sum_{\ell=1}^j \bigg( \frac{1}{b(\mv_{\ell-1}, \mv_\ell)} \bigg[\sum_{k=\ell}^{j-1} m(\mv_k) + \sum_{k=j}^\infty m(\mv_k) \bigg] \bigg)^\frac{1}{p-1} \\&= \sum_{\ell=1}^j \bigg( \frac{1}{b(\mv_{\ell-1}, \mv_\ell)} \sum_{k=\ell}^\infty m(\mv_k) \bigg)^\frac{1}{p-1} =\udirbpath(\mv_j)
\end{align*}
which completes the proof 
in the case of infinite graphs as $\taudirbpath \in \wdirb(\mV,m)$. (Note that this also yields uniqueness of the $p$-torsion function, as any (weak) solution is, by the same proof, necessarily of the form given in \eqref{eq:p-torsion-function-path-left-dirichlet}.) 

To conclude, let us focus on the linear case of $p=2$, for a finite graph with $b=b_{\mathrm{st}}$.
If $m\equiv \mathbf{1}$, then applying~\eqref{eq:p-torsion-function-path-left-dirichlet} we find for the $p$-torsion function $\taudirbpathlin$ of $\mP$ that
\begin{equation*}
\begin{split}
\taudirbpathlin(\mv_j) &= \sum_{\ell=1}^j \sum_{k=\ell}^{n-1} 1 = \sum_{k = 1}^j (n-k) = nj - \sum_{\ell=1}^j \ell \\
&= nj-\frac{j(j+1)}{2} = j\bigg( n-\frac{j+1}{2}\bigg),
\end{split}
\qquad
\hbox{for all }j=1,\dots,n-1.
\end{equation*}
This shows { by \autoref{prop:variational-char}} that
\begin{align}\label{eq:formula-torsional-rigidity-stadard-and-m-equiv-1}
\begin{aligned}
    T_2(\mP;\mV_0) &= \sum_{j=1}^{n-1} \taudirbpathlin(\mv_j) = \sum_{j=1}^{n-1} j\bigg(n-\frac{j+1}{2}\bigg) \\
    &= n \sum_{j=1}^{n-1} j - \frac{1}{2}\sum_{j=1}^{n-1} j(j+1) \\&= \frac{(n-1)n(2n-1)}{6}.
\end{aligned}
\end{align}  
Likewise, for $m=\deg$
\begin{equation*}
\begin{split}
\taudirbpathlin(\mv_j) &= \sum_{\ell=1}^j \sum_{k=\ell}^{n-1} \deg(\mv_k) = \sum_{k = 1}^j \big(2(n-1-k)+1\big) \\
&= (2n-1)j - j(j+1)\\
&= j\Big( 2n-2-j\Big),
\end{split}
\qquad
\hbox{for all }j=1,\dots,n-1.
\end{equation*}
Therefore,
\begin{equation}\label{eq:formula-torsional-rigidity-stadard-and-m=deg}
\begin{split}
T_2(\mP;\mV_0)&=
\sum_{j=1}^{n-1} \taudirbpathlin(\mv_j)\deg(\mv_j) = 2\sum_{j=1}^{n-2} j\Big(2n-2-j\Big) + (n-1)^2 \\
&= 2\bigg((n-1)^2 (n-2) -\frac{(n-1)(n-2)(2n-3)}{6}\bigg) + (n-1)^2\\
&= \frac{(n-1)(n-2)(4n-3)}{3} + (n-1)^2 \\&= \frac{(n-1)(4n^2-8n+3)}{3} = \frac{(n-1)(2n-1)(2n-3)}{3}.
\end{split}
\end{equation}
This concludes the proof. 
 \end{proof}

\begin{rem}
(i) If one imposes a Dirichlet condition at \emph{both} ends of a finite path graph $\mP=\Vmb$, it seems to be more difficult to derive a general formula for the $p$-torsion function $\taudirbpath$, unless additional symmetry assumptions on the weights $b$ and the underlying discrete measure $m$ are imposed and the underlying path graph has an \emph{odd} number of vertices. Namely, if $n=2k+1$ for some $k \in \mathbb{N}$ and
\begin{align*}
m(\mv_\ell) = m(\mv_{2k-\ell}), \quad b(\mv_{\ell-1},\mv_\ell) = b(\mv_{2k-\ell}, \mv_{2k-\ell+1}) \quad \text{for $\ell=1,\dots,k$},
\end{align*}
then the $p$-torsion function $\taudirbpath$ is given by $\taudirbpath(\mv_0) = \taudirbpath(\mv_{n-1}) = 0$ and
\begin{align}\label{eq:p-torsion-two-dirichlet-odd}
 \taudirbpath(\mv_j) = \sum_{\ell=1}^j \bigg(\frac{1}{b(\mv_{\ell-1},\mv_\ell)} \bigg[\sum_{h=\ell}^{k-1} m(\mv_h) + \frac{m(\mv_k)}{2} \bigg] \bigg)^\frac{1}{p-1} = \taudirbpath(\mv_{2k-j}), \quad \text{for $j=1,\dots,k$.}
\end{align}
Indeed, if $k \geq 2$, writing down the equations for $\taudirbpath(\mv_{k-1})$, $\taudirbpath(\mv_k)$ and $\taudirbpath(\mv_{k+1})$ and using the symmetry of the $p$-torsion function, one reaches at the equations
\begin{align*}
\frac{1}{m(\mv_{k-1})} &\big(b(\mv_{k-2}, \mv_{k-1})\vert \nabla_{\mv_{k-1},\mv_{k-2}}\taudirbpath \vert^{p-2} \nabla_{\mv_{k-1},\mv_{k-2}}\taudirbpath \\ &\qquad +  b(\mv_{k-1}, \mv_{k})\vert \nabla_{\mv_{k-1},\mv_{k}}\taudirbpath \vert^{p-2} \nabla_{\mv_{k-1},\mv_{k}}\taudirbpath \big) = 1,
\end{align*}
and
\begin{align*}
\frac{2}{m(\mv_{k})}b(\mv_{k-1},\mv_{k-2}) \vert \nabla_{\mv_{k},\mv_{k-1}}\taudirbpath \vert^{p-2}\nabla_{\mv_{k},\mv_{k-1}}\taudirbpath = 1.
\end{align*}
Thus, taking only the first $k+1$ vertices, we reach at a path graph having only one Dirichlet end, unchanged weights and point measures with the difference that the last point measure $m(\mv_k)$ is divided by $2$. Therefore, \autoref{prop:p-torsion-function-pat} yields \eqref{eq:p-torsion-two-dirichlet-odd}. Moreover, one easily verifies that \eqref{eq:p-torsion-two-dirichlet-odd} also holds for the case where $k=1$.

(ii) We emphasize that uniqueness of the pointwise solution $\udirbpath$ represented in \eqref{eq:p-torsion-function-path-left-dirichlet} does not hold in general:
indeed, in the infinite case, all functions $f_\gamma:\N \to  \R$ defined by
\begin{align*}
f_\gamma(\mv_j) 
=
\begin{cases}
0, &j=0,\\
  \sum\limits_{\ell=1}^j \bigg(\frac{1}{b(\mv_{\ell-1}, \mv_{\ell})}  \bigg(\sum\limits_{\substack{k \in J\\ k>\ell-1 }} m(\mv_k) + \gamma\bigg) \bigg)^\frac{1}{p-1},\qquad &j\ne 0,
  \end{cases}
\end{align*}
are also pointwise solutions of $L_p^\mG f = \mathbf{1}$ for any $\gamma\ge 0$. However, the strict convexity of $\mathfrak{F}_p^\mP$ on $\wdirb(\mV,m)$, see \autoref{lem:fp-convex}, implies the uniqueness of a weak solution of class $\wdirb(\mV,m)$, and indeed the only such $f_\gamma$ belonging to $\wdirb(\mV,m)$ is $\udirbpath$ in \autoref{prop:p-torsion-function-pat}.

\end{rem}
The arguments used in the proof \autoref{prop:p-torsion-function-pat} can be used to derive formulas for the $p$-torsion function for other graphs, too: a very simple example is given by star graphs.
\begin{prop}\label{prop:formula-n-star}
Let $\mathsf{S}:=\Vmb$ be a star graph on $n+1$ vertices
$\mV = \{ \mv_0,\mv_1,\dots,\mv_n \}$ with Dirichlet condition at $\mV_0 = \{\mv_0 \}$ and central vertex $\mv_1$.
Then the $p$-torsion function $\taudirbstar$ is given by
\begin{equation}\label{eq:formula-torsion-star}
\taudirbstar(\mv_j) =
\begin{cases}
0, &j=0,\\
 \bigg(\frac{\sum_{k=1}^{n} m(\mv_k)}{b(\mv_1,\mv_0)} \bigg)^\frac{1}{p-1}, &j=1,\\
  \bigg(\frac{\sum_{k=1}^{n} m(\mv_k)}{b(\mv_1,\mv_0)} \bigg)^\frac{1}{p-1} + \bigg(\frac{m(\mv_j)}{b(\mv_1,\mv_j)} \bigg)^\frac{1}{p-1},\quad &j\in\{2,\ldots,n\}.
  \end{cases}
\end{equation}

If, in particular, standard edge weights $b=b_{\mathrm{st}}$ are considered and $p=2$, then
\begin{equation}\label{eq:tors-star-1}
T_2(\mathsf{S};\mV_0)=n^2+n-1\qquad \hbox{for }m=\mathbf{1} 
\end{equation}
and
\begin{equation}\label{eq:tors-star-deg}
T_2(\mathsf{S};\mV_0)=4n^2-n\qquad \hbox{for }m=\deg. 
\end{equation}
\end{prop}

\begin{figure}[h]
\begin{tikzpicture}[scale=0.50]
      \tikzset{enclosed/.style={draw, circle, inner sep=0pt, minimum size=.10cm, fill=gray}, every loop/.style={}}

      \node[enclosed, label={below right: $\:\; \mv_1$}] (Z) at (0,4) {};
      \node[enclosed, label={above: $\mv_5$}] (A) at (0,6.5) {};
      \node[enclosed, fill=white, label={left: $\mv_0$}] (B) at (-2.5,4.75) {};
      \node[enclosed, label={right: $\mv_4$}] (C) at (2.5,4.75) {};
      \node[enclosed, label={left: $\mv_2$}] (D) at (-1.5,2) {};
      \node[enclosed, label={right: $\mv_3$}] (E) at (1.5,2) {};

      \draw (Z) edge node[above] {} (A) node[midway, above] (edge1) {};
      \draw (Z) edge node[above] {} (B) node[midway, above] (edge2) {};
      \draw (Z) edge node[above] {} (C) node[midway, above] (edge3) {};
      \draw (Z) edge node[above] {} (D) node[midway, above] (edge4) {};
      \draw (Z) edge node[above] {} (E) node[midway, above] (edge5) {};
     \end{tikzpicture}
     \vspace{-1cm}
     \caption{A $5$-star with center $\mv_1$ and a Dirichlet condition at $\mV_0 = \{ \mv_0 \}$.}
     \end{figure}
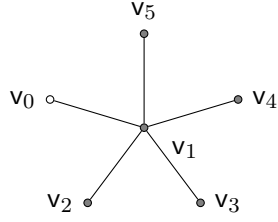
     
\begin{proof}
As the $p$-torsion function $\taudirbstar$ should satisfy the equation $\Ldirbstar \taudirbstar = \mathbf{1}$, we deduce
\begin{align}
\frac{1}{m(\mv_j)} b(\mv_j, \mv_1) \vert \nabla_{\mv_j,\mv_1} \taudirbstar \vert^{p-2} \nabla_{\mv_j,\mv_1} \taudirbstar &= 1,\quad  j \in \{2,\dots,n \},\label{eq:equation-star-graphs}
\\ 
\frac{1}{m(\mv_1)} \sum_{k=0}^{n} b(\mv_1,\mv_k) \vert \nabla_{\mv_1,\mv_k} \taudirbstar \vert^{p-2} \nabla_{\mv_1,\mv_k} \taudirbstar &= 1.
\label{eq:equation-star-graphs-2}
\end{align}
By plugging the sum of all equations in \eqref{eq:equation-star-graphs} into \eqref{eq:equation-star-graphs-2} we obtain
where the last equation can be written as
\begin{align*}
-\sum_{k=2}^{n} m(\mv_k) + b(\mv_1,\mv_0) \vert \taudirbstar(\mv_1) \vert^{p-2}\taudirbstar(\mv_1) = m(\mv_1),´
\end{align*}
or equivalently
\[
\taudirbstar(\mv_1)^{p-1} = \vert \taudirbstar(\mv_1) \vert^{p-2}\taudirbstar(\mv_1) = \frac{\sum_{k=1}^{n} m(\mv_k)}{b(\mv_1,\mv_0)},
\]
where the first equality holds by positivity of the $p$-torsion function. Thus, taking the $(p-1)$-th root on both sides, this yields the expression for $\taudirbstar(\mv_1)$. Now, plugging this back into \eqref{eq:equation-star-graphs}, we find
\begin{align*}
\Bigg\vert \taudirbstar(\mv_j) - \bigg(\frac{\sum_{k=1}^{n} m(\mv_k)}{b(\mv_1,\mv_0)} \bigg)^\frac{1}{p-1} \Bigg\vert^{p-2} \Bigg( \taudirbstar(\mv_j) - \bigg(\frac{\sum_{k=1}^{n} m(\mv_k)}{b(\mv_1,\mv_0)} \bigg)^\frac{1}{p-1} \Bigg) = \frac{m(\mv_j)}{b(\mv_j,\mv_1)}
\end{align*}
for all $j \in \{2,\dots,n \}$; and because the right-hand side is positive, it follows that 
\[
\taudirbstar(\mv_j) > \bigg(\frac{\sum_{k=1}^{n} m(\mv_k)}{b(\mv_1,\mv_0)} \bigg)^\frac{1}{p-1}
\]
and rearranging of the equation yields the claimed expression for all $\taudirbstar(\mv_j)$, $j=2,\dots,n$.

Let us finally check the formula for the torsional rigidity in the linear case of $p=2$, $b=b_{\mathrm{st}}$ and $m\equiv \mathbf{1}$. 
Then \eqref{eq:formula-torsion-star} reads
\[
\taudirbstarlin(\mv)=\begin{cases}
0, \quad & \mv=\mv_0,\\
n, & \mv=\mv_1,\\
n+1, &\mv=\mv_2,\ldots,\mv_n.
\end{cases}
\]
Accordingly, the torsional rigidity is given by
\[
T_2(\mathsf{S};\mV_0) = \|\taudirbstarlin\|_{\ell^1(\mV\setminus \mV_0,m)}=n+(n-1)(n+1),
\]
thus showing~\eqref{eq:tors-star-1}; \eqref{eq:tors-star-deg} can be proven likewise, using the Handshake Lemma.
\end{proof}

\begin{rem}\label{rem:reduction-star-graphs}
We can use \autoref{rem:multiple} in combination with \autoref{conv:dirichlet-singleton}, to derive a formula for star graphs having an arbitrary number of outer Dirichlet conditions: starting with a star graph $\mathsf{S} = \Vmb$ on $n+m$, $n,m \in \mathbb{N}$ vertices with $m$ Dirichlet conditions at $\mV_0 := \{ \mv_0^1,\ldots,\mv_0^m \}$ and remaining vertices $\mv_1,\ldots,\mv_n$, where $\mv_1$ is the central vertex and the edges between $\mv_1$ and $\mv_0^1,\ldots,\mv_0^m$ having edge weights
\[
b(\mv_0^i,\mv_1) \qquad \text{for $i=1,\ldots,m$},
\]
we can identify all the Dirichlet vertices $\mv_0^1,\ldots,\mv_0^m$ (without changing the values of the $p$-torsion function, cf.\ \autoref{conv:dirichlet-singleton}) and denote the corresponding new vertex with $\mv_0$ (cf.\ the first transformation in Figure \ref{fig:transformation-mutliple-dirichlet-vertices}). In this way, we reach at a new graph (which is regardless of $m$ at those vertices, since it does not change the torsion function at all) with new vertex set $\mV' = \{\mv_0,\mv_1,\dots,\mv_n\}$ and a singleton $\mV_0' = \{ \mv_0 \}$, where we have $m$ multiple edges $\me_{\mv_0,\mv_1}^1, \ldots, \me_{\mv_0,\mv_1}^m$ having the same edge weights as before. As described in \autoref{rem:multiple}, we can replace those edges by just a single edge between $\mv_0,\mv_1$, and considering a new edge weight 
$b(\mv_0,\mv_1) = \sum_{i=1}^m b(\mv_0^i,\mv_1)$ (cf.\ the second transformation in Figure \ref{fig:transformation-mutliple-dirichlet-vertices}). In this way, we obtain a star graph on $n+1$ vertices with exactly one Dirichlet condition at $\mv_0$. Thus, according to \autoref{prop:formula-n-star}, we obtain for the torsion function
\begin{align*}
\taudirbstar(\mv_1) = \bigg(\frac{\sum_{k=1}^{n} m(\mv_k)}{b(\mv_1,\mv_0)} \bigg)^\frac{1}{p-1} = \bigg(\frac{\sum_{k=1}^{n} m(\mv_k)}{\sum_{i=1}^m b(\mv_0^i,\mv_1)} \bigg)^\frac{1}{p-1}
\end{align*}
as well as 
\begin{align*}
\taudirbstar(\mv_j) = \bigg(\frac{\sum_{k=1}^{n} m(\mv_k)}{\sum_{i=1}^m b(\mv_0^i,\mv_1)} \bigg)^\frac{1}{p-1} + \bigg(\frac{m(\mv_j)}{b(\mv_1,\mv_j)} \bigg)^\frac{1}{p-1} \qquad \text{for $j=2,\dots,n$.}
\end{align*}
\end{rem}
\vspace{-0.3cm}
\begin{figure}[h]
\begin{tikzpicture}[scale=0.50]
      \tikzset{enclosed/.style={draw, circle, inner sep=0pt, minimum size=.10cm, fill=gray}, every loop/.style={}}

      \node[enclosed, label={below right: $\:\; \mv_1$}] (Z) at (0,4) {};
      \node[enclosed, fill=white, label={above: $\mv_0^3$}] (A) at (0,6.5) {};
      \node[enclosed, fill=white, label={left: $\mv_0^2$}] (B) at (-2.5,4.75) {};
      \node[enclosed, label={right: $\mv_3$}] (C) at (2.5,4.75) {};
      \node[enclosed, fill=white, label={left: $\mv_0^1$}] (D) at (-1.5,2) {};
      \node[enclosed, label={right: $\mv_2$}] (E) at (1.5,2) {};
      \node[enclosed, white] (X) at (4.5,4.375) {};
      \node[enclosed, white] (Y) at (7.5,4.375) {};

      \draw (Z) edge node[above] {} (A) node[midway, above] (edge1) {};
      \draw (Z) edge node[above] {} (B) node[midway, above] (edge2) {};
      \draw (Z) edge node[above] {} (C) node[midway, above] (edge3) {};
      \draw (Z) edge node[above] {} (D) node[midway, above] (edge4) {};
      \draw (Z) edge node[above] {} (E) node[midway, above] (edge5) {};
      \draw[dashed] (D) edge[bend left] node[above] {} (B) node[midway, above] (edge6) {};
      \draw[dashed] (A) edge[bend right] node[above] {} (B) node[midway, above] (edge7) {};
      \draw[->, ultra thick] (X) edge node[above] {} (Y) node[midway, above] (arrow1) {};
      
      \node[enclosed, label={below right: $\:\; \mv_1$}] (Z') at (12,4) {};
      \node[enclosed, fill=white, label={left: $\mv_0$}] (B') at (9.5,4.75) {};
      \node[enclosed, label={right: $\mv_3$}] (C') at (14.5,4.75) {};
      \node[enclosed, label={right: $\mv_2$}] (E') at (13.5,2) {};
	  \node[enclosed, white] (X') at (16.5,4.375) {};
      \node[enclosed, white] (Y') at (19.5,4.375) {};      
      
      \draw (Z') edge node[above] {} (B') node[midway, above] (edge8) {};
      \draw (Z') edge[bend left] node[above] {} (B') node[midway, above] (edge8') {};
      \draw (Z') edge[bend right] node[above] {} (B') node[midway, above] (edge8'') {};
      \draw (Z') edge node[above] {} (C') node[midway, above] (edge9) {};
      \draw (Z') edge node[above] {} (E') node[midway, above] (edge10) {};
      \draw[->, ultra thick] (X') edge node[above] {} (Y') node[midway, above] (arrow2) {};
      
      \node[enclosed, label={below right: $\:\; \mv_1$}] (Z'') at (24,4) {};
      \node[enclosed, fill=white, label={left: $\mv_0$}] (B'') at (21.5,4.75) {};
      \node[enclosed, label={right: $\mv_3$}] (C'') at (26.5,4.75) {};
      \node[enclosed, label={right: $\mv_2$}] (E'') at (25.5,2) {};
      \draw[thick] (Z'') edge node[above] {} (B'') node[midway, above] (edge11) {};
      \draw (Z'') edge node[above] {} (C'') node[midway, above] (edge12) {};
      \draw (Z'') edge node[above] {} (E'') node[midway, above] (edge13) {};
      \draw[->, ultra thick] (X') edge node[above] {} (Y') node[midway, above] (arrow2) {};
     \end{tikzpicture}
     \vspace{-0.5cm}
     \caption{Transformation of a $5$-star with three outer Dirichlet conditions into a $3$-star having just \emph{one} Dirichlet condition without changing the torsion function.}
     \label{fig:transformation-mutliple-dirichlet-vertices}
     \end{figure}
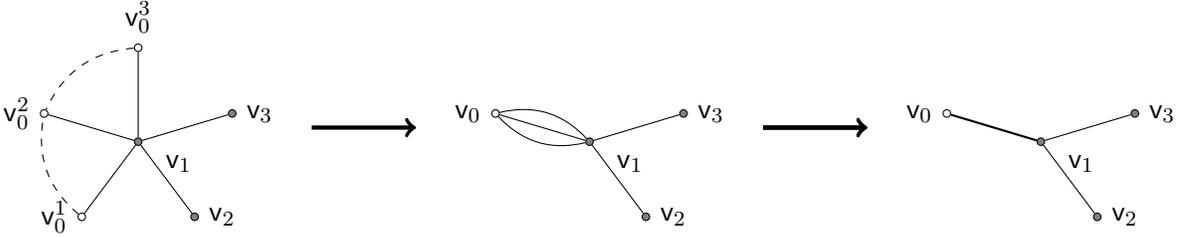

\section{Estimates of the $p$-torsional rigidity}\label{sec:bounds} 

This section is devoted to developing a torsional geometry of graphs: in other words, we are going to derive upper and lower estimates on the $p$-torsional rigidity of a given graph in terms of other quantities, like the number of vertices or the graph's inradius.

\subsection{Upper bounds and a Saint-Venant inequality} \label{sec:upper-bounds}

Inspired by Fiedler's celebrated bound on the smallest positive eigenvalue of the $p$-Laplacian \textit{without Dirichlet conditions} in \cite[4.3]{Fie73} for $p=2$, and by its extension to general $p\in (1,\infty)$ in \cite[Theorem~2.3]{BerKenKur17}, let us now impose Dirichlet conditions on a non-empty vertex set and pursue the goal of estimating both the $p$-torsional rigidity and the bottom of the $p$-spectrum; in passing, we also remove the condition that $\mG=\Vmbc$ has standard edge weights $b=b_{\mathrm{st}}$, vanishing potential $c\equiv 0$, and trivial vertex weights $m\equiv \mathbf{1}$. Motivated by the classical Max-Flow-Min-Cut Theorem, we introduce the following.

\begin{defi}
The \emph{minimal cut weight} $\eta_\mG$ of a graph $\mG=\Vmbc$ is 
\[
\eta_\mG := \min \sum_{\mv \in \mV_1, \mw \in \mV_2} b(\mv,\mw),
\]
where the minimum is taken over all \emph{cuts} in $\mV$, i.e., over all pairs of disjoint and mutually complementary subsets $\mV_1,\mV_2$ of $\mV$. 
\end{defi}

Clearly, a necessary condition for $\mG=\Vmbc$ to have minimal cut weight $\eta_\mG>0$ is that the graph is connected.
We stress that an infinite graph need \textit{not}  have a minimal cut weight. On the other hand, finite graphs certainly do, and indeed a finite graph $\mG$ with standard edge weights $b=b_{\mathrm{st}}$ has minimal cut weight $\eta_\mG $ if and only if its edge-connectivity (in the sense of \cite[Section~1.4]{Die05}) is $\eta_\mG$, i.e., if and only if it stays connected upon removal of any number $\eta'<\eta_\mG$ of edges: this is the setting considered in~\cite{Fie73,BerKenKur17}.

We are now in the position to present a symmetrization lemma that extends the scope of (and is inspired by) the method developed in the proof of~\cite[Theorem~2.3]{BerKenKur17}.
Recall that, by \autoref{conv:dirichlet-singleton}, we are assuming without loss of generality the Dirichlet set $\mV_0$ to be a singleton.

\begin{lemma}\label{thm:Pólya-discr}
Let $\mG=\Vmbc$ be a connected  graph with minimal cut weight $\eta_\mG$ and such that \autoref{ass:compact-embedding-l1} is satisfied.
Then there exists \underline{some} permutation $\sigma: \mV \rightarrow \mV$  (with $\sigma(\mV_0) \subset \mV_0$
and with Dirichlet conditions imposed at one endpoint if $\mV_0\ne\emptyset$) such that
 the $p$-torsional rigidity
 satisfies
\begin{equation}
  \label{eq:Pólya-discr-c}
T_p(\mG)\le \frac{1}{\eta_\mG}T_p(\mP)
\qquad \bigg(\hbox{resp., }
T_p(\mG;\mV_0)\le \frac{1}{\eta_\mG}T(\mP;\mV_0)\hbox{ if }\mV_0 \neq \emptyset\bigg),
\end{equation}
where the expression on the right-hand side of either inequality denotes the $p$-torsional rigidity of a path graph $\mP:=  (\mV,m', b_{\mathrm{st}}, c')$ on the same vertex set, with standard edge weights and with vertex weights $m'$ and potential $c'$ given by
\begin{equation}\label{eq:symm-m-c}
m'(\mv) := m(\sigma(\mv)), \qquad c'(\mv) := c(\sigma(\mv)), \qquad \mv \in \mV.
\end{equation}

Moreover, equality in~\eqref{eq:Pólya-discr-c} holds if and only if $\mG$ is 
a homogeneous path graph (with Dirichlet conditions imposed at one endpoint if $\mV_0\ne\emptyset$).
\end{lemma}

Here, a path graph is said to be \emph{homogeneous} if its edge weight function is constant.

We stress that while the left hand side in both inequalities in \eqref{eq:Pólya-discr-c} is finite by assumption, the right hand side may be infinite if $\mG$ is infinite.

\begin{proof}
We will adapt to our context the symmetrization method developed in the proof of \cite[Theorem~2.3]{BerKenKur17}. We will only discuss the case of $\Lneubc$, the case of $\Ldirbc$ being perfectly equivalent.

Let us take the $p$-torsion function with respect to $\Lneubc$, i.e., the maximizer $\tauneubc$ of the Pólya quotient $\Polneubc$ on $\mG = \Vmbc$, which exists by \autoref{ass:compact-embedding-l1}: as $\mV$ is assumed to be countable, there is some canonical enumeration of the the vertex set, but we prefer to enumerate the  vertices in such a way  that $\tauneubc$ is increasing, i.e.,
  \begin{equation}
    \label{eq:vertex_ordering}
\tauneubc(\mv_i) \le \tauneubc(\mv_{i+1})\qquad\hbox{for all }i:
  \end{equation}
this new enumeration induces a permutation $\sigma$ on $\mV$, and we introduce $m',c'$ as in~\eqref{eq:symm-m-c}.
Moreover, observe that $\tauneubc$ need not attain a maximum or minimum, but this does not pose a problem).
  
Now, we first form a new graph $\hat{\mG}:=(\mV,m',\hat{b},c')$ as follows:  Suppose there is an edge whose endpoints are $(\mv_i,\mv_j)$ (or, in other words, $b(\mv_i,\mv_j) > 0$), for some $i <
  j$, $j-i \geq 2$ (we call this a ``long edge''): we then \textit{replace} this edge by the (necessarily finite) sequence of edges with endpoints $(\mv_i, \mv_{i+1})$,
  $(\mv_{i+1}, \mv_{i+2}), \ldots, (\mv_{j-1}, \mv_j)$.  Here it is essential
  that we temporarily allow for pairs of vertices to be possibly connected by multiple edges, that is, we  create the edges $(\mv_i, \mv_{i+1})$, $(\mv_{i+1}, \mv_{i+2})$, \ldots,
  $(\mv_{j-1}, \mv_j)$ \emph{in addition} to the possibly already existing edges between
  these vertices; and we stipulate that each of them is given the same weight $\hat{b}(\mv_{\ell},\mv_{\ell+1}):=b(\mv_i,\mv_j)$ for $\ell=i,i+1,\dots,j-1$. (We are eventually going to get rid of these parallel edges as in \autoref{rem:multiple}.)
  
This operation increases the Pólya quotient,  i.e., $\Polneubc(\tauneubc)\le \Polneubchat(\tauneubc)$, since
$\|\tauneubc\|_{\ell^1(\mV,m')}=\|\tauneubc\|_{\ell^1(\mV,m)}$ but $\mathcal Q^{\hat{\mG}}(\tauneubc)\le \Qneubc(\tauneubc)$ because
  \begin{align}\label{eq:estimate-l^p-norm-smaller-l^1-norm}
  \begin{aligned}
    \left[\sum_{\ell=i}^{j-1}\hat{b}(\mv_\ell,\mv_{\ell+1}) |\tauneubc(\mv_{\ell})-\tauneubc(\mv_{\ell+1})|^{{p}}
    \right]^{1/{p}} 
    &\leq b(\mv_i,\mv_j)^\frac{1}{p}\sum_{\ell=i}^{j-1} |\tauneubc(\mv_{\ell})-\tauneubc(\mv_{\ell+1})| \\
    &= b(\mv_i,\mv_j)^\frac{1}{p}\sum_{\ell=i}^{j-1} \big(\tauneubc(\mv_{\ell+1})-\tauneubc(\mv_\ell) \big) \\&= b(\mv_i,\mv_j)^\frac{1}{p}\big(\tauneubc(\mv_j)-\tauneubc(\mv_i) \big),
  \end{aligned}
  \end{align}
 where the inequality follows from the inequality $\|x\|_{\ell^p}\le \|x\|_{\ell^1}$ that holds for any positive vector $x$ of finite length, and the first equality from \eqref{eq:vertex_ordering}.
In other words,
\begin{align}\label{eq:symmetrization-proof}
    \sum_{\ell=i}^{j-1}\hat{b}(\mv_\ell,\mv_{\ell+1}) |\tauneubc(\mv_{\ell})-\tauneubc(\mv_{\ell+1})|^{{p}}
    &\leq b(\mv_i,\mv_j)\big(\tauneubc(\mv_j)-\tauneubc(\mv_i) \big)^p.
  \end{align}
  
Repeating this operation for every \emph{long} edge, we obtain a pumpkin chain $\hat{\mG}:=(\mV,m',\hat{b},c')$ (i.e., a graph consisting in a sequence of vertices $\mv_0,\mv_1,\ldots,\mv_n$ such that $\mv_\ell$ is only adjacent with $\mv_{\ell\pm 1}$ for $\ell=1,\dots,n-1$ by means of possibly more than one edge) with larger $p$-torsional rigidity; indeed, by rewiring the graph we have only made the gradient term in the energy functional smaller, but neither the potential term in the denominator of the Pólya quotient, nor its numerator have been changed. 

Technically speaking, pumpkin chains do not fit into our general framework, as they generally contain multiple edges: however, we already know that, for all purposes regarding the analysis of a $p$-Laplacian, such a pumpkin chain $\hat{\mG}$ can be equivalently regarded as a graph $\hat{\hat{\mG}}:=(\mV,m',\hat{\hat{b}},c')$ without multiple edges, but modified weights, see \autoref{rem:multiple}, where $\hat{\hat{b}}(\mv_\ell,\mv_{\ell+1})$ is defined as the sum of all weights $\hat{b}$ over all parallel edges between $\mv_\ell,\mv_{\ell+1}$.
 
It follows from the Max-Flow-Min-Cut Theorem that the minimal cut weight of $\mG'$ satisfies
\[
\eta_{\hat{\hat{\mG}}}:=\min_{\ell=0,\ldots,n-1}\hat{\hat{b}}(\mv_\ell,\mv_{\ell+1}) \ge \eta_\mG.
\]
By \autoref{prop:monotonicity-subgraphs}, replacing the path graph $\hat{\hat{\mG}}=(\mV,m',\hat{\hat{b}},c')$ by a path graph $\mP':=(\mV,m',b',c')$ whose edges have constant weight $b':=\eta_{\hat{\hat{\mG}}}b_{\mathrm{st}}$ further raises the $p$-torsional rigidity.
We finally conclude that 
\begin{equation}\label{eq:allafiera}
T_p(\mG) \leq T_p\big(\hat{\mG}\big)=T_p\big(\hat{\hat{\mG}}\big) \leq T_p\big(\mP') = \frac{1}{\eta_{\hat{\hat{\mG}}}} T_p(\mP)
\le \frac{1}{\eta_\mG} T_p(\mP),
\end{equation}
A close look at the above proof reveals that equality can only hold in~\eqref{eq:symmetrization-proof} if $b$ is constant and no long edge is replaced by a chain of short edges; in this case, also both further inequalities in~\eqref{eq:allafiera} hold.
 This concludes the proof in the case of graphs with no Dirichlet conditions.

All the above arguments carry over without major changes to the case of $\mV_0 \neq \emptyset$ and $\Ldirbc$.
\end{proof}
Note that, according to \eqref{eq:symm-m-c}, $m' = m$ and $c' = c$, whenever $m$ and $c$ are constant maps on $\mV$. Moreover, in the case where $m = \deg_\mG$ and $c \equiv 0$, one can trivially estimate
\begin{align}\label{eq:normalized-saint-venant}
T(\mG;\mV_0) \leq \frac{1}{\eta_\mG} T(\mP;\mV_0) \leq \frac{\max_{\mv \in \mV \setminus \mV_0} \deg_\mG(\mv)}{\eta_\mG} T(\mP';\mV_0),
\end{align}
where $\mP' := (\mV, m = \mathbf{1}, b'=b_{\mathrm{st}}, c \equiv 0)$. 
However, in order to apply \autoref{thm:Pólya-discr} in more generality, a complete knowledge of the weight $m'$ and the potential $c'$ and, in turn, of the permutation $\sigma$ that induces the re-numbering is needed.
However, this is generally impossible, as the permutation $\sigma$ is only given \emph{implicitly} by ordering along the values of the $p$-torsion function of the graph $\mG = \Vmbc$. 

Nevertheless, we may trivially estimate the measure $m'$ from above by $\big( \sup_{\mv \in \mV \setminus \mV_0} m(\mv) \big) \mathbf{1}$ to obtain a rough upper bound on $T_p(\mG)$ and $T_p(\mG;\mV_0)$ by means of \autoref{thm:Pólya-discr}. A smarter choice is available, at least in the case of finite graphs, if we happen to know (typically, by a symmetry argument) \textit{a} vertex $\mv_{\max}\in\mV$ where the $p$-torsion function $\tauneubc$ attains its maximum. At the risk of being redundant, let us stress that finite graphs with $\max_{\mv\in\mV}c(\mv)>0$ or $\mV_0\ne\emptyset$ satisfy \autoref{ass:compact-embedding-l1}.

\begin{cor}\label{cor:tors-fullnew}
Let $\mG = \Vmbc$ be a connected  finite graph with minimal cut weight $\eta_\mG$ and $\mV_0\ne \emptyset$. Then
\begin{align}\label{eq:saint-venant-explicit-mprime}
T_p(\mG;\mV_0) \leq \frac{1}{\eta_\mG}T_p(\mP;\mV_0),
\end{align}
where $\mP := (\mV,\widetilde{m}, b_{\mathrm{st}},c \equiv 0)$ denotes the path graph having one Dirichlet condition at one end and $\widetilde{m}:\mV \rightarrow (0,\infty)$ given by
\begin{align*}
\widetilde{m}(\mv) := \begin{cases}
\min\limits_{\mv \in \mV \setminus \mV_0} m(\mv), & \text{if $\mv \neq \mv_{\max}$}, \\ m(\mv_{\max}) + \sum\limits_{\stackrel{\mw \in \mV \setminus \mV_0}{\mw \neq \mv_{\max}}} \Big(m(\mw) - \min\limits_{\mv \in \mV \setminus \mV_0} m(\mv) \Big), & \text{if $\mv = \mv_{\max}$},
\end{cases}
\end{align*}
where $\mv_{\max} \in \mV \setminus \mV_0$, $\mV_0 \subset \mV$ such that $\tauneubc(\mv_{\max}) = \Vert \tauneubc \Vert_{\ell^\infty(\mV)}$ (resp., $\taudirb(\mv_{\max}) = \Vert \taudirb \Vert_{\ell^\infty(\mV \setminus \mV_0)}$ if $\mV_0 \neq \emptyset$). 
\end{cor}
\begin{proof}
According to \autoref{thm:Pólya-discr} and \autoref{prop:monotonicity-subgraphs}, it follows that $T_p(\mG) \leq \frac{1}{\eta_\mG}T_p(\mP')$, where $\mP' := (\mV,m',b=b_{\mathrm{st}},c \equiv 0)$ is a path graph with standard edge weights, a Dirichlet condition at one endpoint and the vertex $\mv_{\max}$ at the other endpoint and $m':=m\circ \sigma$ for some permutation $\sigma:\mV\to \mV$ with $\sigma(\mV_0)\subset \mV_0$. Therefore
 \[
\begin{split}
\Vert \taudirbpathpr \Vert_{\ell^1(\mV \setminus \mV_0,m')} &= \sum_{\stackrel{\mw \in \mV \setminus \mV_0}{\mw \neq \mv_{\max}}} \taudirbpathpr(\mw) \Big( m'(\mw) - \min_{\mv \in \mV \setminus \mV_0} m'(\mv) \Big) \\
& \qquad + \sum_{\stackrel{\mw \in \mV \setminus \mV_0}{\mw \neq \mv_{\max}}} \taudirbpathpr(\mw) \min_{\mv \in \mV \setminus \mV_0} m'(\mv) + \taudirbpathpr(\mv_{\max}) m'(\mv_{\max}) \\
&\leq \taudirbpathpr(\mv_{\max}) \sum_{\stackrel{\mw \in \mV \setminus \mV_0}{\mw \neq \mv_{\max}}} \Big( m(\mw) - \min_{\mv \in \mV \setminus \mV_0} m(\mv) \Big) \\& \qquad + \sum_{\stackrel{\mw \in \mV \setminus \mV_0}{\mw \neq \mv_{\max}}} \taudirbpathpr(\mw) \min_{\mv \in \mV \setminus \mV_0} m(\mv) + \taudirbpathpr(\mv_{\max}) m(\mv_{\max}) \\&= \Vert \taudirbpathpr \Vert_{\ell^1(\mV \setminus \mV_0,\widetilde{m})},
\end{split}
\]
which shows \eqref{eq:saint-venant-explicit-mprime} according to \autoref{prop:variational-char}.
\end{proof}

If $c\equiv 0$, we can further bound the right-hand side in the second inequality in~\eqref{eq:Pólya-discr-c}, hence turning the qualitative comparison in~\autoref{thm:Pólya-discr} into a quantitative estimate.
This upper bound is reminiscent of Pólya's so-called \textit{Saint-Venant inequality} on 
the $p$-torsional rigidity of a planar domain, see \cite[Theorem~4.6]{MugPlu23} for a related result on metric graphs.

\begin{cor}\label{cor:estimate-p-torsion-standard-weights}
    Let $\mG=\Vmbc$
be a connected  finite graph with minimal cut weight $\eta_\mG$, and let $\emptyset \neq \mV_0$. If $\mV\setminus\mV_0$ contains $n$ vertices, then 
    \begin{equation}\label{eq:upper-torsion-m}
    T_p(\mG;\mV_0) \leq \frac{n^{p-1}}{\eta_\mG} m(\mV \setminus \mV_0)^{p}.
    \end{equation}
Let, additionally, $p=2$ and $m \equiv \mathbf{1}$. Then the torsional rigidity satisfies the improved estimate
    \begin{equation}\label{eq:upper-torsion}
    T_2(\mG;\mV_0) \leq \frac{n(n+1)(2n+1)}{6\eta_\mG},
    \end{equation}
 and equality is attained if and only if $\mG$ is the path graph with Dirichlet conditions at only one endpoint.
\end{cor}

\begin{proof}
Let $\mP:=(\mV,m',b_{\mathrm{st}}, c \equiv 0)$ be the path graph on the same vertex set but with standard edge weights constructed in~\autoref{thm:Pólya-discr}. As in the proof of~\autoref{thm:Pólya-discr}, we generally have no information about the permutation $\sigma$ and, hence, about the profile of $m'=m\circ \sigma$, but by~\eqref{eq:p-torsion-function-path-left-dirichlet} (remember our \autoref{conv:dirichlet-singleton}), we observe that the $p$-torsion function of a path graph on $n+1$ vertices (with Dirichlet conditions at $\mV_0=\{\mv_0\}$) satisfies
\begin{align}\label{eq:upper-bound-torsion-using-path-graph}
\begin{aligned}
\Vert \taudirbpath \Vert_{\ell^1(\mV,m')} &= \sum_{j = 1}^n \taudirbpath(\mv_j)m'(\mv_j) = \sum_{j=1}^n \sum_{\ell = 1}^j \bigg( \sum_{k=\ell}^n m'(\mv_k) \bigg)^\frac{1}{p-1}m'(\mv_j) \\&=  \sum_{\ell=1}^n \sum_{j = \ell}^n \bigg(\sum_{k=\ell}^n m'(\mv_k) \bigg)^\frac{1}{p-1}m'(\mv_j) =  \sum_{\ell=1}^n \bigg(\sum_{k=\ell}^n m'(\mv_k) \bigg)^\frac{p}{p-1} \leq n m'(\mV \setminus \mV_0)^\frac{p}{p-1}.
\end{aligned}
\end{align}
Now, observe that 
the measure $m'$ and the potential $c'$ preserve the total mass of the graph in the sense that
\begin{align*}
m(\mV) = m'(\mV) \:\: \text{(resp., $m(\mV \setminus \mV_0) = m'(\mV \setminus \mV_0)$)}, \quad c(\mV) = c'(\mV) \:\: \text{(resp., $c(\mV \setminus \mV_0) = c'(\mV \setminus \mV_0)$)}.
\end{align*}
Thus, combining \eqref{eq:upper-bound-torsion-using-path-graph} with the latter inequality in \eqref{eq:Pólya-discr-c},  \eqref{eq:upper-torsion-m} follows.

Likewise, letting now $p=2$, $m \equiv \mathbf{1}$ and imposing standard edge weights $b=b_{\mathrm{st}}$ on the original graph, we obtain \eqref{eq:upper-torsion}  by ``reducing'' $\mG$ to a path graph $\mP:=(\mV,m,b'=b_{\mathrm{st}}, c \equiv 0)$ with same number of vertices and using the known formula for the torsional rigidity of $\mP$ from \autoref{prop:p-torsion-function-pat}.
Now, \eqref{eq:upper-torsion} follows from
plugging \eqref{eq:formula-torsional-rigidity-stadard-and-m-equiv-1} into \autoref{thm:Pólya-discr}. 
\end{proof}

\begin{rem}
Upon an exhaustion argument, \autoref{cor:estimate-p-torsion-standard-weights} suggests that an infinite graph cannot generally be expected to have finite $p$-torsional rigidity unless $\eta_\mG=0$: roughly speaking, this means that the periphery of the graph must be very loosely connected.
\end{rem}

Adapting an idea that goes back to~\cite{PolSze51}, we can find the following relation between the $p$-torsional rigidity 
and the
\emph{bottom of the spectrum} of the $p$-Laplacian $\Lneubc$ (resp., $\Ldirbc$); we recall that the latter need not to be an eigenvalue.

The proof of the following result is analogous to those of \cite[Proposition~5.1]{MugPlu23} and \cite[Corollary~5.5]{MazTol23} and we therefore omit it. The proof of the strict inequality in \eqref{eq:p-torsion-lowest-eigenvalue-dirichlet} is based on the existence of a unique weak (hence pointwise) $\ell^1$-solution $\tauneubc$ (resp., $\taudirbc$) of \eqref{eq:discr-ellipt-1-p} (resp., \eqref{eq:discr-ellipt-1-p-dir}) -- which is guaranteed by \autoref{thm:well-p-l1} -- and on Hölder's inequality applied to $\|\tauneubc\cdot \mathbf{1}\|_{\ell^1(\mV,m)}$ or $\|\taudirbc\cdot \mathbf{1}\|_{\ell^1(\mV,m)}$: Hölder's inequality becomes an equality if and only if $\tauneubc$ and $\taudirb$ are multiple of $\mathbf{1}$, which would necessarily contradict the fact that $\tauneubc$ (resp., $\taudirbc$) is an $\ell^1$-solution of \eqref{eq:discr-ellipt-1-p} (resp., \eqref{eq:discr-ellipt-1-p-dir}), whenever $c$ is not constant (resp., whenever $\mV_0\ne\emptyset$).

\begin{prop}\label{prop:p-cheeger}
Let $\mG=\Vmbc$ be a graph. If \autoref{assum:finite-meas} is satisfied, then
\begin{align}\label{eq:p-torsion-lowest-eigenvalue-neumann}
\lambda_{0,p}(\mG) T_p(\mG) \le  m(\mV)^{p-1}
\end{align}
if $\sup_{\mv\in\mV}c(\mv)>0$ (resp., 
    \begin{align}\label{eq:p-torsion-lowest-eigenvalue-dirichlet}
    \lambda_{0,p}(\mG;\mV_0)T_p(\mG;\mV_0) \le m(\mV\setminus\mV_0)^{p-1}
    \end{align}
if $\mV_0\ne \emptyset$).
Whenever, additionally, \autoref{ass:compact-embedding-l1} is satisfied, the inequality \eqref{eq:p-torsion-lowest-eigenvalue-neumann} is strict if $c$ is not constant (resp., \eqref{eq:p-torsion-lowest-eigenvalue-dirichlet} is strict).
\end{prop}

\subsection{Lower bounds} \label{sec:lower-bounds}

We begin by observing a simple lower bound, which can be proved plugging the constant function $\mathbf{1}$ (resp., its restriction $\mathbf{1}_{\mV\setminus\mV_0}$ to $\mV\setminus \mV_0$) into \eqref{eq:tors-def-1}, under the condition that $(\mV,m)$ is a finite measure space.

\begin{lemma}\label{lem:lower-element}
Let $\mG =\Vmbc$ be a graph. If \autoref{assum:finite-meas} is satisfied,  then the $p$-torsional rigidity satisfies
\begin{align}\label{eq:trivial-estimate-p-torsion-below}
T_p(\mG) \geq \frac{m(\mV)^{p}}{\sum\limits_{\mv \in \mV} c(\mv)}, \quad \bigg(\text{resp. } T_p(\mG;\mV_0) \geq \frac{m(\mV \setminus \mV_0)^{p}}{\sum\limits_{(\mv,\mw) \in \mV\setminus\mV_0 \times \mV_0} b(\mv,\mw)+\sum\limits_{\mv \in \mV\setminus \mV_0} c(\mv)}\ \bigg)\hbox{}.
\end{align}
\end{lemma}

Let us derive a further, slightly more sophisticated lower estimate on the $p$-torsional rigidity $T_p(\mG;\mV_0)$ 
 using a simple comparison principle: we start by estimating the torsional rigidity of path graphs,  as in \autoref{prop:p-torsion-function-pat}.

\begin{lemma}\label{prop:lower-bound-path-graph}
Let $\mP = \Vmb$ be a $1-p$-torsional-admissible path graph. 
Then its $p$-torsional rigidity $T_p(\mP;\mV_0)$ 
with Dirichlet conditions at $\mV_0=\{\mv_0\}$
satisfies 
\begin{align}\label{eq:lower-bound-inradius-path}
T_p(\mP;\mV_0)  \geq  \mathrm{Inr}_{p}(\mP;\mV_0)^{-1} m(\mV \setminus \mV_0)^p.
\end{align}
In particular, $\mP$ has finite measure.
\end{lemma}

\begin{proof}
To fix the ideas, let us denote the vertex set by $\{\mv_j:j\in J\}$, for some $J\subset \N$, with $\mV_0=\{\mv_0\}$.

Using the explicit formula for the $p$-torsion function $\taudirbpath$ for a path graph having a Dirichlet condition at one end derived in \eqref{eq:p-torsion-function-path-left-dirichlet}, we are able to estimate its corresponding torsional rigidity $T_p(\mP;\mV_0)$ from below. To begin with, we observe that, by \eqref{eq:estimate-l^p-norm-smaller-l^1-norm},
\begin{align}\label{eq:estimate-path-inradius}
\begin{aligned}
    T_p(\mP;\mV_0) &= \Bigg[\sum_{ j \in J} \sum_{\ell=1}^j \bigg( \frac{1}{b(\mv_{\ell-1}, \mv_{\ell})} \sum_{\substack{k \in J\\ k>\ell-1 }} m(\mv_k) \bigg)^{\frac{1}{p-1}} m(\mv_j) \Bigg]^{p-1} 
    \\&= \Bigg[\sum_{\ell \in J} \sum_{\substack{k \in J\\ k>\ell-1 }} \bigg( \frac{1}{b(\mv_{\ell-1}, \mv_{\ell})}  \sum_{\substack{k \in J\\ k>\ell-1 }} m(\mv_k) \bigg)^{\frac{1}{p-1}} m(\mv_j) \Bigg]^{p-1} \\&= \Bigg[\sum_{\ell \in J} \bigg( \frac{1}{b(\mv_{\ell-1}, \mv_{\ell})} \bigg)^\frac{1}{p-1} \bigg(  \sum_{\substack{k \in J\\ k>\ell-1 }} m(\mv_k)\bigg)^{\frac{p}{p-1}} \Bigg]^{p-1} \\& \geq \sum_{ \ell \in J} \frac{1}{b(\mv_{\ell-1}, \mv_\ell)}\bigg( \sum_{\substack{k \in J\\ k>\ell-1 }} m(\mv_k) \bigg)^p.
\end{aligned}
\end{align}
Now by Hölder's inequality with respect to $p$ and $p' = \frac{p}{p-1}$
, one has that
\[
\|f\|_1 \le \|fg\|_p  \|g^{-1}\|_\frac{p}{p-1} \qquad \text{for all $f \in \ell^1(J)$ and $g \in \mathbb{R}^J$ with $g$ strictly positive on $J$}:
\]
applying this to the pair of functions $f,g \in \mathbb{R}^J$ given by
\[
f(\ell) := \sum_{\stackrel{k \in J}{k>\ell-1}} m(\mv_k), \qquad \text{and} \qquad g(\ell) := b(\mv_{\ell-1},\mv_\ell)^{-\frac{1}{p}}, \qquad \text{for $\ell \in J$}, 
\]
(note that $f$ indeed belongs to $\ell^1(J)$ as $\mP$ is $q-p$-torsional admissible and due to the representation of the $p$-torsion function in \eqref{eq:p-torsion-function-path-left-dirichlet}) \eqref{eq:estimate-path-inradius} thus implies that
\begin{equation}\label{eq:lower-estimate-inradius-path-proof}
\begin{aligned}
T_p(\mP;\mV_0) &\geq \bigg( \sum_{ \ell \in J} b(\mv_{\ell-1}, \mv_\ell)^\frac{1}{p-1} \bigg)^{1-p} \bigg( \sum_{\ell \in J} \sum_{\substack{k \in J\\ k>\ell-1 }} m(\mv_k) \bigg)^p \\&=  \mathrm{Inr}_{p}(\mP;\mV_0)^{-1} \bigg( \sum_{ \ell \in J} \sum_{\substack{k \in J\\ k>\ell-1 }} m(\mv_k) \bigg)^p \\&\geq \mathrm{Inr}_{p}(\mP;\mV_0)^{-1} m(\mV \setminus \mV_0)^p.
\end{aligned}
\end{equation}
This concludes the proof.
\end{proof}
This leads to the following, more general lower estimate.

\begin{cor}\label{cor:better-estimate-from-below}
Let $\mT:=\Vmb$ be a finite tree graph
 and let $\mV_0\ne \emptyset$. 
Then 
\begin{equation}\label{eq:lower-tors-stars-estim}
T_p(\mT;\mV_0) \geq \mathrm{Inr}_{p}(\mT;\mV_0)^{-1}\min_{\mv \in \mV\setminus \mV_0} m(\mv)^p.
\end{equation}
holds.
\end{cor}
\begin{proof}
By \autoref{cor:monotonocitiy-torsion-subgraphs} in combination with \eqref{eq:lower-bound-inradius-path} it follows that 
\[
T_p(\mT;\mV_0) \geq  T_p(\mP;\mV_0) \ge \mathrm{Inr}_{p}(\mP;\mV_0)^{-1}m(\mV(\mP)\setminus \mV_0)^p \geq \mathrm{Inr}_{p}(\mP;\mV_0)^{-1}\min_{\mv \in \mV\setminus \mV_0} m(\mv)^p
\]
for any path subgraph $\mP$ (whose vertex set we denote by $\mV(\mP)$) of $\mT$. Taking the supremum over all such paths  yields the claim.
 \end{proof}

\subsection{An application: quantitative lower estimates on the bottom of the $p$-spectrum} \label{sec:appliaction-lower-estimates}

A qualitative lower bound for the smallest positive eigenvalue of $p$-Laplacians without Dirichlet conditions had been obtained in \cite{BerKenKur17}.
By adapting the proof of \autoref{thm:Pólya-discr} to the Rayleigh quotient instead of the Pólya quotient to characterize the smallest eigenvalue of $\Lneubc$ and $\Ldirbc$, respectively, we can extend \cite[Theorem~2.3]{BerKenKur17} to $p$-Schrödinger operators with general potential, with general weights, and possibly with Dirichlet conditions.
\begin{prop}\label{prop:eigenvalue-eta-fold-edge-connected-minimized path}
Let $\mG=\Vmbc$ be a connected  graph with minimal cut weight $\eta_\mG$. If $\sup_{\mv \in \mV} c(\mv) > 0$ (resp., $\mV_0 \neq \emptyset$), and if $\lambda_{0,p}(\mG)$ (resp., $\lambda_{0,p}(\mG;\mV_0)$) is an eigenvalue, then there exists 
a permutation $\sigma: \mV \rightarrow \mV$ with $\sigma(\mV_0) \subset \mV_0$ (and with Dirichlet conditions imposed at one endpoint if $\mV_0 \neq \emptyset$) such that
\begin{equation}\label{eq:Rayleigh-discr}
\lambda_{0,p}(\mG)\ge \eta_\mG \lambda_{0,p}(\mP)
\qquad \big( \hbox{resp., }
\lambda_{0,p}(\mG;\mV_0)\ge \eta_\mG \lambda_{0,p}(\mP;\mV_0) \big),
\end{equation}
where $\mP:=(\mV,m',b_{\mathrm{st}},c')$ is a path graph on the same vertex set, 
with standard edge weights, and with vertex weights $m'$ and potential $c'$ given by
\begin{align*}
m'(\mv) := m(\sigma(\mv)), \quad c'(\mv) := c(\sigma(\mv)), \quad \mv \in \mV.
\end{align*}
Equality in~\eqref{eq:Rayleigh-discr} holds if and only if $\Vmbc$ is a homogeneous path graph (with Dirichlet conditions imposed at one endpoint if $\mV_0 \neq \emptyset$). 
\end{prop}
Note that if $\wneubc(\mV,m)$ (resp.\ $\wdirbc(\mV,m)$) is compactly embedded in $\ell^p(\mV,m)$, then there indeed exists a ground state in $\wneubc(\mV,m)$ (resp.\ $\wdirbc(\mV,m)$).

As in \autoref{thm:Pólya-discr}, the vertex weight $m'$ as well as the potential $c'$ are only given \emph{implicitly} via ordering the ground state in terms of its values; however, we can formulate the following counterpart of \autoref{cor:tors-fullnew}.
\begin{cor}
Let $\mG = \Vmbc$ be a connected  finite graph with minimal cut weight $\eta_\mG$. Let $\mV_0\ne \emptyset$, and consider a ground state $\varphi_0^\mG \in \wdirbc(\mV,m)$.
Then
\begin{align}\label{eq:saint-venant-explicit-mprime-bottom-of-p-spec}
\lambda_{0,p}(\mG;\mV_0) \geq \eta_\mG \lambda_{0,p}(\mP;\mV_0),
\end{align}
where $\mP := (\mV,\widetilde{m}, b_{\mathrm{st}},c \equiv 0)$
denotes the path graph having one Dirichlet condition at one end and $\widetilde{m}:\mV \rightarrow (0,\infty)$ given by
\begin{align*}
\widetilde{m}(\mv) := \begin{cases}
\min\limits_{\mv \in \mV \setminus \mV_0} m(\mv), & \text{if $\mv \neq \mv_{\max}$}, \\ m(\mv_{\max}) + \sum\limits_{\stackrel{\mv \in \mV \setminus \mV_0}{\mv \neq \mv_{\max}}} \Big(m(\mv) - \min\limits_{\mv \in \mV \setminus \mV_0} m(\mv) \Big), & \text{if $\mv = \mv_{\max}$},
\end{cases}
\end{align*}
where $\mv_{\max} \in \mV \setminus \mV_0$, $\mV_0 \subset \mV$ such that $\varphi_0^\mG(\mv_{\max}) = \Vert \varphi_0^\mG \Vert_{\ell^\infty(\mV \setminus \mV_0)}$. 
\end{cor}

Moreover, using \autoref{prop:eigenvalue-eta-fold-edge-connected-minimized path} one can also derive a counterpart of \autoref{cor:estimate-p-torsion-standard-weights} as done for the $p$-torsional rigidity, 
in terms of the following lower bound for the bottom of the $p$-spectrum in the Dirichlet case: this can also be seen as a \emph{Fiedler-type bound} allowing for Dirichlet conditions.
\begin{cor}
Let $p=2$ and $\mG = (\mV,m \equiv \mathbf{1},b=b_{\mathrm{st}}, c \equiv 0)$ be a finite graph with $\mV_0 \neq \emptyset$ such that $\mV \setminus \mV_0$ consists of $n$ vertices and with minimal cut weight $\eta_\mG$. Then the smallest eigenvalue satisfies
\begin{equation}
\lambda_{0,2}(\mG;\mV_0) \geq 2 \eta_\mG \left(1-\cos\left(\frac{\pi}{2n+1}\right)\right).
\end{equation}
\end{cor}
\begin{proof}
The claimed inequality follows immediately from \eqref{eq:Rayleigh-discr} and~\eqref{eq:l2fiedl} below by computing the smallest eigenvalue $\lambda_{0,2}(\mP;\mV_0)$ for a path graph $\mP = (\mV,m \equiv \mathbf{1}, b=b_{\mathrm{st}},c \equiv 0)$ on the same number of vertices as $\mG$ having a Dirichlet condition at one end.
\end{proof}

Another lower bound on the bottom of the $p$-spectrum (and, in turn, an upper bound for the $p$-torsional rigidity) can be obtained in terms of the inradius and mean distance introduced in~\eqref{eq:inradius-defi} and \eqref{eq:q-meanstdist-defi}. The proof uses the edge weight inverted graph introduced in~\eqref{eq:def-g-1} and is similar to that of \autoref{lem:geohaekelp}.
 
\begin{theo}\label{prop:estimate-inradius}
Let $\mG=\Vmbc$ be a graph such that $\Mp(\mV,m,b^{-1};\mV_0)$ is finite,
and let $\mV_0\ne \emptyset$
Under \autoref{assum:finite-meas} there holds
\begin{equation}\label{eq:lower-l0-inr}
\lambda_{0,p}(\mG;\mV_0) \geq \frac{1}{m(\mV\setminus \mV_0)\Mp(\mG^{-1};\mV_0)} \geq \frac{1}{m(\mV\setminus \mV_0) \mathrm{Inr}_p(\mG^{-1};\mV_0)}:
\end{equation}
in particular, an $\ell^p-\ell^p$-Poincaré inequality holds and
\begin{equation}\label{eq:upper-T-inr}
T_p(\mG;\mV_0) < m(\mV \setminus \mV_0)^p \Mp(\mG^{-1};\mV_0) \leq m(\mV \setminus \mV_0)^p \mathrm{Inr}_p(\mG^{-1};\mV_0)
\end{equation}
and $\mG$ is (uniquely) $1-p$-torsional-admissible.\end{theo}

Let us remark that finiteness of $\Mp(\mV,m,b^{-1};\mV_0)$ implies that $\mG$ is connected.
Also, observe that the left hand side in \eqref{eq:lower-l0-inr} (resp., in \eqref{eq:upper-T-inr}) is monotonically increasing (resp., decreasing) in $c$, but the corresponding right hand side is not: in other words, these estimates are optimal for $c\equiv 0$.

In fact, this result represents the discrete counterpart to \cite[Theorem~4.4.3]{Plu21b} and \cite[Remark~5.2]{MugPlu23}, whose proofs are, in turn, presenting similarities with the method proposed in~\cite[Lemma~1.9]{Chu97} for the unweighted, normalized Laplacian without boundary conditions.
\begin{proof}
Let $(\mv_0,\mv_1,\dots,\mv_{n+1})$ be an arbitrary path in $\mG = (\mV,m,b,c)$ connecting any vertex $\mv \in \mV\setminus \mV_0$ with $\mw \in \mV_0$ (noting that w.l.o.g.\ $\mV_0$ is a singleton). By Hölder's inequality
\begin{align*}
\vert f(\mv) \vert^p &\leq \bigg( \sum_{j=0}^n \vert \nabla_{\mv_j,\mv_{j+1}}f \vert \bigg)^p = \bigg( \sum_{j=0}^n b(\mv_j,\mv_{j+1})^{-\frac{1}{p}}b(\mv_j,\mv_{j+1})^{\frac{1}{p}}\vert \nabla_{\mv_j,\mv_{j+1}}f \vert \bigg)^p \\&\le\bigg( \sum_{j=0}^n b(\mv_j,\mv_{j+1})^{\frac{1}{1-p}} \bigg)^{p-1} \bigg( \sum_{j=0}^n b(\mv_j,\mv_{j+1}) \vert \nabla_{\mv_j,\mv_{j+1}}f \vert^p \bigg),
\end{align*}
and as $(\mv_0,\mv_1,\dots,\mv_{n+1})$ was arbitrary, this yields
\begin{align*}
\vert f(\mv) \vert^p \leq \mathrm{dist}_{p,b^{-1}}(\mv,\mV_0)^{p-1}\Qdirbc(f).
\end{align*}
Summing against $m(\mv)$ over all elements of $\mV\setminus \mV_0$, we find
\[
\Vert f \Vert_{\ell^p(\mV\setminus \mV_0,m)}^p = \sum_{\mv \in \mV\setminus \mV_0} \vert f(\mv) \vert^p m(\mv) \leq m(\mV\setminus \mV_0)\Mp(\mG^{-1};\mV_0)\Qdirbc(f).
\]
Taking the infimum over $f\in \ell^p(\mV\setminus \mV_0,m)$ finally yields \eqref{eq:lower-l0-inr}.

Now, \eqref{eq:upper-T-inr} follows from~\autoref{prop:p-cheeger}. Moreover, \autoref{cor:wellp} yields a unique $p$-torsion function in $\wppdirbc(\mV,m) \hookrightarrow \wdirbc(\mV,m)$ (and the embedding is dense).
\end{proof}

\begin{exa}
Let $\mP=(\mV,m \equiv \mathbf{1}, b = b_{\mathrm{st}}, c \equiv 0)$ be the unweighted path graph. If a Dirichlet condition is imposed on one endpoint only (say, $\mV_0=\{\mv_0\}$) and $\mV\setminus \mV_0$ consists of $n-1$ vertices, then 
\[
\Mt(\mP;\mV_0)=\frac{1}{n-1}\sum_{k=1}^{n-1}k=\frac{n}{2}
\]
(and, for general $p\in \N$, 
\[
\Mp(\mP;\mV_0)=p^{-1}\sum_{k=0}^{p-1}\begin{pmatrix}
p\\ k
\end{pmatrix}B_k (n-1)^{p-1-k}
\]
by Faulhaber's formula), whereas $\mathrm{Inr}_2(\mP;\mV_0)={n-1}$:
whence, for $p=2$, \autoref{prop:estimate-inradius} implies
\[
\lambda_{0,2}(\mP;\mV_0)\ge \frac{2}{n(n-1)}.
\]
The same estimate can be obtained invoking the theory of landscape functions: indeed, for all $p\in (1,\infty)$
\[
\|\tau^{\mP;\mV_0}_p\|_\infty=\tau^{\mP;\mV_0}_p(\mv_{n-1}) =  \sum_{\ell=1}^{n-1} \bigg(\frac{1}{b_{\mathrm{st}}(\mv_{\ell-1}, \mv_{\ell})} \sum_{k=\ell}^{n-1} m(\mv_k) \bigg)^\frac{1}{p-1}
=  \sum_{k=1}^{n-1} (n-k)^\frac{1}{p-1},
\]
by \autoref{prop:p-torsion-function-pat} and because we are considering standard edge weights. Accordingly, by~\cite[Proposition~4.1]{Mug23}
\begin{equation}\label{eq:appl-landscape-f-p}
\lambda_{0,p}(\mP;\mV_0)\ge \left(\frac{1}{ \sum\limits_{k=1}^{n-1} (n-k)^\frac{1}{p-1}}\right)^{p-1},    
\end{equation}
and in particular
\[
\lambda_{0,2}(\mP;\mV_0)\ge \frac{1}{ \sum\limits_{k=1}^{n-1} (n-k)}=\frac{2}{n(n-1)}.
\]
By means of the above formula for $\mathrm{Inr}_2(\mP;\mV_0)$, and in view of \autoref{prop:estimate-inradius} and \eqref{eq:lower-tors-stars-estim}, we also find the two-sided bound
\[
n-1 \leq T_2(\mP;\mV_0) \leq \frac{(n-1)^2n}{2}
\]
for the torsional rigidity.
\end{exa}

Combining \eqref{eq:appl-landscape-f-p} with \autoref{prop:eigenvalue-eta-fold-edge-connected-minimized path}, we immediately deduce the following lower bound for the bottom of the spectrum of $\Ldirbc$. We are not aware of any other lower bound on the bottom of the spectrum of the $p$-Laplacian on graphs with Dirichlet conditions (but see the Cheeger-type-inequalities in~\cite[Theorem~3]{Amg03}, \cite[Theorem~5.1]{BuhHei09b} and~\cite[Theorem~3.10]{KelMug16} for the graph $p$-Laplacian without boundary conditions).

\begin{cor}\label{prop:fiedler-dirich}
Let $\mG=(\mV,m \equiv \mathbf{1},b,c \equiv 0)$ be a connected finite graph with minimal cut weight $\eta_\mG$. If $\mV_0 \neq \emptyset$ and $\mV\setminus \mV_0$ consists of $n-1$ vertices, then
\begin{equation}\label{eq:fiedler-dirich}
\lambda_{0,p}(\mG;\mV_0)\ge \eta_\mG \left(\frac{1}{ \sum\limits_{k=1}^{n-1} (n-k)^\frac{1}{p-1}}\right)^{p-1}.
\end{equation}
\end{cor}
Using an elementary symmetry argument we deduce the following, where we denote by $\lambda_{1,p}(\mG)$ the lowest positive eigenvalue of $\Lneubc$.
\begin{cor}\label{prop:fiedler-neum}
Let $\mG=(\mV,m \equiv \mathbf{1},b,c\equiv 0)$ be a connected  finite graph 
with minimal cut weight $\eta_\mG$. If $\mV$ consists of $2n-1$ or $2n$ vertices, then the lowest positive eigenvalue $\lambda_{1,p}(\mG)$ satisfies
\begin{equation}\label{eq:fiedler-neum}
\lambda_{1,p}(\mG) \ge \eta_\mG \left(\frac{1}{ \sum\limits_{k=1}^{n-1} (n-k)^\frac{1}{p-1}}\right)^{p-1}.
\end{equation}
\end{cor}
\begin{proof}
Up to treating general edge weights $b$ like in the proof of \autoref{thm:Pólya-discr}, we can mimc the proof of \cite[Theorem~2.3]{BerKenKur17} and deduce that $\lambda_{1,p}(\mG) \ge \eta_\mG \lambda_{1,p}(\mP)$ where $\mP$ is an appropriate path graph $\mP=(\mV,m\equiv \mathbf{1},b_{\mathrm{st}},c\equiv 0)$; we can assume $\mP$ to consist of an odd number $2n-1$ of vertices, otherwise we replace it by a new path graph -- which for simplicity we denote again by $\mP$ with one more node -- as this only further lowers $\lambda_{1,p}$.
We know that the ground state is a constant function, and  that the eigenfunctions $\phi^\mP$ associated with $\lambda_{1,p}$ of $\mP$ change sign; by symmetry, they must vanish in the central vertex of $\mP$. Because the frequency of $\phi^\mP$ agrees with the frequency of the ground state on the path of length $n$ with Dirichlet condition at one endpoint, the claim follows from \autoref{prop:fiedler-dirich}.
\end{proof}

\subsection{A discrete Kohler-Jobin inequality}\label{sec:kohler}

A deep result in the theory of torsional rigidity is an inequality conjectured (for $p=2$) by Pólya in~\cite{PolSze51}, and finally proved by Kohler-Jobin in~\cite{Koh78}, and then generalized to $p\ne 2$ by Brasco in~\cite{Bra14}: it states that, for the $p$-Laplacian with Dirichlet boundary conditions and upon consideration of a natural scaling exponent $\alpha(p,d)$ needed to turn it into an absolute number, the product
\[
T^{\alpha(p,d)}_p \lambda_{0,p}
\]
of the $p$-torsional rigidity $T_p = T_p(\Omega)$ with the bottom of the spectrum $\lambda_{0,p}=\lambda_{0,p}(\Delta_p)$ of the $p$-Laplacian is minimized by the ball, among all $d$-dimensional open domains $\Omega \subset \mathbb{R}^d$ of finite Lebesgue measure. 
 It has been known since the isoperimetric inequality in~\cite{Fie73} that, in the discrete settings, path graphs play the role of balls. In this section we will discuss several partial extensions of the Kohler-Jobin inequality to the combinatorial case, under the standing assumption that 
 \[
p=2, \quad \mG\hbox{ is connected}, \quad \text{and} \quad c \equiv 0. 
 \]
 
\begin{exa}
We already know from~\autoref{prop:p-torsion-function-pat} that for path graphs $\mP$ with standard edge weights on $n$ edges and Dirichlet conditions on one endpoint
\begin{equation}\label{eq:torsional-rigidity-path-m1-n}
T_2(\mP;\mV_0)=\frac{n(n+1)(2n+1)}{6}\qquad\hbox{for }m=\mathbf{1}
\end{equation}
and
\begin{equation}\label{eq:torsional-rigidity-path-mdeg-n}
T_2(\mP;\mV_0)=\frac{n(2n-1)(2n+1)}{3}\qquad\hbox{for }m=\deg;
\end{equation}
furthermore, it can be derived from \cite[Section~4.4]{Fie73} and \cite[Example~1.4]{Chu97}, and by elementary symmetry arguments, that
\begin{equation}\label{eq:l2fiedl}
\lambda_{0,2}(\mP;\mV_0)=2\left(1-\cos\left(\frac{\pi}{2n+1}\right)\right)\qquad\hbox{for }m=\mathbf{1}
\end{equation}
and
\begin{equation}\label{eq:chung}
\lambda_{0,2}(\mP;\mV_0)=1-\cos\left(\frac{\pi}{2n}\right)\qquad\hbox{for }m=\deg.
\end{equation}
Let us now consider the same quantities for the star graph $\mS$ on $n$ edges, with Dirichlet conditions on all vertices of degree 1. Clearly, the corresponding Laplacian is nothing but a multiple $n$ of the Laplacian on the path graph on one edge, with Dirichlet condition on one of the vertices. Accordingly, the above quantities become
\begin{equation}\label{eq:torsional-rigidity-star-m1}
T_2(\mS;\mV_0)=\frac{1}{n}\qquad\hbox{for both }m=\mathbf{1}\hbox{ and }m=\deg,
\end{equation}
whereas
\begin{equation}\label{eq:lambda0-star-m1}
\lambda_{0,2}(\mS;\mV_0)=2n\left(1-\cos\left(\frac{\pi}{3}\right)\right)\qquad\hbox{for }m=\mathbf{1}
\end{equation}
and
\begin{equation}\label{eq:lambda0-star-mdeg}
\lambda_{0,2}(\mS;\mV_0)=1-\cos\left(\frac{\pi}{2}\right) =1 \qquad\hbox{for }m=\deg.
\end{equation}
This easy example allows us to observe a few interesting things. First of all, there is no universal power $\alpha$ such that $T_2(\mG;\mV_0)^\alpha \lambda_{0,2}(\mG;\mV_0)$ is an absolute number that is independent of the graph's measure (viz, $m(\mV\setminus \mV_0)$): indeed,  in the case of stars  \eqref{eq:torsional-rigidity-star-m1}, \eqref{eq:lambda0-star-m1}, and~\eqref{eq:lambda0-star-mdeg} show that $\alpha=1$ leads to the correct scaling if $m \equiv \mathbf{1}$, whereas $m = \deg$ requires $\alpha=0$. Also,  \eqref{eq:torsional-rigidity-path-m1} and \eqref{eq:torsional-rigidity-path-mdeg} along with  \eqref{eq:l2fiedl} and \eqref{eq:chung} show that, in the case of paths, $\alpha=\frac{2}{3}$ guarantees at least a correct \textit{asymptotic} scaling, for large $n$, both for $m\equiv \mathbf{1}$ and $m=\deg$.

In particular, as we look for a Kohler-Jobin-type inequality, we cannot expect to find a lower bound on $T_2(\mG;\mV_0)^\alpha \lambda_{0,2}(\mG;\mV_0)$ for a universal exponent $\alpha>0$. However, we may still hope to find a sharp lower bound on $T_2(\mG;\mV_0)^{\alpha(n)} \lambda_{0,2}(\mG;\mV_0)$ for all graphs on a given number of edges. Indeed, we can observe numerically that (both for $m=\mathbf{1}$ and $m=\deg$!) the product $T_2(\mG;\mV_0)^\alpha\lambda_{0,2}(\mG;\mV_0)$ is smaller for the path than for the star on the same number of vertices,
 for any value of $\alpha\le \frac{2}{3}$.

Alternatively, one may think of looking for suitable functions $\Phi, \Psi$ such that $\Phi(T_2(\mG;\mV_0)) \Psi(\lambda_{0,2}(\mG;\mV_0))$ has a universal lower bound: we will explore this idea next.
\end{exa}

A metric graph version of the Kohler-Jobin inequality has been obtained in~\cite{MugPlu23} (for $p=2$ only), where the sharp lower bound is, in this case, given by a path graph: this can be used to prove the following discrete counterpart (for $m=\deg_\mG$).
 
\begin{lemma}\label{theo:kohler-jobin}
Let $\mathsf{G} = (\mV, m=\deg, b=b_{\mathrm{st}}, c \equiv 0)$ be a connected finite graph on $E<\infty$ edges, and let $\mV_0\ne\emptyset$. Then
\begin{align}\label{eq:modified-kohler-jobin-product}
\bigg( T_2(\mG;\mV_0) +\frac{E}{3} \bigg)^\frac{2}{3}\arccos\big(1-\lambda_{0,2}(\mG;\mV_0)\big)^2 \geq   \bigg( \frac{\pi}{\sqrt[3]{6}} \bigg)^2.
\end{align}
Within the class of graphs with $E$ edges, $m = \deg$, edge weight $b=b_{\mathrm{st}}$ and $c \equiv 0$, equality holds if and only if $\mG$ is a path graph. 
\end{lemma}

We regard \eqref{eq:modified-kohler-jobin-product} as a Kohler-Jobin-type inequality on combinatorial graphs.

\begin{proof}
Let $\mathcal{G}$ denote the equilateral metric graph (without parallel edges!) that arises from $\mG$ assigning length $1$ to each edge (in particular, $\mathcal{G}$ has a total length of $\vert \mathcal{G} \vert = E$). It is proven in \cite[Theorem~3.9]{MugPlu23} that there is a one-to-one correspondence between the (combinatorial) torsional rigidity $T_2(\mG;\mV_0)$ and the (metric) torsional rigidity on the same Dirichlet vertex set $\mV_0$, which will be denoted in the following with $T_2(\mathcal{G};\mV_0)$. More precisely, 
\begin{align}\label{eq:one-to-one-correspondence-torsion-altproof}
 T_2(\mathcal{G};\mV_0) = \frac{1}{4}T_2(\mG;\mV_0) + \frac{E}{12}.
\end{align}
Moreover, it is shown in \cite[Theorem~3.1]{Nic85}, that there is also a one-to-one correspondence between the lowest eigenvalue $\lambda_{0,2}(\mG;\mV_0)$ of the normalized Laplacian and the lowest eigenvalue $\lambda_{0,2}(\mathcal{G};\mV_0)$ of the (positive semi-definite) Laplacian $-\Delta^{\mathcal{G};\mV_0}_2$ defined on the metric graph $\mathcal{G}$ subject to Dirichlet conditions at $\mV_0$ and so-called \emph{continuity and Kirchhoff conditions} at $\mV \setminus \mV_0$ (see, for instance, \cite[Section~2]{MugPlu23} for a more precise description of this setup), in the sense that
\begin{align}\label{eq:one-to-one-correspondence-first-eigenvalue-altproof}
 \lambda_{0,2}(\mathcal{G};\mV_0) = \arccos\big(1-\lambda_{0,2}(\mG;\mV_0)\big)^2.
\end{align}
Now according to \cite[Theorem~5.8]{MugPlu23}, the mapping
\begin{equation*}\label{eq:k-j-metric}
\Graph\mapsto 
T_2(\mathcal{G};\mV_0)^{\frac{2}{3}}\lambda_{0,2}(\mathcal{G};\mV_0), \qquad
\end{equation*}
 is minimized (among all metric graphs with a Dirichlet condition on at least one point: thanks to scale invariance it is not necessary to prescribe their total length)  by (all!) metric path graphs having a Dirichlet condition at precisely  one end. In this infinite class we can find precisely one metric path graph that corresponds (in the way described above) to a combinatorial path graph with the same number of edges as $\mathsf{G}$. Therefore, due to \eqref{eq:one-to-one-correspondence-torsion-altproof} and \eqref{eq:one-to-one-correspondence-first-eigenvalue-altproof}, it follows that 
\begin{equation}\label{eq:k-j-metric-comb}
 \mG \mapsto \bigg(T_2(\mG;\mV_0) + \frac{E}{3} \bigg)^\frac{2}{3} \arccos\big(1-\lambda_{0,2}(\mG;\mV_0)\big)^2
\end{equation}
is, within the class of all combinatorial graphs with same number $E$ of edges, minimized by a combinatorial path graph with a Dirichlet condition imposed at precisely one endpoint: indeed starting with any combinatorial graph $\mG$ and considering the corresponding equilateral metric graph $\Graph$, we observe by \cite[Theorem~5.8]{MugPlu23} that:
\begin{align}\label{eq:minimizing-property-kohler-jobin}
\begin{aligned}
&\bigg(T_2(\mP;\mV_0) + \frac{E}{3} \bigg)^\frac{2}{3} \arccos\big(1-\lambda_{0,2}(\mP;\mV_0)\big)^2 = \big(4T_2(\mathcal{P};\mV_0)\big)^\frac{2}{3} \lambda_{0,2}(\mathcal{P};\mV_0) \\& \quad\quad\qquad\leq \big(4T_2(\Graph;\mV_0) \big)^\frac{2}{3} \lambda_{0,2}(\Graph;\mV_0) = \bigg(\frac{1}{4}T_2(\mG;\mV_0) + \frac{E}{12} \bigg)^\frac{2}{3} \arccos\big(1-\lambda_{0,2}(\mG;\mV_0)\big)^2,
\end{aligned}
\end{align}
where $\mathcal{P}$ is a (metric) path graph with one Dirichlet end such that $\vert \mathcal{P} \vert = \vert \Graph \vert = E$ and thus the corresponding combinatorial graph $\mP$ is given by $\mP = (\mV',m'=\deg_\mP, b'=b_{\mathrm{st}}, c \equiv 0)$. This implies \eqref{eq:modified-kohler-jobin-product}, since
\[
\bigg(T_2(\mP;\mV_0) + \frac{E}{3} \bigg)^\frac{2}{3} \arccos\big(1-\lambda_{0,2}(\mP;\mV_0)\big)^2 = \frac{\pi^2}{4E^2}\bigg( \frac{E(2E-1)(2E+1)}{3} + \frac{E}{3} \bigg)^\frac{2}{3} = \bigg(\frac{\pi}{\sqrt[3]{6}} \bigg)^2,
\]
according to \eqref{eq:torsional-rigidity-path-mdeg} and the fact that $\arccos\big(1-\lambda_{0,2}(\mP;\mV_0)\big)^2 = \lambda_{0,2}(\mathcal{P};\mV_0) = \frac{\pi^2}{4E^2}$.
 Vice versa, if equality is attained in \eqref{eq:minimizing-property-kohler-jobin} (and thus, in \eqref{eq:modified-kohler-jobin-product}) for some graph $\mG$, it follows again by \cite[Theorem~3.5]{MugPlu23} that its corresponding metric graph $\Graph$ is equal to a path graph $\mathcal{P}$ having one Dirichlet end. But as the number of edges $E$ is fixed, it follows that $\vert \mathcal{P} \vert =E$ yielding that its corresponding combinatorial path graph is indeed given by $\mP = (\mV',m'=\deg_\mP, b'=b_{\mathrm{st}}, c \equiv 0)$ and eventually that $\mG = \mP$. This completes the proof.
\end{proof}

The main feature of the Kohler-Jobin inequality is that it states that -- among all configurations of same volume -- a geometric object (a ball in~\cite{Koh78}, or path graph in the version in~\cite{MugPlu23}) is the minimizer of the product of two quantities that are, taken individually, minimized and maximized, respectively, by the same geometric object (by the Faber--Krahn and the Saint-Venant inequality, respectively). Also in the case of combinatorial graphs one may recognize the same pattern, using the versions of Faber--Krahn and Saint-Venant that are known to hold for \textit{metric} graphs along with the already mentioned transference principles
\cite[Theorem~3.1]{Nic85} and \cite[Theorem~3.9]{MugPlu23}: indeed, for $m=\deg$, $\lambda_{0,2}$ (resp., $T_2$) is minimized  (resp., maximized) by the path graph, among all graphs on the same number of edges. 

And yet, the Kohler-Jobin-type product $T_2^{\frac{2}{3}}\lambda_{0,2}$ is, altogether, minimized by a path graph.
The following is the main result of this subsection.

\begin{theo}\label{theo:lower-bound-classical-kohler-jobin-product}
Let $\mathsf{G} = (\mV,m =\deg, b=b_{\mathrm{st}}, c \equiv 0)$ be a connected finite graph with standard edge weights on $E< \infty$ edges, and let $\mV_0 \neq \emptyset$. Then
\begin{align}\label{eq:bound-classical-kohler-jobin-without-e}
T_2(\mG;\mV_0)^\frac{2}{3}
\lambda_{0,2}(\mG;\mV_0) \geq 1,
\end{align}
and equality holds if and only if $\mathsf{G}$ is a path graph on a single edge with one Dirichlet condition.
\end{theo}

\begin{proof}
By \autoref{theo:kohler-jobin}, the map
\[
\mG \mapsto \bigg(T_2(\mG;\mV_0) + \frac{E}{3} \bigg)^\frac{2}{3}\arccos\big(1-\lambda_{0,2}(\mG;\mV_0)\big)^2
\]
is minimized
by a path graph within the class of graphs with $E$ edges, $m= \deg$, edge weight $b=b_{\mathrm{st}}$ and $c \equiv 0$, and with a Dirichlet condition at one endpoint.
Now, considering $F: [0,1] \rightarrow \mathbb{R}$ given by
\[
F(x) = \frac{\pi}{2}\sqrt{x} - \arccos(1-x), \qquad x \in [0,1],
\]
it follows that $F$ is a concave function with $\min_{x \in [0,1]} F(x) = F(0) = F(1) = 0$, yielding 
that for $x \in [0,1]$
\begin{equation}\label{eq:taylor-no-F}
\arccos(1-x)^2 
\leq 
\frac{\pi^2}{4}x,
\end{equation}
where the factor $\frac{\pi^2}{4}$ is asymptotically optimal for $\lambda_{0,2}(\mG;\mV_0)\to 1$, since for every $\alpha \geq 0$ such that $\arccos(1-x)^2 \leq \alpha x$ for all $x \in [0,1]$ it follows that $\frac{\pi^2}{4} = \arccos(1-1)^2 \leq \alpha$.

On the other hand, due to \eqref{eq:trivial-estimate-p-torsion-below} and the fact that $m=\deg$ and $b=b_{\mathrm{st}}$ (taking into account \autoref{conv:dirichlet-singleton}, i.e., without loss of generality $\mV_0 = \{\mv_0\}$), we observe
\begin{align}\label{eq:torsional-ridigity-deg-geq-e}
T_2(\mG;\mV_0) \geq  \frac{\Big(\sum\limits_{\mv \in \mV \setminus \mV_0} \deg(\mv)\Big)^2}{\sum\limits_{\mv \sim \mv_0} 1} = \frac{\Big(2E - \deg(\mv_0) \Big)^2}{\deg(\mv_0)} \geq \frac{E^2}{E} = E,
\end{align}
since $\deg(\mv_0) \leq E$. Combining now \eqref{eq:taylor-no-F} and \eqref{eq:torsional-ridigity-deg-geq-e} with \eqref{eq:modified-kohler-jobin-product}, we finally deduce
\begin{align}\label{eq:kohler-jobin-equality-if-and-only-if}
\begin{aligned}
&\bigg(\frac{\pi}{\sqrt[3]{6}}\bigg)^2 T_2(\mG;\mV_0)^\frac{2}{3} \lambda_{0,2}(\mG;\mV_0) = \bigg( \frac{4}{3} T_2(\mG;\mV_0) \bigg)^\frac{2}{3} \bigg(\frac{\pi^2}{4}\lambda_{0,2}(\mG;\mV_0)\bigg) \\&\quad\quad\qquad \geq \bigg(T_2(\mG;\mV_0) + \frac{E}{3} \bigg)^\frac{2}{3}\arccos\big(1-\lambda_{0,2}(\mG;\mV_0)\big)^2 \geq \bigg(\frac{\pi}{\sqrt[3]{6}}\bigg)^2,
\end{aligned}
\end{align}
which yields \eqref{eq:bound-classical-kohler-jobin-without-e}.

Clearly, if $\mG$ is a path graph on a single edge (having a Dirichlet condition at one end) equality is attained in \eqref{eq:bound-classical-kohler-jobin-without-e} due to \eqref{eq:torsional-rigidity-path-mdeg} and \eqref{eq:chung}. Conversely, if equality holds in \eqref{eq:bound-classical-kohler-jobin-without-e}, it follows that equality holds also in \eqref{eq:kohler-jobin-equality-if-and-only-if} and thus in \eqref{eq:modified-kohler-jobin-product} which implies according to \autoref{theo:kohler-jobin} that $\mG$ is given by a path graph on $E$ edges. Let us finally check that $E=1$: indeed, again using \eqref{eq:torsional-rigidity-path-mdeg} and \eqref{eq:chung}, it follows that
\[
1 \leq E^2\bigg(1-\cos\Big(\frac{\pi}{2E}\Big) \bigg) \leq \bigg( \frac{E(2E-1)(2E+1))}{3} \bigg)^\frac{2}{3} \bigg(1-\cos\Big(\frac{\pi}{2E}\Big) \bigg) =1
\]
yielding $E^2 (1-\cos(\frac{\pi}{2E})) = 1$ which is only the case when $E=1$.
\end{proof}

\begin{rem}
Our results crucially use the Kohler-Jobin-type inequality proved in \cite{MugPlu23}, which -- while only formulated for the case $p=2$ -- could perhaps be extended to the case of $p\in (1,\infty)$. However, our proof also deeply relies on the transference principles that relate the discrete and metric torsional rigidity as well as the lowest nontrivial eigenvalues $\lambda_{0,2}(\mG;\mV_0)$ and $\lambda_{0,2}(\Graph;\mV_0)$ of $-\Delta_2^{\Graph;\mV_0}$ induced by the corresponding metric graph $\Graph$: it has been an open question since~\cite{Bel85,Nic85} whether such a transference principle holds for $p$-Laplacians in the non-linear range, too.
Thus, using this approach, an extension of \autoref{theo:kohler-jobin} to the case $p\neq 2$ or even to the case of sharp $\ell^p-\ell^q$-Poincaré constants (cf.\ also \cite{KohJob82}) is -- at best -- an interesting open question.
\end{rem}

\autoref{theo:lower-bound-classical-kohler-jobin-product} now readily implies the following counterpart for the case of $m \equiv \mathbf{1}$.

\begin{cor}\label{cor:kohlerjobin-unnorm}
Let $p = 2$ and $\mG = (\mV, m \equiv \mathbf{1}, b=b_{\mathrm{st}}, c \equiv 0)$ be a  connected finite graph endowed with standard edge weights. Then the classical Kohler-Jobin type product satisfies
\begin{align*}
T_2(\mG;\mV_0)^\frac{2}{3}\lambda_{0,2}(\mG;\mV_0) \geq \frac{\min\limits_{\mv \in \mV \setminus \mV_0}\deg_{\mG}(\mv)}{\Big(\max\limits_{\mv \in \mV \setminus \mV_0}\deg_{\mG}(\mv)\Big)^\frac{4}{3}}
\end{align*}
and equality holds if $\mG$ is a path graph on a single edge with one Dirichlet condition.
\end{cor}
\begin{proof}
Looking at the Rayleigh and Pólya quotients, one observes that
\begin{align*}
T_2(\mG';\mV_0) \leq \max_{\mv \in \mV \setminus \mV_0} \deg_{\mG}(\mv)^2 T_2(\mG;\mV_0) \quad \text{ and } \quad \lambda_{0,2}(\mG';\mV_0) \leq \frac{1}{\min_{\mv \in \mV \setminus \mV_0} \deg_{\mG}(\mv)} \lambda_0(\mG;\mV_0), 
\end{align*}
where $\mG' := (\mV,m' \equiv \deg_{\mG}, b'=b_{\mathrm{st}}, c \equiv 0)$. Then \eqref{eq:bound-classical-kohler-jobin-without-e} implies
\begin{align*}
T_2(\mG;\mV_0)^\frac{2}{3}\lambda_{0,2}(\mG;\mV_0) &\geq \frac{\min\limits_{\mv \in \mV \setminus \mV_0}\deg_{\mG}(\mv)}{\Big(\max\limits_{\mv \in \mV \setminus \mV_0}\deg_{\mG}(\mv)\Big)^\frac{4}{3}} T_2(\mG';\mV_0)^\frac{2}{3}\lambda_{0,2}(\mG';\mV_0) \geq \frac{\min\limits_{\mv \in \mV \setminus \mV_0}\deg_{\mG}(\mv)}{\Big(\max\limits_{\mv \in \mV \setminus \mV_0}\deg_{\mG}(\mv)\Big)^\frac{4}{3}}. \qedhere
\end{align*}
\end{proof}

Let again $m\equiv 1$ and consider the classes of star graphs $\mathsf{S} = (\mV, m \equiv \mathbf{1},b=b_{\mathrm{st}}, c \equiv 0)$ and of path graphs $\mathsf{P} = (\mV, m \equiv \mathbf{1},b=b_{\mathrm{st}}, c \equiv 0)$ on the same number of edges, both with a Dirichlet condition imposed at exactly one vertex (of degree 1).
Comparing the products $T_2(\mP;\mV_0)^\frac{2}{3}\lambda_{0,2}(\mP;\mV_0)$ and $T_2(\mathsf{S};\mV_0)^\frac{2}{3} \lambda_{0,2}(\mathsf{S};\mV_0)$ in dependence of the number $E$ of edges, it can be observed that the Kohler-Jobin-type product is larger for a star graph, as we saw in the case for $m = \deg$ in \autoref{theo:lower-bound-classical-kohler-jobin-product}, cf.\ Figure \ref{fig:plots-classical-kohler-jobin-product}.

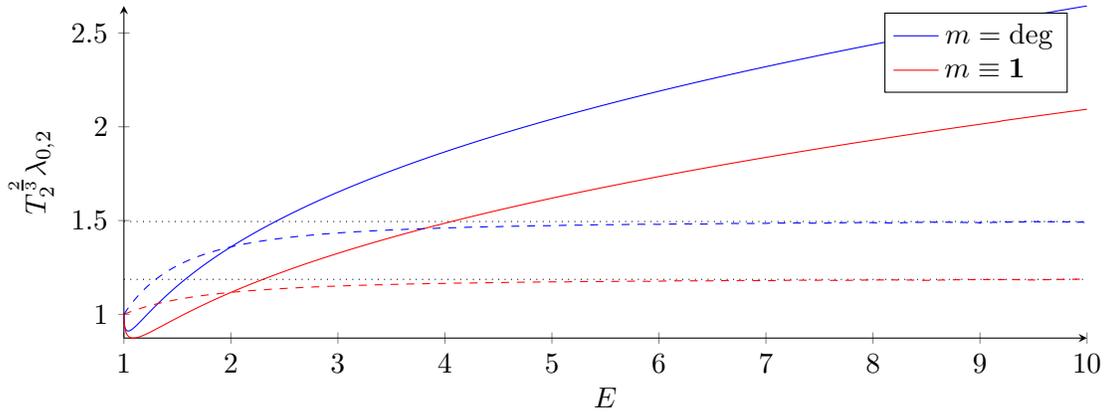
\begin{figure}[ht]
\begin{tikzpicture}
\begin{axis}[
  width=0.85\linewidth,
  	height=6cm,
	enlarge x limits=0.05,
    axis lines = left,
    xlabel = \(E\),
   ylabel = {\(T_2^\frac{2}{3} \lambda_{0,2}\)},
   ]

\addplot [
    domain=1:10,
   restrict y to domain=0:3,
    samples=2000, 
   color=blue,
]
{(4*x^2-3*x)^(2/3)*(1-sqrt((x-1)/x)};
\addlegendentry{\(m = \deg\)}

\addplot [
    domain=1:10, 
   restrict y to domain=0:3,
   samples=2000, 
    color=red,
    ]
   {1/2*(x^2+x-1)^(2/3)*(x+1-sqrt((x-1)*(x+3)))};
   \addlegendentry{\(m \equiv \mathbf{1} \:\:\:\,\,\)}

\addplot [
    domain=1:10,
    restrict y to domain=0:3,
    samples=2000, 
   color=black,
    dotted,
]
{1/2^(1/3)*(pi/(12^(1/3)))^2};

\addplot [
    domain=1:10,
   restrict y to domain=0:3,
    samples=2000, 
    color=black,
   dotted,
]
{(pi/(24^(1/3)))^2};

\addplot [
     domain=1:10,
     restrict y to domain=0:3,
    samples=2000, 
    dashed,
    color=red,
]
{(1/6*x*(x+1)*(2*x+1))^(2/3)*2*(1-cos(deg(pi/(2*x+1)))};

\addplot [
   domain=1:10,
       restrict y to domain=0:3,
   samples=2000,
    dashed, 
   color=blue,
]
{(1/3*x*(2*x+1)*(2*x-1))^(2/3)*(1-cos(deg(pi/(2*x)))};
\end{axis}
\end{tikzpicture}
\caption{Plots of the Kohler-Jobin-type product $T_2^{\frac{2}{3}}\lambda_{0,2}$ for star (continuous line) and path graphs (dashed line) o the same number of edges $E$ in the cases for $m=\deg$ (blue) and $m \equiv \mathbf{1}$ (red). All plots have a common crossing point at $E=1$ (in this case all graphs coincide and $\deg \equiv \mathbf{1}$) and a separate crossing point at $E=2$ (since a $2$-star coincides with a path having $2+1$ vertices, \emph{but} $\deg \not\equiv \mathbf{1}$). For path graphs $T_2^{\frac{2}{3}}\lambda_{0,2}$ converges to $\frac{1}{\sqrt[3]{2}}\big( \frac{\pi}{\sqrt[3]{12}} \big)^2 \approx 1.495$ (if $m = \deg$) and $\big( \frac{\pi}{\sqrt[3]{24}} \big)^2 \approx 1.186$ (if $m \equiv \mathbf{1}$) from below, respectively, whereas for star graphs $T_2^{\frac{2}{3}}\lambda_{0,2}$ tends to $+ \infty$.}\label{fig:plots-classical-kohler-jobin-product}
\end{figure}

\bibliographystyle{plain}
\bibliography{literatur.bib}
\end{document}